\documentclass[12pt,reqno]{amsart}

\usepackage[colorlinks=true,
            linkcolor=blue,
            citecolor=blue,
            urlcolor=blue,
            pdfstartview=FitH,
            bookmarks=true,
            bookmarksopen=true]{hyperref}

\usepackage{amssymb, amsmath, amsfonts, latexsym}
\usepackage{enumerate}

\usepackage{color}
\usepackage{diagbox}
\newtheorem{thm}{Theorem}[]
\newtheorem{lem}{Lemma}[section]

\newtheorem{rmk}{Remark}[section]
\theoremstyle{definition}

\usepackage{graphicx}
\usepackage{subcaption}
\numberwithin{equation}{section} \theoremstyle{remark}

\title[Product of Ginibre]{\bf From Gaussian to Gumbel: 
extreme eigenvalues of complex Ginibre products with exact rates}

\author{Y\MakeLowercase{utao} M\MakeLowercase{a} and X\MakeLowercase{ujia} M\MakeLowercase{eng}}
\address{School of Mathematical Sciences $\&$ Laboratory  of Mathematics and Complex Systems of Ministry of Education, Beijing Normal University, 100875 Beijing, China.} 
\email{mayt@bnu.edu.cn, 202321130122@mail.bnu.edu.cn} 

\begin{document}
\maketitle
 
\begin{abstract} 

We consider the product of \(k_{n}\) independent \(n\times n\) complex Ginibre matrices and denote its eigenvalues by \(Z_{1},\ldots ,Z_{n}\). Let \(\alpha = \lim_{n\to\infty} n / k_{n}\). Using the determinantal point process method, we reduce the study of extremal eigenvalues to the evaluation of determinants of certain \(n\times n\) matrices. In the modulus case, rotational invariance makes the relevant matrix diagonal, which yields a product representation in terms of Gamma tail probabilities. In the real-part case, the matrix is no longer diagonal; we handle this by a polar-coordinate reduction that introduces an independent uniform angle and leads to explicit formulas involving Gamma variables and trigonometric integrals.

After appropriate rescaling, the spectral radius \(\max_{1\leq j\leq n}|Z_{j}|\) converges weakly to a nontrivial distribution \(\Phi_{\alpha}\) when \(\alpha \in (0, +\infty)\), to the Gumbel distribution when \(\alpha = +\infty\), and to the standard normal distribution when \(\alpha = 0\). The family \(\{\Phi_{\alpha}\}_{\alpha >0}\) extends continuously to the boundary regimes: \(\Phi_{\alpha}\) converges weakly to the standard normal law as \(\alpha \to 0^{+}\) and to the Gumbel law as \(\alpha \to +\infty\). Thus the three limiting regimes are connected by the single parameter \(\alpha\), yielding a continuous transition from Gaussian to Gumbel distribution. For the spectral radius, we obtain the exact rates of convergence both in the fixed-\(\alpha\) regime and at the boundaries \(\alpha = 0\) and \(\alpha = +\infty\). For the rightmost eigenvalue \(\max_{1\leq j\leq n}\Re Z_{j}\), we establish the convergence rates in the boundary regimes, while for \(\alpha \in (0, +\infty)\) we show that the limiting distribution, though not available in closed form, still interpolates continuously between the normal and Gumbel laws.
\end{abstract} 

Keywords: Product of Ginibre ensemble; spectral radius; rightmost eigenvalue; Berry-Esseen bound; continuous transition.

AMS Classification Subjects 2020: 60B20, 15B52, 60G70, 60F05.

\tableofcontents

\section{Introduction}
Products of random matrices form a central class of models in modern random matrix theory and arise naturally in wireless communications, disordered systems, quantum transport, dynamical systems, and non-Hermitian statistical mechanics(\cite{Ipsen2015}). Among them, products of complex Ginibre matrices are especially tractable: for a fixed number of factors, the joint density of the eigenvalues is explicit, and the underlying determinantal structure makes it possible to study fine spectral statistics in considerable detail; see, for example, Akemann and Burda \cite{AkBur} and Adhikari {\it et al.} \cite{Adhi}. For various results concerning
products of random matrices, the readers are referred to \cite{AkBur, AkIp15, BaiY98, Bordenave, BouLa, BurdaJW10, BJLN10, Burda13, Crisanti12, Forrester15, GT, Ipsen, IpsenKie14, JQ17, JQ19, Kopel2020, LW15, LWW23,  LW24, QiXie, RS11, RRSV15, Tik, Zeng17}.

In this paper we study extreme eigenvalues of products of complex Ginibre matrices in a regime where the number of factors is allowed to vary with the dimension. Let
${\bf A}_1, \dots, {\bf A}_{k_n}$ be independent $n\times n$ complex Ginibre matrices, and let $Z_1,\dots,Z_n$ be the eigenvalues of the product ${\bf A}_1{\bf A}_2\cdots{\bf A}_{k_n}$. We focus on the two natural edge observables
\[
\max_{1\le j\le n}|Z_j|
\qquad \text{and} \qquad 
\max_{1\le j\le n}\Re Z_j,
\]
namely the spectral radius and the rightmost eigenvalue. The relevant asymptotic parameter is
\[
\alpha:=\lim_{n\to\infty}\frac{n}{k_n}\in[0,+\infty].
\]
The three cases $\alpha=0$, $\alpha\in(0,+\infty)$, and $\alpha=+\infty$ correspond, respectively, to dense-factor regime, the  proportional regime, and the sparse-factor regime.

For Ginibre products, the spectral radius has a rich asymptotic theory. We first recall the results of Jiang and Qi in \cite{JQ17}. 
\begin{thm}[Jiang-Qi] 
Let $k_n$ and $\alpha_n$ be defined as above. The following holds.

(a). If \(\alpha=+\infty\), then 
\[\sqrt{\alpha_n \log \alpha_n}\big( n^{-k_n/2} \max_{1 \leq j \leq n} |Z_j| - 1 \big) - (\log \alpha_n - \log(\sqrt{2\pi} \log \alpha_n))\]
converges weakly to the Gumbel distribution.

(b). If \(\alpha \in (0, \infty)\), for any $x>0,$ we have 
\[\lim_{n\to\infty}\mathbb{P}(\max_{1 \leq j \leq n} |Z_j|\le x n^{k_n/2})=\prod_{j=0}^{+\infty}\Phi(\frac{j}{\sqrt{\alpha}}+\frac{1}{2\sqrt{\alpha}}+2\sqrt{\alpha}\log x).\] 

(c). If \(\alpha=0\), then
\[\lim_{n\to\infty}\mathbb{P}(\max_{1 \leq j \leq n} \log |Z_j|\le  \frac12 k_n \psi(n)+\frac{x}{\sqrt{2\alpha_n}})=\Phi(x).\]
Here, $\Phi$ is the distribution function of standard normal and $\psi(n)$ is the digamma function.
	\end{thm}

Thus, depending on the growth rate of \(k_n\) relative to \(n\), three completely different limit laws appear. Wang \cite{Wang18} further studied order statistics of the moduli in a broader class of polynomial ensembles, while Qi and Xie \cite{QiXie} obtained a more unified formulation on the logarithmic scale for products of rectangular complex Ginibre matrices. Ma and Qi \cite{MaQi2024} obtained similar results for products of Ginibre matrices and their inverses.

These works reveal a clear trichotomy for radial extremes, but they also leave open a basic structural issue. The normalizations used in the three regimes are different, and in the dense regime the natural observable is the \emph{logarithm} of the spectral radius rather than the spectral radius itself. As a consequence, the earlier results do not by themselves exhibit a genuine one-parameter interpolation between the Gaussian, intermediate, and Gumbel regimes.

The situation is even subtler for the rightmost eigenvalue. Unlike the spectral radius, the largest real part is not a radial observable, and rotational symmetry no longer diagonalizes the relevant operator. Even for a single Ginibre matrix, the analysis of the rightmost eigenvalue typically relies on determinantal kernels and Fredholm determinants, rather than on a reduction to independent radial variables. In related settings, Bender \cite{Benderellip} obtained an edge transition for elliptic Ginibre ensembles, and for single Ginibre matrices effective error bounds and sharp convergence rates for extremal statistics have been established in recent work; see, for instance, \cite{Cipolloni22Directional, HuMa25}. More recently, universality of extremal eigenvalue fluctuations has also been proved for broad classes of complex i.i.d.\ non-Hermitian matrices; see \cite{CipoErXu}. However, these results do not provide a direct counterpart for products of complex Ginibre matrices when $\alpha\in[0, +\infty]$. To the best of our knowledge, for the rightmost eigenvalue of Ginibre products there has been no closed-form limiting law, no quantitative analysis, and no explicit Fredholm-determinant asymptotics available in the literature.

The present paper addresses both the structural and the quantitative aspects of this problem. Our first goal is to show that the three spectral-radius regimes can in fact be embedded into a single continuous picture. We introduce an $\alpha$-dependent rescaling under which the spectral radius converges, for every fixed $\alpha\in(0,+\infty)$, to a family of distribution functions $\{\Phi_\alpha\}_{\alpha>0}$, and this family extends continuously to the boundary regimes:
\[
\Phi_\alpha \Longrightarrow \Phi \quad \text{as }\alpha\to0^+,
\qquad
\Phi_\alpha \Longrightarrow \Lambda \quad \text{as }\alpha\to+\infty,
\]
where $\Phi$ is the standard normal law and $\Lambda(x)=e^{-e^{-x}}$ is the Gumbel law. Thus the single parameter $\alpha$ governs a continuous transition from Gaussian to Gumbel behavior.

Our second goal is quantitative. For the spectral radius, we obtain exact asymptotics for the Berry Esseen bound  to the limiting law in all three regimes. In particular, we derive exact convergence rates when $\alpha=0$ and $\alpha=+\infty$, as well as exact fixed-$\alpha$ asymptotics in the proportional regime. We also determine the boundary asymptotics of the interpolating family $\Phi_\alpha$ itself as $\alpha\to0^+$ and $\alpha\to+\infty$. In this sense, the paper gives both a unified limiting picture and a precise quantitative theory for the spectral radius of Ginibre products.

We also prove an analogous continuous transition for the rightmost eigenvalue. After a suitable rescaling of $\max_{1\le j\le n} \Re Z_j$, we recover the Gaussian limit at $\alpha=0$ and the Gumbel limit at $\alpha=+\infty$. For every fixed $\alpha\in(0,+\infty)$, the limiting distribution is given by the Fredholm determinant of an explicit trace-class operator. This provides a continuous interpolation from Gaussian to Gumbel for the largest real part as well. The strength of the quantitative result, however, is necessarily different from that for the spectral radius: in the rightmost problem we obtain exact convergence rates in the two boundary regimes and sharp boundary asymptotics of the limiting family, but we do not obtain a full fixed-$\alpha$ convergence-rate formula in the interior regime. The obstruction is precisely that the limiting object is no longer an explicit infinite product, but a genuinely non-diagonal Fredholm determinant.

Our approach is based on the determinantal structure of the eigenvalue point process, but the two observables lead to markedly different algebraic problems. For the spectral radius, rotational invariance makes the relevant matrix diagonal. This reduces the gap probability to a product representation involving tail probabilities of Gamma variables, and this explicit structure is what allows us to identify the interpolating family and extract exact convergence rates. For the rightmost eigenvalue, by contrast, the corresponding matrix is not diagonal. A polar-coordinate reduction introduces an additional independent angular variable and leads to explicit integral formulas involving Gamma variables together with trigonometric terms. This distinction is the main technical feature of the paper: the modulus problem is essentially diagonal, while the real-part problem is inherently non-diagonal.

A further methodological point is that the standard approximation
\[
\det({\rm I}-\mathbb{K}_n|_E)\approx \exp\bigl(-{\rm Tr}(\mathbb{K}_n|_E)\bigr)
\]
is sufficient only in the sparse-factor regime, where the relevant operator norm is small enough for first-order trace asymptotics to dominate. In the regimes $\alpha=0$ and $\alpha\in(0,\infty)$, this mechanism is no longer adequate. Our analysis instead relies on refined asymptotics of the matrix entries arising in the determinantal reduction, which are then assembled to identify the limit and, in the spectral-radius case, to derive exact quantitative errors. 

\subsection{Statement of the results}

We now state our main results separately for the two edge observables
\[
\max_{1\le j\le n}|Z_j|,
\qquad
\max_{1\le j\le n}\Re Z_j,
\]
under the standing assumption
\[
\alpha_n:=\frac{n}{k_n}\longrightarrow \alpha\in[0,+\infty].
\]
Throughout, \(\Phi\) denotes the density and distribution function of the
standard normal law, and
\[
\Lambda(x):=e^{-e^{-x}},\qquad x\in\mathbb R,
\]
denotes the Gumbel distribution function.

\paragraph{\bf Spectral radius.}
Define the rescaled spectral radius by
\[
X_n:=b_n^{-1}\Bigl[\sqrt{\alpha_n}\bigl(2\log \max_{1\le j\le n}|Z_j|-k_n\psi(n)\bigr)-a_n\Bigr],
\]
where
\[
a_n:=\sqrt{\log(\alpha_n+1)}
-\frac{\log\!\bigl(\sqrt{2\pi}\log(\alpha_n+e^{1/\sqrt{2\pi}})\bigr)}
{\sqrt{\log(\alpha_n+e)}},
\qquad
b_n:=\frac{1}{\sqrt{\log(\alpha_n+e)}},
\]
and \(\psi(x)=\Gamma'(x)/\Gamma(x)\) is the digamma function.

For \(\alpha\in(0,+\infty)\), define
\[
a:=\sqrt{\log(\alpha+1)}
-\frac{\log\!\bigl(\sqrt{2\pi}\log(\alpha+e^{1/\sqrt{2\pi}})\bigr)}
{\sqrt{\log(\alpha+e)}},
\qquad
b:=\frac{1}{\sqrt{\log(\alpha+e)}},
\]
\[
v_\alpha(j,x):=\frac{j-1}{\sqrt{\alpha}}+a+bx,
\qquad j\ge 1,\ x\in\mathbb R,
\]
and
\[
\Phi_\alpha(x):=\prod_{j=1}^{\infty}\Phi\bigl(v_\alpha(j,x)\bigr).
\]

To state the exact fixed-\(\alpha\) asymptotics, we further define
\[
q_1(j,x)
:=\frac{1}{12\sqrt{\alpha}}
\Bigl[2\alpha\bigl(v^2_\alpha(j,x)-1\bigr)
-3\sqrt{\alpha}(2j-1)v_\alpha(j,x)+6j(j-1)\Bigr],
\]
and
\[
q_2(j,x):=c_1-c_2x-\frac{j-1}{2\alpha^{3/2}},
\]
where, writing
\[
w(t):=2t\log t,
\]
we set
\[
c_1:=
\frac{\sqrt{\log(\alpha+1)}}{w(\alpha+1)}
+\frac{2}{w(\alpha+e^{1/\sqrt{2\pi}})\sqrt{\log(\alpha+e)}}
-\frac{\log\!\bigl(\sqrt{2\pi}\log(\alpha+e^{1/\sqrt{2\pi}})\bigr)}
{w(\alpha+e)\sqrt{\log(\alpha+e)}},
\]
and
\[
c_2:=\frac{1}{2(\alpha+e)(\log(\alpha+e))^{3/2}}.
\]

\begin{thm}[Spectral radius]\label{main}
Let \(X_n\), \(\alpha\), \(a_n\), \(v_\alpha\), \(q_1(j,\cdot)\), \(q_2(j,\cdot)\), and
\(\Phi_\alpha\) be defined as above. Then the following Berry Esseen bound 
hold for the spectral radius.

\begin{enumerate}
\item If \(\alpha=+\infty\), then
\[
\sup_{x\in\mathbb R}\big|\mathbb P(X_n\le x)-e^{-e^{-x}}\big|
=\frac{(\log\log \alpha_n)^2}{2e\log \alpha_n}(1+o(1))
\]
for all sufficiently large \(n\).

\item If \(\alpha=0\), then
\[
\sup_{x\in\mathbb R}\big|\mathbb P(X_n\le x)-\Phi(x)\big|
=(1+o(1))
\sup_{x\in\mathbb R}\phi(x)\big|\sqrt{\alpha_n}-\frac{x}{4n}\big|
\]
as \(n\to+\infty\).

\item If \(\alpha\in(0,+\infty)\), then
\[
\sup_{x\in\mathbb R}\big|\mathbb P(X_n\le x)-\Phi_\alpha(x)\big|
\]
\[
=(1+o(1))
\sup_{x\in\mathbb R}
\Phi_\alpha(x)
\big|
\sum_{j=1}^{\infty}
\frac{\phi(v_\alpha(j,x))}{\Phi(v_\alpha(j,x))}
\bigl[n^{-1}q_1(j,x)+(\alpha_n-\alpha)q_2(j,x)\bigr]
\big|
\]
whenever \(n\) is sufficiently large.

\item As the tuning parameter \(\alpha\) varies from \(0\) to \(+\infty\), one has
\[
\lim_{\alpha\to0^+}
\frac{1}{\sqrt{\alpha}}
\sup_{x\in\mathbb R}\left|\Phi_\alpha(x)-\Phi(x)\right|
=\frac{1}{\sqrt{2\pi}},
\]
and
\[
\lim_{\alpha\to+\infty}
\frac{\log\alpha}{(\log\log\alpha)^2}
\sup_{x\in\mathbb R}\big|\Phi_\alpha(x)-e^{-e^{-x}}\big|
=\frac{1}{2e}.
\]
\end{enumerate}
\end{thm}

\paragraph{\bf Rightmost eigenvalue.}
For the largest real part, it is more convenient to formulate the scaling through
threshold events. Define \(\widetilde X_n\) by
\[
\widetilde{X}_n:=\widetilde{b}_n^{-1}\Bigl[\sqrt{\alpha_n}\bigl(2\log \max_{1\le j\le n}\Re Z_j-k_n\psi(n)\bigr)-\widetilde{a}_n\Bigr]
\]
where
\[
\widetilde b_n:=\frac{\sqrt{2}}{\sqrt{\log(\alpha_n+e^2)}},
\]
and
\[
\widetilde a_n
:=
\sqrt{\frac{\log(\alpha_n+1)}{2}}
-
\frac{
\sqrt{2}\bigl(
\log(2^{-3/4}\pi)
+\frac54\log\log(\alpha_n+e^{2^{3/5}\pi^{-4/5}})
\bigr)
}
{\sqrt{\log(\alpha_n+e^2)}}.
\]
Let \(\widetilde a\) and \(\widetilde b\) denote the corresponding limits of
\(\widetilde a_n\) and \(\widetilde b_n\), respectively.

For \(\alpha\in(0,+\infty)\), define the infinite-dimensional matrix
\[
\widetilde M(x,\alpha)
=
\bigl(\widetilde M_{j,k}(x,\alpha)\bigr)_{j,k\ge 1}
\]
by
\[
\widetilde M_{j,k}(x,\alpha)
=
\frac{2}{\pi}
\exp\!\big(-\frac{(j-k)^2}{4\alpha}\big)
\int_0^{\pi/2}
\cos((j-k)\theta)\,
\Psi\!\big(
\widetilde v_\alpha\!\big(\frac{j+k}{2},x\big)
-\sqrt{\alpha}\log\cos^2\theta
\big)
\,d\theta
\]
whenever \(j-k\) is even, and
\[
\widetilde M_{j,k}(x,\alpha)=0
\qquad\text{whenever } j-k \text{ is odd},
\]
where
\[
\Psi(y):=1-\Phi(y),
\qquad
\widetilde v_\alpha(j,x):=\frac{j-1}{\sqrt{\alpha}}+\widetilde a+\widetilde b x.
\]

\begin{thm}[Rightmost eigenvalue]\label{thmrealpart}
Let \(\widetilde X_n\), \(\alpha\), and \(\alpha_n\) be defined as above.

\begin{enumerate}
\item If \(\alpha=+\infty\), then
\[
\sup_{x\in\mathbb R}\big|\mathbb P(\widetilde X_n\le x)-e^{-e^{-x}}\big|
=
\frac{25(\log\log \alpha_n)^2}{16e\log \alpha_n}(1+o(1))
\]
for all sufficiently large \(n\).

\item If \(\alpha=0\), then
\[
\sup_{x\in\mathbb R}\big|\mathbb P(\widetilde X_n\le x)-\Phi(x)\big|
=
(1+o(1))
\sup_{x\in\mathbb R}
\phi(x)\big|
\frac{\sqrt{2}+4\ln 2}{2}\sqrt{\alpha_n}
-\frac{x}{4n}
\big|
\]
as \(n\to+\infty\).

\item If \(\alpha\in(0,+\infty)\), then for each fixed \(x\in\mathbb R\), the
infinite-dimensional matrix \(\widetilde M(x,\alpha)\) defines a trace-class operator on
\(\ell^2(\mathbb N)\). Hence
\[
\widetilde\Phi_\alpha(x):=\det({\rm I}-\widetilde M(x,\alpha))
\]
is well defined. Moreover, \(\widetilde\Phi_\alpha\) is a distribution function and
\[
\lim_{n\to+\infty}\mathbb P(\widetilde X_n\le x)=\widetilde\Phi_\alpha(x),
\qquad x\in\mathbb R.
\]
The family \(\{\widetilde\Phi_\alpha\}_{\alpha>0}\) interpolates continuously between the
Gaussian and Gumbel laws, and its boundary asymptotics are given by
\[
\lim_{\alpha\to0^+}
\frac{1}{\sqrt{\alpha}}
\sup_{x\in\mathbb R}\left|\widetilde\Phi_\alpha(x)-\Phi(x)\right|
=
\frac{\sqrt{2}+4\ln 2}{2\sqrt{2\pi}},
\]
and
\[
\lim_{\alpha\to+\infty}
\frac{\log\alpha}{(\log\log\alpha)^2}
\sup_{x\in\mathbb R}\left|\widetilde\Phi_\alpha(x)-e^{-e^{-x}}\right|
=
\frac{25}{16e}.
\]
\end{enumerate}
\end{thm}

The difference between the two theorems is part of the main message of the paper.
For the spectral radius, the determinantal reduction becomes diagonal and leads to
an explicit infinite-product limit together with full fixed-\(\alpha\) quantitative
asymptotics. For the rightmost eigenvalue, the limiting object in the proportional
regime is a genuinely non-diagonal Fredholm determinant, and this is why the
interior fixed-\(\alpha\) rate is not available in closed form.

The key estimates employed in the proof of Theorem \ref{main} can be readily adapted to establish the convergence rate in the $W_1$-Wasserstein distance.
\begin{rmk}\label{w1cor}
		Under the assumptions and notation of Theorem \ref{main}, the corresponding \( W_1 \)-Wasserstein distances satisfy the following asymptotics.

\begin{enumerate}
\item[(1)] Case \(\alpha = 0\):
\[
W_{1}\bigl(\mathcal{L}(X_n),\Phi\bigr)=\Bigl[\sqrt{\alpha_n}\,\bigl(2\Phi(4n\sqrt{\alpha_n})-1\bigr)+\frac{1}{2n}\,\phi(4n\sqrt{\alpha_n})\Bigr](1+o(1)).
\]

\item[(2)] Case \(\alpha \in (0,+\infty)\):
\[
\begin{aligned}
W_{1}\bigl(\mathcal{L}(X_n),\Phi_{\alpha}\bigr)
= \int_{-\infty}^{+\infty} \Phi_\alpha(x) \,
\Bigl| \sum_{j=1}^{\infty} \frac{\phi(v_\alpha(j, x))}{\Phi(v_\alpha(j, x))}
\bigl( n^{-1} q_1(j, x) + (\alpha_n - \alpha) q_2(j, x) \bigr) \Bigr| \, dx.
\end{aligned}
\]

\item[(3)] Case \(\alpha = +\infty\):
\[
W_{1}\bigl(\mathcal{L}(X_n),\Lambda\bigr)=\frac{(\log \log \alpha_n)^{2}}{2\log \alpha_n}\,(1+o(1)).\]
\end{enumerate}

Here, \(\mathcal{L}(X_n)\) denotes the distribution of \(X_n\) and \(\Lambda\) is the Gumbel distribution. Similar results hold for \(\widetilde{X}_n\) when \(\alpha = 0\) or \(\alpha = +\infty\).
\end{rmk} 

\begin{rmk}
It is worth noting that when $k_n = 1,$ the product reduces to a single complex Ginibre ensemble. In this case, the Berry-Esseen bound for the rescaled spectral radius relative to the Gumbel distribution is of order $O\big( \frac{(\log \log n)^2}{\log n} \big).$  This differs from the order $O\big( \frac{\log \log n}{\log n} \big)$ obtained in \cite{HuMa25, MaMeng2025}. As elucidated in \cite{HuMa25, MaTian}, this discrepancy is attributed to the use of different rescaling coefficients.
\end{rmk} 

\begin{rmk}\label{upperbound}
    The supremum expressions appearing in Theorem~\ref{main} for \( \alpha \in (0, +\infty) \) do not admit a closed form. We therefore provide simple, though not sharp, upper bounds as follows:
    \[
    \sup_{x \in \mathbb{R}} \Phi_\alpha(x) \big| \sum_{j=1}^{\infty} \frac{q_1(j, x) \phi(v_\alpha(j, x))}{\Phi(v_\alpha(j, x))} \big| \leq \frac{4}{3} \big( \alpha + \sqrt{\alpha} + 1 \big),
    \]
    and
    \[
    \sup_{x \in \mathbb{R}} \Phi_\alpha(x) \big| \sum_{j=1}^{\infty} \frac{q_2(j, x) \phi(v_\alpha(j, x))}{\Phi(v_\alpha(j, x))} \big| \leq \frac{2}{e \ln 2} \big( c_1 + \frac{c_2(\alpha - 1)}{b} + \frac{1}{\alpha} \big) (1 + \sqrt{\alpha}).
    \]
\end{rmk}

Given the complexity of the proofs of Theorems \ref{main} and \ref{thmrealpart}, we outline the main tools and steps to enhance accessibility.

\subsection{Sketch of the proofs of Theorems 1 and 2}

The proofs are based on a common determinantal reduction followed by a
regime-dependent asymptotic analysis of the resulting matrix entries. As in
\cite{MaMeng2025}, we first localize \(x\) to a central window on which sharp asymptotics
hold; the two tail regions contribute only lower-order terms and are treated
separately. We record only the main reductions here.

\subsubsection{The determinantal reduction}

We use the determinantal point process method of \cite{HKPV2009} to relate the probabilities $\mathbb{P}(X_n\le x)$ and $\mathbb{P}(\widetilde{X}_n\le x)$ to determinants of two $n\times n$ matrices.

Indeed, by Theorem 1 in Adhikari {\it et al.} \cite{Adhi}, $(Z_1, \cdots, Z_n)$ forms a determinantal  point process with correlation kernel 
$$\mathbb{K}_{n}(z, w)=\frac{\sqrt{\varphi(z)\varphi(w)}}{\pi^{k_n} }\sum_{j=0}^{n-1}\frac{(z\bar{w})^j}{(j!)^{k_n}},$$ where
$\varphi$ satisfies $\varphi(z)=\varphi(|z|)$ for any $z\in\mathbb{C}.$  

Define
$$A(x):=\big\{z\in\mathbb{C}|\log | z|\ge \frac{1}{2}k_n\psi(n)+\frac{a_n+b_n x}{2\sqrt{\alpha_n}}\big\}$$ 
and  
$$\widetilde{A}(x):=\big\{z|\log\Re z\ge \frac{1}{2}k_n\psi(n)+\frac{\widetilde{a}_n+\widetilde{b}_n x}{2\sqrt{\alpha_n}}\big\}.$$ Then the basic property of determinantal point processes yields 
 $$\mathbb{P}(X_n\le x)={\rm det}({\rm I}-{\mathbb{K}_n}|_{A(x)}^{} ) \quad \text{and} \quad \mathbb{P}(\widetilde{X}_n\le x)={\rm det}({\rm I}-{\mathbb{K}_n}|_{\widetilde{A}(x)} ).$$ 

  Let $$\phi_j(z)= \frac{\sqrt{\varphi(z)}}{\pi^{k_{n}/2}} \cdot \frac{z^{j}}{(j!)^{k_{n}/2}}, \quad 0\le j\le n-1,$$ which form a standard orthogonal basis for $\mathbb{K}_{n}.$
 Define an $n\times n$ matrix $M^{(n)}(x)=(M_{j, k}^{(n)}(x))_{1\le j, k\le n}$ by      	\[
     	M_{j, k}^{(n)}(x) = \int_{A(x)} \phi_{n-j}(z)\;\overline{\phi_{n-k}(z)}\,  \; d^2 z, \qquad  1\le k, j \le n,
     	\]
  and analogously 
    $$ \widetilde{M}_{j, k}^{(n)}(x) = \int_{\widetilde{A}(x)} \phi_{n-j}(z)\;\overline{\phi_{n-k}(z)}\,  \; d^2 z, \qquad  1\le k, j \le n,$$ 
     	where $\bar{w}$ denotes the conjugate of $w$ for $w\in\mathbb{C}$.  

Since \(\mathbb{K}_n\) is of finite rank (see \cite{HKPV2009}), the Fredholm determinant reduces to a finite determinant:  
$${\rm det}({\rm I}-{\mathbb{K}_n}|_{A(x)}^{} )={\rm det}({\rm I}_n-M^{(n)}(x)), \quad  {\rm det}({\rm I}-{\mathbb{K}_n}|_{\widetilde{A}(x)}^{} )={\rm det}({\rm I}_n-\widetilde{M}^{(n)}(x)).$$ Consequently  
\begin{equation}\label{Pdet0} \aligned \mathbb{P}(X_n\le x)&={\rm det}({\rm I}_n-M^{(n)}(x)), \quad 
\mathbb{P}(\widetilde{X}_n\le x)&={\rm det}({\rm I}_n-\widetilde{M}^{(n)}(x)).
\endaligned \end{equation}  
Thus the problem reduces to analyzing the matrices $M^{(n)}(x)$ and $\widetilde{M}^{(n)}(x).$ 

The identity 
\begin{equation}\label{CJQnewunderstanding}\int_{\mathbb{C}} f(z)\varphi(z)d^2 z=\int_{z_1\cdots z_{k_n}=z} f(z_1\cdots z_{k_n})\exp(-\sum_{m=1}^{k_n} |z_m|^2)\prod_{m=1}^{k_n} d^2 z_m \end{equation}
serves as the basic input in the reduction and is taken from \cite{CJQ25}. Since $\varphi$ is rotation-invariant and the  domain $A(x)$ depends only on the radial component, the matrix $M^{(n)}(x)$ is diagonal. Hence \[
\mathbb{P}(X_n \le x) = \prod_{j=1}^n \bigl(1 - M^{(n)}_{j,j}(x)\bigr).
\]
To express the diagonal entries, define $Y_j = \prod_{r=1}^{k_n} S_{j,r}$, where $(S_{j,r})_{1 \le r \le k_n}$ are i.i.d. random variables with common density
\[
\frac{1}{(j-1)!} y^{j-1} e^{-y} \mathbf{1}_{\{y > 0\}}.
\]
Using the identity \eqref{CJQnewunderstanding}, one can show that
\begin{equation}\label{Mjjnx}
M_{j, j}^{(n)}(x) = \mathbb{P}\bigl(\log Y_{n+1-j} \ge k_n \psi(n) + \frac{a_n + b_n x}{\sqrt{\alpha_n}}\bigr).
\end{equation}

For the matrix $\widetilde{M}^{(n)}(x)$, the analysis is more delicate because the domain $\widetilde{A}(x)$ is no longer radial. 
Its entries admit the following representation. For the diagonal terms,
\[
\widetilde{M}_{j,j}^{(n)}(x)=\mathbb{P}\bigl(\log Y_{n+1-j}+\log\cos^2\Theta \ge k_n\psi(n)+\frac{\widetilde{a}_n+\widetilde{b}_n x}{\sqrt{\alpha_n}}\bigr),
\] 
where $\Theta$ is uniformly distributed on $[0, \frac{\pi}{2})$ and is independent of $\{Y_j\}_{1\le j\le n}.$ For the off-diagonal entries, $\widetilde{M}_{j, k}^{(n)}(x)=0$ when $j-k$ is odd, while for even $j\neq k,$
$$
\begin{aligned}
\widetilde{M}_{j, k}^{(n)}(x) &= \frac{2(\big( n-\frac{j+k}{2} \big)!)^{k_n}}{\pi \bigl( (n-j)! (n-k)! \bigr)^{k_n/2}} \\
&\quad \times \int_{0}^{\pi/2} \cos\bigl( (j-k)\theta \bigr) \, \mathbb{P}\bigl( \log Y_{n+1-\frac{j+k}{2}} + \log\cos^2\theta \ge k_n\psi(n)+\frac{\widetilde{a}_n+\widetilde{b}_n x}{\sqrt{\alpha_n}} \bigr) \, d\theta
.\end{aligned} $$

\subsubsection{Proof strategy for Theorem 1}

Since \(M^{(n)}(x)\) is diagonal, the proof of Theorem \ref{main} reduces to the asymptotic
analysis of the one-dimensional tails \(M_{j,j}^{(n)}(x)\).

\begin{enumerate}
\item \textbf{Case \(\alpha = +\infty\).} 
For this case, the entries \(M_{j,j}^{(n)}(x)\) are uniformly small, so the
determinant is asymptotically governed by the trace:
\[
\det({\rm I}_n-M^{(n)}(x))
=
\exp\bigl(-{\rm Tr}(M^{(n)}(x))\bigr)(1+o(1)).
\]
This yields the Gumbel limit together with the exact boundary rate.

\item \textbf{Case \(\alpha = 0\).} 
The first diagonal term dominates and one obtains
\[
\mathbb{ P}(X_n\le x)
=
\bigl(1-M_{1,1}^{(n)}(x)\bigr)(1+o(1)).
\]
The corresponding asymptotic expansion of \(M_{1,1}^{(n)}(x)\) gives the Gaussian
limit and the exact rate.

\item \textbf{Case \(\alpha \in (0, +\infty)\).} 
An Edgeworth expansion for \(1-M_{j,j}^{(n)}(x)\), uniform
for \(1\le j\le r_n\) on the central window, is combined with a truncation argument:
\[
\mathbb P(X_n\le x)
=
\prod_{j=1}^{r_n}\bigl(1-M_{j,j}^{(n)}(x)\bigr)(1+o(1)).
\]
The same truncation applies to the limiting product,
\[
\Phi_\alpha(x)
=
\prod_{j=1}^{r_n}\Phi(v_\alpha(j,x))(1+o(1)).
\]
The exact fixed-\(\alpha\) asymptotic and the convergence rate follow from comparing the two products
term by term. 
\end{enumerate}

\subsubsection{Proof strategy for Theorem 2}

The rightmost eigenvalue is substantially more delicate because \(\widetilde M^{(n)}(x)\) is not
diagonal, and its structure depends strongly on the regime of $\alpha.$ 

\begin{enumerate}
\item \textbf{Case \(\alpha = +\infty\).}  
In this case, we prove that the Hilbert-Schmidt norm is negligible:  \[\|\widetilde{M}^{(n)}(x)\|_{\rm HS}^2:=\sum_{j, k=1}^n (\widetilde{M}_{j, k}^{(n)}(x))^2\ll 1.\] Hence   
\[
\det\bigl({\rm I}_n-\widetilde{M}^{(n)}(x)\bigr) = (1+o(1))\exp\bigl(-\operatorname{Tr}(\widetilde{M}^{(n)}(x))\bigr),
\]  
which leads to the Gumbel limit and its exact boundary rate.

\item \textbf{Case \(\alpha = 0\).}  
Here the Hilbert-Schmidt norm \(\|\widetilde{M}^{(n)}(x)\|_{\rm HS}\) is no longer negligible. Instead, we show that the off-diagonal contribution is small relative to the diagonal one:
\[
\sum_{1\le k<j\le n}\frac{\bigl(\widetilde{M}_{j,k}^{(n)}(x)\bigr)^2}{\bigl(1-\widetilde{M}_{j,j}^{(n)}(x)\bigr)\bigl(1-\widetilde{M}_{k,k}^{(n)}(x)\bigr)} \ll 1,
\]  
which implies  
\[
\det\bigl({\rm I}_n-\widetilde{M}^{(n)}(x)\bigr) = (1+o(1))\prod_{j=1}^n\bigl(1-\widetilde{M}_{j,j}^{(n)}(x)\bigr).
\]  

A further estimate shows that all but the first diagonal factor are asymptotically
negligible, so the problem again reduces to the leading diagonal term.

\item \textbf{Case \(\alpha \in (0, +\infty)\).}  
This is the most involved regime. We first prove the entrywise convergence \[
\widetilde{M}_{j,k}(x,\alpha) = \lim_{n\to\infty} \widetilde{M}^{(n)}_{j,k}(x), \qquad j, k\ge 1,
\]  
where $\widetilde{M}_{j,k}(x,\alpha)$ is the infinite-dimensional matrix defined in the statement of Theorem \ref{thmrealpart}. Let \(\widehat{M}^{(n)}(x,\alpha)\) denote its truncation to the first \(n\) rows and columns. Then 
\[
\lim_{n\to\infty} \det\bigl({\rm I}_n-\widetilde{M}^{(n)}(x)\bigr) = \lim_{n\to\infty} \det\bigl({\rm I}_n-\widehat{M}^{(n)}(x,\alpha)\bigr) = \det\bigl({\rm I}-\widetilde{M}(x,\alpha)\bigr).
\]  
The trace-class property of \(\widetilde M(x,\alpha)\) makes the Fredholm determinant
well defined and identifies the limiting distribution.

The continuity of the transition is handled directly at the Fredholm-determinant
level. As \(\alpha\to0^+\), we prove
\[
\lim_{\alpha\to0^+}\det({\rm I}-\widetilde M(x,\alpha))=\Phi(x),
\]
while as \(\alpha\to+\infty\),
\[
\lim_{\alpha\to+\infty}\det({\rm I}-\widetilde M(x,\alpha))
=
\exp\big(-\lim_{\alpha\to+\infty}{\rm Tr}(\widetilde M(x,\alpha))\big)
=
e^{-e^{-x}}.
\]
The first limit follows from entrywise continuity together with uniform summability,
and the second from the fact that \(\|\widetilde M(x,\alpha)\|_{\mathrm{HS}}\ll1\) for large
\(\alpha\). 
\end{enumerate}

\subsubsection{A few words on the method}

For determinantal point processes, the standard route to the limiting law of an
edge observable is the first-order approximation
\[
\det({\rm I}-\mathbb{K}_n|_E)\approx \exp\bigl(-{\rm Tr}(\mathbb{K}_n|_E)\bigr),
\]
which is effective when \(\|\mathbb{K}_n|_E\|_{\mathrm{HS}}^2\ll1\). In the present paper, this
mechanism works directly only in the sparse-factor regime \(\alpha=+\infty\). 
For
\(0\le \alpha<+\infty\), the determinant must be analyzed through finer asymptotics
of the reduced matrices \(M^{(n)}(x)\) and \(\widetilde M^{(n)}(x)\). This is precisely
what makes the diagonal modulus problem and the non-diagonal real-part problem
behave so differently. We expect that this approach can also be adapted to other matrix products, such as products of truncated unitary matrices or the spherical ensemble.

\subsection{Structure and notations}

The paper is organized as follows. Section~2 collects several lemmas that hold
uniformly in \(k_n\). Sections~3 and~4 treat the boundary regimes
\(\alpha=+\infty\) and \(\alpha=0\), respectively. Section~5 is devoted to the
proportional regime \(\alpha\in(0,+\infty)\), with Subsection~5.1 for the spectral
radius and Subsection~5.2 for the rightmost eigenvalue. Section~6 verifies the
continuity of the two interpolating families at the boundaries \(\alpha=0\) and
\(\alpha=+\infty\). The remaining technical proofs are collected in Section~7.

We use the following asymptotic notation. For a positive sequence \(z_n>0\), we write
\(t_n=O(z_n)\) if \(\limsup_{n\to\infty}|t_n|/z_n<+\infty\), and \(t_n=o(z_n)\) if
\(t_n/z_n\to0\). When \(t_n\ge0\), we write \(t_n\ll z_n\) (equivalently \(z_n\gg t_n\))
to mean \(t_n=o(z_n)\). For non-negative functions \(f\) and \(g\), we write
\(f\lesssim g\) if there exists a constant \(C>0\) such that \(f\le Cg\) for all
admissible arguments; \(f\gtrsim g\) is defined analogously, and \(f\asymp g\) means
both \(f\lesssim g\) and \(g\lesssim f\).

We also fix several pieces of notation used throughout the paper. The matrix
\(M^{(n)}(x)\) is the finite-dimensional matrix associated with the spectral-radius, while \(\widetilde M^{(n)}(x)\) is the corresponding matrix for the
rightmost eigenvalue. In the proportional regime, \(\widetilde M(x,\alpha)\) denotes the
trace-class limiting operator on \(\ell^2(\mathbb N)\). Tilded symbols always refer to the rightmost-eigenvalue problem; untilded symbols refer to the spectral-radius problem.

Finally, \(\Phi\) denotes the standard normal distribution
function, \(\Psi=1-\Phi\), and \(\Lambda(x)=e^{-e^{-x}}\) denotes the Gumbel
distribution function. The symbol \(W_1\) stands for the \(1\)-Wasserstein distance.
When \(T\) is a trace-class operator, \(\det({\rm I}-T)\) denotes its Fredholm determinant.

\section{Preliminaries}\label{sec:preliminaries}
As noted in the introduction, we require the proper asymptotics of \(M_{j,j}^{(n)}(x)\) and \(\widetilde{M}_{j,k}^{(n)}(x).\) In this section, we first present two lemmas that give exact expressions for the entries of \(M^{(n)}(x)\) and \(\widetilde{M}^{(n)}(x)\) in terms of the random variables \(\{Y_j\}_{1\le j\le n}\). We then establish further lemmas in a unified form that covers both the interior regime \(\alpha\in(0, +\infty)\) and the boundary cases \(\alpha=0\) or \(\alpha=+\infty\). These results will be used in the subsequent sections to derive the precise asymptotics of \(M_{j,j}^{(n)}(x)\) and \(\widetilde{M}_{j,k}^{(n)}(x).\)
 
 Recall the matrices $M^{(n)}(x)$ and $\widetilde{M}^{(n)}(x),$ whose entries are defined by      	\[
     	M_{j, k}^{(n)}(x) = \int_{A(x)} \phi_{n-j}(z)\;\overline{\phi_{n-k}(z)}\,  \; d^2 z, \qquad  1\le k, j \le n,
     	\]
     as well as 
    $$ \widetilde{M}_{j, k}^{(n)}(x) = \int_{\widetilde{A}(x)} \phi_{n-j}(z)\;\overline{\phi_{n-k}(z)}\,  \; d^2 z, \qquad  1\le k, j \le n.$$

\begin{lem}\label{ctor}
	Let \(Z_1, \cdots, Z_n\) be the eigenvalues of \(\prod_{j=1}^{k_n} \boldsymbol{A}_j\), and 
	Let \(\{S_{j, \,r},\ 1\le j\le n,\ 1\le r\le k_n\}\) be independent random variables such that \(S_{j, \,r}\) has density function \(y^{j-1}e^{-y}/(j-1)!\) for $y>0$ and define \(Y_j = \prod_{r=1}^{k_n} S_{j, \,r}\) for \(1 \le j \le n\).  We have $M_{j, k}^{(n)}(x)=0$ if $j\neq k$ and 
$$
	\aligned M_{j, j}^{(n)}(x)&=\mathbb{P}\big(\log Y_{n+1-j}\geq k_n\psi (n)+\frac{a_n+b_nx}{\sqrt{\alpha_n}}\big).
	\endaligned 
	$$
\end{lem}
\begin{proof}   
By definition, 
$$\aligned M_{j, j}^{(n)}(x)&=\int_{A(x)} |\phi_{n-j}(z)|^2 \; d^2 z =\int_{A(x)}\frac{\varphi(z)}{\pi^{k_{n}}} \cdot \frac{|z|^{2(n-j)}}{((n-j)!)^{k_{n}}}d^2z.\\
\endaligned $$	
Set $L_n(x)=\frac{1}{2}(k_n\psi (n)+\frac{a_n+b_nx}{\sqrt{\alpha_n}})$ and apply \eqref{CJQnewunderstanding} to get 
$$\aligned M_{j, j}^{(n)}(x)=\frac{1}{(\pi(n-j)!)^{k_n}}\int_{\log |z_1\cdots z_{k_n}|\ge L_n(x)} \prod_{m=1}^{k_n}(|z_m|^{2(n-j)} e^{-|z_m|^2})  \prod_{m=1}^{k_n} d^2 z_m.
\endaligned$$ 
We leverage the spherical coordinate $z_m=r_m e^{i\theta_m}$ and then $t_m=r_m^2$  to derive  	
\begin{equation}\label{widehatMjj} \aligned  
M_{j, j}^{(n)}(x)&=\frac{2^{k_n}}{((n-j)!)^{k_n}}\int_{\log (\prod_{m=1}^{k_n}r_m)\ge L_n(x)} \prod_{m=1}^{k_n}(r_m^{2(n-j)+1} e^{- r_m^2})  \prod_{m=1}^{k_n} d r_m \\
&=\frac{1}{((n-j)!)^{k_n}}\int_{\log (\prod_{m=1}^{k_n}t_m)\ge 2L_n(x)} \prod_{m=1}^{k_n}(t_m^{n-j} e^{-t_m} ) \prod_{m=1}^{k_n} d t_m. \\
 \endaligned 
\end{equation}
Observe that the integrand $\frac{1}{((n-j)!)^{k_n}}\prod_{m=1}^{k_n}t_m^{n-j} e^{- t_m} $ in \eqref{widehatMjj} is the joint density function of $(S_{n-j+1, 1}, \cdots, S_{n-j+1, k_n})$  and then the fact $\sum_{m=1}^{k_n}\log S_{n-j+1, m}$ having the same distribution as $\log Y_{n-j+1}$ implies 
\begin{equation}\label{Mjjnx} \aligned  
M_{j, j}^{(n)}(x)=\mathbb{P}(\log Y_{n-j+1}\ge k_n\psi (n)+\frac{a_n+b_nx}{\sqrt{\alpha_n}}). \endaligned 
\end{equation}

A similar argument leads 
\begin{equation}\label{widehatMjk}\aligned M_{j, k}^{(n)}(x)
&= \frac{1}{(\pi^2 (n-j)!(n-k)!)^{k_n/2}}\int_{A(x)}\varphi(z) \bar{z}^{n-j} z^{n-k} d^2z \\
&=\frac{1}{((2\pi)^2 (n-j)!(n-k)!)^{k_n/2}}\int_{[0, 2\pi]^{k_n}}\exp((j-k)i\sum_{m=1}^{k_n}\theta_m)\prod_{m=1}^{k_n}d\theta_m \\
&\quad \times \int_{\log (\prod_{m=1}^{k_n}t_m)\ge 2L_n(x)}\prod_{m=1}^{k_n}t_m^{n-\frac{j+k}{2}} e^{-t_m}  \prod_{m=1}^{k_n} d t_m \\
&=0,\endaligned \end{equation}  
where $0$ comes from the first integral and this reflects the rotation invariance of the correlation kernel.  
The proof is completed. 
\end{proof}

\begin{rmk}  The expressions \eqref{Mjjnx} and \eqref{widehatMjk} imply that
\begin{equation}\label{transreal}
\mathbb{P}(X_n\le x)={\rm det}({\rm I}_n-M^{(n)}(x))=\prod_{j=1}^n \mathbb{P}\big(\log Y_{n+1-j}\geq k_n\psi (n)+\tfrac{a_n+b_nx}{\sqrt{\alpha_n}}\big).
\end{equation}

Jiang and Qi \cite{JQ17} used the rotation invariance of $\varphi$ in \eqref{CJQnewunderstanding} together with the characteristic function method to show that
$$g(|Z_1|,\dots,|Z_n|)\stackrel{d}{=}g(|Y_1|,\dots,|Y_n|)$$
for any symmetric function $g$, which directly yields \eqref{transreal}. This reduction phenomenon was first observed by Kostlan \cite{Kostlan1992} for the complex Ginibre ensemble and has since been extended to various classes of complex non-Hermitian random matrices. Here we present an alternative proof based on the determinantal point process method in \cite{HKPV2009}, which paves the way for analyzing the largest real-part. \end{rmk}

The matrix $\widetilde{M}^{(n)}(x)$ associated with $\mathbb{P}(\widetilde{X}_n\le x)$ is no longer diagonal. This non-diagonality renders the determinant $\det({\rm I}_n-\widetilde{M}^{(n)}(x))$ considerably more difficult to evaluate. 
 
\begin{lem}\label{Mjjlem} Let \( \widetilde{M}_{j, k}^{(n)}(x) \)  be defined as above, and let \( (Y_j)_{1\le j\le n} \) be the sequence of random variables given in Lemma \ref{ctor}. Let \( \Theta \) be a random variable, independent of \( (Y_j)_{1\le j\le n} \), following a uniform distribution on \( [0, \pi/2] \). Then, for \( 1 \le j \le n \),
\begin{equation}\label{Mjjfor}
\widetilde{M}_{j, j}^{(n)}(x) = \mathbb{P}\bigl( \log Y_{n+1-j} + \log \cos^2\Theta \ge k_n\psi(n) + \frac{\widetilde{a}_n + \widetilde{b}_n x}{\sqrt{\alpha_n}} \bigr)
\end{equation}
and \( \widetilde{M}_{j, k}^{(n)}(x) = 0 \) whenever \( j-k \) is odd. For even \( j-k \), 
\begin{equation}\label{Mjkfor}
\begin{aligned}
\widetilde{M}_{j, k}^{(n)}(x) &= \frac{2(\big( n-\frac{j+k}{2} \big)!)^{k_n}}{\pi \bigl( (n-j)! (n-k)! \bigr)^{k_n/2}} \\
&\quad \times \int_{0}^{\pi/2} \cos\bigl( (j-k)\theta \bigr) \, \mathbb{P}\bigl( \log Y_{n+1-\frac{j+k}{2}} + \log\cos^2\theta \ge k_n\psi(n) + \frac{\widetilde{a}_n + \widetilde{b}_n x}{\sqrt{\alpha_n}} \bigr) \, d\theta.
\end{aligned}
\end{equation}
Consequently, for $1\le j\neq k\le n,$ we obtain the estimate
\[
|\widetilde{M}_{j, k}^{(n)}(x)| \le \min\big\{\widetilde{M}_{\frac{j+k}{2}, \frac{j+k}{2}}^{(n)}(x), \sqrt{\widetilde{M}_{j, j}^{(n)}(x)\widetilde{M}_{k, k}^{(n)}(x)}\big\}.
\]
 \end{lem} 
\begin{proof}
Setting $\widetilde{L}_n(x)=\frac{1}{2}(k_n\psi (n)+\frac{\widetilde{a}_n+\widetilde{b}_nx}{\sqrt{\alpha_n}})$ and following the same reasoning as for \eqref{widehatMjj}, we have 	
\begin{equation}\label{newtr} \aligned  
\widetilde{M}_{j, j}^{(n)}(x)
&=\frac{1}{(2\pi(n-j)!)^{k_n}}\int_{\log (\cos^2(\sum_{m=1}^{k_n}\theta_m)\prod_{m=1}^{k_n}t_m)\ge 2\widetilde{L}_n(x)} \prod_{m=1}^{k_n}t_m^{n-j} e^{-t_m}  \prod_{m=1}^{k_n} d t_m d\theta_m \endaligned 
\end{equation}
and then we can understand \eqref{newtr} in the similar way as 
 $$\widetilde{M}_{j, j}^{(n)}(x)=\mathbb{P}(\log Y_{n+1-j}+\log \cos^2(\sum_{m=1}^{k_n}\Theta_m)\ge 2\widetilde{L}_n(x)).$$ 
Here, $\{\Theta_m\}_{1\le m\le k_n}$ are i.i.d. uniform on $[0, 2\pi).$ As is well known, the sum modulo $2\pi$ of such independent random variables is again uniform on $[0, 2\pi).$ The property of $\cos^2 \theta$ reduces $[0, 2\pi)$ to $[0, \pi/2)$ and then 
 \begin{equation}\label{Mjjf}  
\widetilde{M}_{j, j}^{(n)}(x)=\mathbb{P}(\log Y_{n+1-j}+\log \cos^2\Theta\ge 2\widetilde{L}_n(x)).
\end{equation}
Similar argument as \eqref{widehatMjk} leads  
$$\aligned \widetilde{M}_{j, k}^{(n)}(x)&=\frac{1}{\pi((n-j)!(n-k)!)^{k_n/2}}\times\\
&\int_{\log (\cos^2\theta\prod_{m=1}^{k_n}t_m)\ge 2\widetilde{L}_n(x), \theta\in [0, \pi]} \cos((j-k)\theta)\prod_{m=1}^{k_n}t_m^{n-\frac{j+k}{2}} e^{-t_m}  d\theta\prod_{m=1}^{k_n} d t_m\endaligned $$   and then similarly as for \eqref{Mjjf} we derive $$\aligned \widetilde{M}_{j, k}^{(n)}(x)&=\frac{((n-\frac{j+k}{2})!)^{k_n}}{\pi((n-j)!(n-k)!)^{k_n/2}}\\
&\times \int_{0}^{\pi} \cos((j-k)\theta)\mathbb{P}(\log Y_{n+1-\frac{j+k}{2}}+\log\cos^2\theta\ge 2\widetilde{L}_n(x)) d\theta.\endaligned $$
When \(j-k\) is odd, the integrand combines a symmetric probability factor with \(\cos((j-k)\theta)\), which is antisymmetric with respect to \(\theta = \pi/2\); hence the integral vanishes. For even \(j-k\), the entire integrand is symmetric about \(\theta = \pi/2\), so the integral over \([0, \pi]\) equals twice the integral over \([0, \pi/2]\). This completes the verification of \eqref{Mjkfor}. 
Review an elementary inequality 
$\frac{((n-\frac{j+k}{2})!)^2}{(n-j)!(n-k)!}\le 1$ and the facts $|\cos((j-k)\theta)|\le 1,$ whence 
	the formula \eqref{Mjkfor} helps us to derive  
	$$
		\aligned |\widetilde{M}_{j, k}^{(n)}(x)|\le\frac{2}{\pi}\int_{0}^{\frac{\pi}{2}}\mathbb{P}(\log Y_{n+1-\frac{j+k}{2}}+\log\cos^2\theta\ge 2\widetilde{L}_{n}(x))d\theta=M_{\frac{j+k}{2}, \frac{j+k}{2}}^{(n)}(x). \endaligned
	$$ 
Finally, the proof is complete, as the Cauchy-Schwarz inequality  
$$ \big| \int_{\widetilde{A}(x)} \phi_j(z) \overline{\phi_k(z)} \, d^2z \big|^2 \le \int_{\widetilde{A}(x)} |\phi_j(z)|^2 \, d^2z \int_{\widetilde{A}(x)} |\phi_k(z)|^2 \, d^2z $$  
is equivalent to  
\begin{equation}\label{CSIneq} 
|\widetilde{M}_{j, k}^{(n)}(x)|^2 \le \widetilde{M}_{j, j}^{(n)}(x) \widetilde{M}_{k, k}^{(n)}(x),
\end{equation}  
which immediately yields the last conclusion.
\end{proof}

Now, we present several lemmas concerning the asymptotics of \(M^{(n)}_{j, j}(x)\), which will also be applicable to \(\widetilde{M}_{j, j}^{(n)}(x)\).  
As noted in the introduction, to derive the asymptotic expression for \(M^{(n)}_{j, j}(x)\), we shall employ either the central limit theorem for the i.i.d. setting or the Edgeworth expansion. Both approaches require knowledge of the moments of \(\log S_{j, 1}\) for \(1 \le j \le n\).

We begin by presenting properties of the digamma function \cite{Abramowitz1968}, which appears in the expectations related to \(\{\log S_{j,1}\}_{j=1}^n\). Subsequently, we state the corresponding expectations. These two lemmas will underpin our subsequent analysis.

\begin{lem}\label{diagammapro}  Let $\psi(x) = \Gamma'(x)/\Gamma(x)$ be the digamma function, with $\Gamma$ being the Gamma function.  
	\begin{enumerate}
			\item[\textup{(a)}.]  For any $s \geq 1$,
		\(
		\psi(j + s) - \psi(j) \leq \frac{s}{j}.
		\)
				\item[\textup{(b)}.]  For sufficiently large $z$, the following asymptotics hold:
		\begin{align*}
			\psi(z) &= \log z - \frac{1}{2z} + O(z^{-2}), \quad\quad
			\psi'(z) = \frac{1}{z} + \frac{1}{2z^2} + O(z^{-3});\\
			\psi^{(2)}(z)  &= -\frac{1}{z^2} - \frac{1}{z^3} + O(z^{-4}),\quad\quad
			\psi^{(3)}(z)  = \frac{2}{z^3} + \frac{3}{z^4} +  O(z^{-5}).
		\end{align*}
	\end{enumerate}
\end{lem}

Now, we directly state the expectations related to $\{\log S_{i,1}\}_{i=1}^{n}$ without proof. 
\begin{lem}\label{le}
	Let $(S_{j, 1})_{1\le j\le n}$ be defined as above. The following assertions hold:
\begin{enumerate}	
	\item[\textup{(a)}.]  For the moments and moment generating function, we have:
	\begin{align*}
		\mu:=\mathbb{E}[\log S_{j, 1}] = \psi(j)&, \quad\quad
		\sigma^{2}:={\rm Var}[\log S_{j,1}] = \psi'(j); \\
		\mathbb{E}[S_{j, 1} \log S_{j, 1}] = j \psi(j+1)&, \quad\quad
		\mathbb{E}[e^{\lambda \log S_{j,1}}] = \frac{\Gamma(j+\lambda)}{\Gamma(j)}.
	\end{align*}
	
\item[\textup{(b)}.] For sufficiently large $j$, the skewness and kurtosis correction terms for $\log S_{j,1}$ are given by:
$$\gamma_{1}:=\frac{\mathbb{E}(\log S_{j,1}-\mu)^{3}}{6\sigma^{3}}=-\frac{1}{6\sqrt{j}}(1+O(j^{-1}))\quad \textup{(skewness correction)};$$
$$\gamma_{2}:=\frac{\mathbb{E}[(\log S_{j,1}-\mu)^4] - 3\sigma^4}{24\sigma^4} =\frac{1}{12j}(1+O(j^{-1}))\quad \textup{(kurtosis correction)}.$$
\end{enumerate}
\end{lem}

Next, using the expectation of 
$ \log Y_{n-j+1} $ from Lemma \ref{le} and the Markov inequality, we derive  an upper bound for $M^{(n)}_{j, j}(x).$

\begin{lem}\label{H}
		Let  $M^{(n)}_{j, j}(x)$ be defined as above. Then uniformly on $2\le j\ll n,$ we have 
	$$	M^{(n)}_{j, j}(x)
	\leq\exp\big\{-\frac{(j-1)^2}{4\alpha_n}-\frac{(j-1)(a_n+b_n x)}{2\sqrt{\alpha_n}} \big\}
	$$
	as $n\to+\infty.$
\end{lem}
\begin{proof}
	 Given any $t>0,$  it follows from the Markov inequality that for each $2\le j \ll n$ and any \( x > 0 \),
	\begin{equation}\label{uppermjj}
	M^{(n)}_{j, j}(x)\leq \mathbb{E}(e^{t \log Y_{n-j+1}})\exp \left\{ -t\left( k_n\psi (n)+\alpha_n^{-1/2}(a_n+b_nx) \right) \right\}.\\ 
	\end{equation} 
Leveraging Lemma \ref{le} and the fact of $\log Y_{n-j+1}=\sum_{r=1}^{k_n}\log S_{n-j+1, r} $, we have  
$$\mathbb{E}(e^{t \log Y_{n-j+1}})=(\mathbb{E}e^{t \log S_{n-j+1, 1}})^{k_n}=(\frac{\Gamma(n-j+1+t)}{\Gamma(n-j+1)})^{k_n}.$$ 
We rewrite 
$$\aligned \log &\frac{\Gamma(n-j+1+t)}{\Gamma(n-j+1)}-t \psi(n-j+1)=\int_{0}^t(\psi(n-j+1+s)-\psi(n-j+1)) ds
\endaligned $$ and then apply Lemma \ref{diagammapro} to get an upper bound  
$$\int_0^t \frac{s}{n-j+1} ds= \frac{t^2}{2(n-j+1)}\le \frac{t^2}{n}.$$ 
Meanwhile, 
$$\psi (n) - \psi(n-j+1)=\log n-\log (n-j+1)-\frac{1}{2n}+\frac{1}{2(n-j+1)}+o(n^{-2})\ge \frac{j-1}{n}.$$
Putting these two bounds into the expression \eqref{uppermjj}, we derive 
\begin{equation}\label{k12}\log M^{(n)}_{j, j}(x)\le \frac{t^2-(j-1)t}{\alpha_n}-\frac{(a_n+b_n x)t}{\sqrt{\alpha_n}}.\end{equation}
	for all $ t > 0 $ and sufficiently large $ n. $ 
	Selecting $ t=\frac{j-1}{2}$ gives
	$$	M^{(n)}_{j, j}(x)
	\leq\exp\big\{-\frac{(j-1)^2}{4\alpha_n}-\frac{(j-1)(a_n+b_n x)}{2\sqrt{\alpha_n}} \big\}.
	$$
	The proof is then completed. 
\end{proof}
By virtue of the properties of 
$\log S_{n-j+1, \,r}$
in Lemma \ref{le}, the Edgeworth expansion of 
$\log S_{n-j+1, \,r}$
can be derived (see \cite{DasGupta2008}; for the proof, see \cite{Esseen1945}). This expansion provides a more precise estimation compared to the central limit theorem.
\begin{lem}\label{ed}
	Given $0\leq m\ll n$ and $n\lesssim k_n.$ For any $|x_n|\lesssim n^{1/6},$
$$\mathbb{P}\big( \frac{\log Y_{n-m}-k_n\psi(n-m)}{\sqrt{k_n\psi'(n-m)}}\leq x_n\big)=\Phi(x_n)-\frac{(1-x_n^{2})}{6\sqrt{k_nn}}\phi(x_n)+O(k_n^{-\frac{3}{2}}+k_n^{-\frac{1}{2}}n^{-\frac{3}{2}}).$$
\end{lem} 
 \begin{proof}
 	 	
 	We write $$\frac{\log Y_{n-m}-k_n\psi(n-m)}{\sqrt{k_n\psi'(n-m)}}=\sum_{r=1}^{k_n}\frac{\log S_{n-m, r}-\psi(n-m)}{\sqrt{k_n\psi'(n-m)}}=:\sum_{r=1}^{k_n}\zeta_{n-m, r}$$ and let $\varphi_{n-m}$ be the characteristic function of $\zeta_{n-m, 1}.$ 
 	Straightforward calculus and Lemma \ref{le} yield 
 	$$\varphi_{n-m
 	}(t)=\exp\big(-\frac{it \psi(n-m)}{\sqrt{\psi'(n-m)}}\big)\frac{\Gamma(n-m+\frac{it}{\psi'(n-m)})}{\Gamma(n-m)}.$$
 	By the modulus formula of the gamma function, we have
 	$$|\varphi_{n-m}(t)|=\prod_{j=0}^{+\infty}\big(1+\frac{t^2}{\psi'(n-m)(n-m+j)^2}\big)^{-1/2}.$$
 	Since for any fixed $t,$
 	$$\frac{t^2}{\psi'(n-m)(n-m+j)^2}=\frac{t^2(n-m)(1+o(1))}{(n-m+j)^2}\to 0, \quad \text{as}\,\, n\to \infty .$$
 	It follows that
 		$$\aligned |\varphi_{n-m}(t)|&=\exp\{-\frac{1}{2}\sum_{j=0}^{+\infty}\log (1+\frac{t^2}{\psi'(n-m)(n-m+j)^2})\}\\&=\exp\{-\frac{1+o(1)}{2}\sum_{j=0}^{+\infty}\frac{t^2(n-m)}{(n-m+j)^2}\}.
 		\endaligned$$
 	We note
 	$$ \sum_{j=0}^{+\infty}\frac{t^2(n-m)}{(n-m+j)^2}\geq \int_{0}^{+\infty}\frac{(n-m)t^2}{(n-m+y)^2} dy=t^2,$$
 	which implies 
 	$$\sup_{|t|>1}|\varphi_{n-m}(t)|\leq \sup_{|t|>1}\exp\{-\frac{(1+o(1))t^2}{2}\}\leq e^{1/3}<1.$$
 	Therefore, \((\log S_{n-m, \,r})_{1\le r\le k_n}\) satisfy the uniform Cram\'er condition.

 For $0\leq m\ll n$, Lemma \ref{le} leads
 	\[
 	\mathbb{E}[(\log S_{n-m, \,1}-\psi(n-m))^4]=\psi^{(3)}(n-m)+3\psi'(n-m)^2,
 	\]
 	whence
 	\[
 	\mathbb{E}[|\zeta_{n-m, r}|^4]
 	= \frac{\psi^{(3)}(n-m)+3(\psi'(n-m))^2}{k_n^2 (\psi'(n-m))^2}\ll 1.
 	\]
Here, the $\ll$ is doe to 
 	$$\frac{\psi^{(3)}(n-m)}{(\psi'(n-m))^2}\asymp \frac{1}{n}\to 0$$ and \( k_n \geq 1 \).

 Thereby, the i.i.d random variable sequence 
  \((\log S_{n-m, \,r})_{1\le r\le k_n}\) fulfills the conditions required for applying the Edgeworth expansion.  
  
  Hence,  
 	\begin{align*}&\mathbb{P}\left( \sum_{r=1}^{k_n}\zeta_{n-m, r}\leq x_n\right)\\
 =&\Phi(x_n) +  \frac{\gamma_{1}(1-x_n^2)\phi(x_n)}{\sqrt{k_n}}  - \frac{\phi(x_n)}{k_n}( \gamma_2(x_n^3-3x_n)+\frac{\gamma_{1}^{2}}{72}(x_n^5-10x_n^3+15x_n))+O(k_n^{-\frac{3}{2}}),	\label{xn}\end{align*}
 where $\gamma_1$ and $\gamma_2$ are the skewness correction and the kurtosis correction of  $\log S_{n-m, 1},$ respectively. Lemma \ref{le} entails 
 $$\gamma_{1}=-\frac{1+O(n^{-1})}{6\sqrt{n}}\quad \text{and}\quad\gamma_{2}=\frac{1+O(n^{-1})}{12n}.$$
 Since $x^{k}\phi(x)$ is a bounded function for $0\leq k\leq 5$. For $|x_n|\leq n^{1/6},$ we see clearly that the third term is negligible with respect to the second term. 
 Therefore,
 $$\mathbb{P}\left( \frac{\log Y_{n-m}-k_n\psi(n-m)}{\sqrt{k_n}\psi'(n-m)}\leq x_n\right)=\Phi(x_n)-\frac{(1-x_n^{2})}{6\sqrt{k_nn}}\phi(x_n)+O(k_n^{-\frac{3}{2}}+k_n^{-\frac{1}{2}}n^{-\frac{3}{2}}).$$
 \end{proof}
 
At last, we borrow a summation property in \cite{MaMeng2025} as follows. 
\begin{lem}\label{sum} Let $\gamma_n$ be a positive sequence  
	and set $$\varsigma_n(j)=\frac{j-1}{\gamma_n}+t_n$$ for $1\le j\le n.$  Given $L\geq 1$ and any $ t_{n}$ such that
	$ 1\ll \varsigma_{n}(L)$ and let $c>0$ be a fixed constant. 
	\begin{enumerate} 
	\item When $\gamma_n$ satisfies $1\ll \varsigma_{n}(L)\ll \gamma_n,$ we have 
	$$ \aligned \sum\limits_{j=L}^{+\infty}\varsigma_{n}^{-1}(j)e^{- c\varsigma_{n}^{2}(j)}&=\frac{\gamma_n e^{- c\varsigma_{n}^{2}(L)}}{2c \varsigma_{n}^{2}(L)}(1+ O(\varsigma_n^{-2}(L)+\varsigma_n(L)\gamma_n^{-1}))\lesssim \frac{\gamma_n e^{- c\varsigma_{n}^{2}(L)}}{ \varsigma_{n}^{2}(L)} . \endaligned 
	$$	
	\item When $\gamma_n$ is bounded, 
	$$ \sum\limits_{j=L}^{+\infty}\varsigma_{n}^{-1}(j)e^{- c\varsigma_{n}^{2}(L)}\lesssim \frac{ e^{- c\varsigma_{n}^{2}(L)}}{\varsigma_{n}(L)}.
	$$
	\end{enumerate}
\end{lem}
 
 \section{Proof of the Theorems for $ \alpha=+\infty$}\label{sec:alphainfinity} 
 
 In this section, we assume 
$\alpha=\lim\limits_{n\to\infty}\frac{n}{k_n}=+\infty,$ under which 
$k_n$ may be of order 
$ O(1). $ A most important advantage for this case is  that the matrix $\widetilde{M}^{(n)}(x)$ satisfies 
$$\|\widetilde{M}^{(n)}(x)\|_{\rm HS}\ll 1
$$
such that 
$${\rm det}({\rm I}_n-\widetilde{M}^{(n)}(x))=\exp(-{\rm Tr}(\widetilde{M}^{(n)}(x)))(1+o(1)).$$ 
Simultaneously, $M^{(n)}_{j, j}(x)=o(1)$ uniformly on $1\le j\le n$ leading  
$$\mathbb{P}(X_n\le x)=\prod_{j=1}^{n}(1-M^{(n)}_{j, j}(x))=\exp(-{\rm Tr}(M^{(n)}(x)))(1+o(1)).$$ 

Then, what we need to do is to find the precise asymptotic of the two traces.

Next, we prepare some lemmas for asymptotics on $M^{(n)}_{j, j}(x)$ and $\widetilde{M}^{(n)}_{j, j}(x).$  
		
\subsection{Estimates on $M^{(n)}_{j, j}(x)$ and $\widetilde{M}^{(n)}_{j, j}(x)$}	
First, we set
\[
j_n=\lfloor\frac{1}{5}\sqrt{\alpha_n\log \alpha_n}\rfloor \qquad \text{and}\qquad 
t_n=\lfloor 8\sqrt{\alpha_n\log \alpha_n}\rfloor.
\]   \begin{lem}\label{333}
	Recall   
 \[a_n = \sqrt{\log(\alpha_n + 1)} - \frac{\log(\sqrt{2\pi}\log(\alpha_n + e^{\frac{1}{\sqrt{2\pi}}}))}{\sqrt{\log(\alpha_n + e)}} \quad \text{and}\quad b_n=\frac{1}{\sqrt{\log(\alpha_n + e)}}.\]Set
	$$u_n(j, x) = \frac{j-1}{\sqrt{\alpha_n}} + a_n + b_n x.$$
	For  $|x|\leq2\log\log \alpha_n$ and as $n\to\infty,$  the following estimates hold uniformly: 
	\begin{enumerate}
		\item If $1\leq j\leq j_n,$ then
		$$ M^{(n)}_{j, j}(x)=\frac{1+O((\log \alpha_n)^{-1})}{\sqrt{2\pi}u_n(j,x)}\exp(-\frac{u^{2}_n(j, x)}{2}).
		$$
		\item If $j_n< j\leq  t_n,$ then
		$$
			M^{(n)}_{j, j}(x)\leq \frac{1}{u_n(j,x)}e^{-\frac{3u^{2}_n(j, x)}{8}}+n^{-4/5} .
		$$
	\end{enumerate}
\end{lem}
\begin{proof}

Even $\log Y_{n-j+1}$ is a sum of i.i.d. random variables, we cannot directly apply the central limit theorem because $k_n$ may be a finite constant. Instead, we relate $\log S_{n-j+1, r}$ to $S_{n-j+1, r}$ by noting that
\begin{equation}\label{newsum}\sum_{r=1}^{k_n} S_{n-j+1, \,r}\stackrel{d}{=}\sum_{i=1}^{(n-j+1) k_n}\xi_i,\end{equation}
where $\{\xi_i\}_{i\ge 1}$ are i.i.d. exponential with parameter $1$.

Indeed, with $\ell:=n-j+1$, $M^{(n)}_{j, j}(x)$ can be rewritten as
$$\aligned \mathbb{P}\big(\sum_{r=1}^{k_n}\frac{S_{\ell, r}-\ell}{\ell}&>k_n(\psi (n)-\log \ell )+\frac{a_n+b_nx}{\sqrt{\alpha_n}}-\sum_{r=1}^{k_n}(\log \frac{S_{\ell,r}}{ \ell}
-\frac{S_{\ell,r}-\ell}{\ell})\big).\endaligned$$

For some $\varepsilon>0$ to be chosen later, define
$$B_{\varepsilon}:=\bigl\{\bigl|\sum_{r=1}^{k_n}\bigl(\log \frac{S_{\ell,r}}{ \ell}-\frac{S_{\ell, r}-\ell}{\ell}\bigr)\bigr|\ge \varepsilon\bigr\}.$$
Then we have the following bounds:
\begin{equation}\label{A}
\aligned M^{(n)}_{j, j}(x)&\le \mathbb{P}\bigl(\sum_{r=1}^{k_n}\frac{S_{\ell, r}-\ell}{\ell}>k_n(\psi (n)-\log \ell)+\frac{a_n+b_nx}{\sqrt{\alpha_n}}-\varepsilon\bigr)+\mathbb{P}(B_{\varepsilon})
\endaligned
\end{equation}
and
\begin{equation}\label{B}
\aligned M^{(n)}_{j, j}(x)&\ge \mathbb{P}\bigl(\sum_{r=1}^{k_n}\frac{S_{\ell,r}-\ell}{\ell}>k_n(\psi (n)-\log \ell)+\frac{a_n+b_nx}{\sqrt{\alpha_n}}+\varepsilon\bigr).
\endaligned
\end{equation}

Let $\varepsilon=\alpha_n^{-3/5}.$ We first consider the case 
$1\le j\leq j_n$, which guarantees
$$1\ll u_n(j, x)\le u_n(j_n, 2\log\log \alpha_n)=\frac{6}{5}\sqrt{\log \alpha_n}(1+o(1)).$$ 
Define 
$$A_{\pm \epsilon}:=\{\sum_{r=1}^{k_n}\frac{S_{\ell,r}-\ell}{\ell}>k_n(\psi (n)-\log \ell)+\frac{a_n+b_nx}{\sqrt{\alpha_n}}\pm\varepsilon\}.$$
When \( 1\ll u_n(j, x) \lesssim \sqrt{\log \alpha_n} \), we claim that  
\begin{equation}\label{equaforsumn} \aligned 
	\mathbb{P}(A_{\pm \epsilon})=\frac{1+O(u_n^{-2}(j, x))}{\sqrt{2\pi}u_n(j, x)}\exp(-\frac12 u_n^2(j, x)),
	\endaligned 
\end{equation}
whose proof is given in the Appendix.
The right hand side of \eqref{equaforsumn} is decreasing in $u_n(j, x),$ and then the fact $u_n(j_n, x)=\frac{6}{5}\sqrt{\log\alpha_n}+O(\log\log \alpha_n)$ indicates
\begin{equation}\label{tailp}\mathbb{P}(A_{\pm \epsilon})\gtrsim u_n^{-1}(j_n, x)\exp(-\frac12 u_n^2(j, x))\asymp  (\log\alpha_n)^{-1/2} \alpha_n^{-18/25}. \end{equation}  

We now give an upper bound for $B_{\varepsilon}.$

By the triangle inequality, we first obtain
$$\aligned \mathbb{P}(B_{\varepsilon})
&\leq \mathbb{P}\Bigl(\Bigl|\sum_{r=1}^{k_n} \bigl(\log \frac{S_{\ell, r}}{\ell} - \frac{S_{\ell, r}-\ell}{\ell}\bigr)-k_n(\psi(\ell)-\log \ell)\Bigr| \geq \varepsilon-k_n|\psi(\ell)-\log \ell|\Bigr). \endaligned$$

Applying Markov's inequality together with basic properties of variance, and under the condition $\varepsilon > k_n|\psi(\ell)-\log \ell|$, we obtain
$$ \mathbb{P}\left(B_{\varepsilon}\right)	
\leq\frac{k_n\operatorname{Var}\bigl(\log S_{\ell, 1} - \frac{S_{\ell, 1}}{\ell}\bigr)}{(\varepsilon-k_n|\psi(\ell)-\log \ell|)^{2}}.  $$	
Lemma \ref{le} gives
$$\aligned \operatorname{Var}\bigl(\log S_{\ell, 1} - \tfrac{S_{\ell, \,1}}{\ell}\bigr)
&=\operatorname{Var}(\log S_{\ell, 1})+\ell^{-2}\operatorname{Var}(S_{\ell, 1})-2\ell^{-1}\mathbb{E}\bigl((S_{\ell, 1}-\ell)\log S_{\ell, 1}\bigr)\\
&=\psi'(\ell)+\ell^{-1}-2\bigl(\psi(\ell+1)-\psi(\ell)\bigr), \endaligned$$
and then by Lemma \ref{diagammapro} we have
$$\aligned \operatorname{Var}\bigl(\log S_{\ell, 1} - \tfrac{S_{\ell, \,1}}{\ell}\bigr)=\frac{1}{2\ell^2}+O(\ell^{-3}).\endaligned$$
Since $\varepsilon=\alpha_n^{-3/5}\gg k_n|\psi(\ell)-\log \ell|$, we derive
\begin{equation}\label{upforb}\mathbb{P}(B_{\varepsilon})\leq n^{-4/5} k_n^{-1/5}.\end{equation}

Comparing the right-hand sides of \eqref{tailp} and \eqref{upforb}, we see that $\mathbb{P}(B_{\varepsilon})$ is negligible in both \eqref{A} and \eqref{B}. Consequently, from \eqref{equaforsumn} we obtain
$$M^{(n)}_{j, j}(x)=\frac{1+O((\log \alpha_n)^{-1})}{\sqrt{2\pi}u_n(j,x)}\exp\bigl(-\frac{u^{2}_n(j, x)}{2}\bigr)$$
uniformly for $1\le j\le j_n$ and $|x|\le 2\log\log \alpha_n$.

For $j_n < j \le  t_n$, it is not necessarily true that $\mathbb{P}(A_{\pm \varepsilon}) \gg \mathbb{P}(B_{\varepsilon})$; in other words, $\mathbb{P}(B_{\varepsilon})$ may not be negligible. Hence, we only provide an upper bound for $M^{(n)}_{j, j}(x)$ using \eqref{A} and \eqref{upforb}:
$$M^{(n)}_{j, j}(x)\leq \mathbb{P}(A_{- \varepsilon})+\mathbb{P}(B_{\varepsilon})\leq \frac{1}{u_n(j,x)}e^{-\frac{3u^{2}_n(j, x)}{8}}+n^{-4/5},$$
where the proof of the inequality
\begin{equation}\label{a83}
\mathbb{P}(A_{- \varepsilon}) \leq \frac{1}{u_n(j,x)}e^{-\frac{3u^{2}_n(j, x)}{8}}
\end{equation}
will be given in the appendix together with \eqref{equaforsumn}.

	\end{proof}

We now provide an appropriate estimate for \(\widetilde{M}^{(n)}_{j, j}(x).\) Recall that
\[
\widetilde{b}_n = \frac{\sqrt{2}}{\sqrt{\log(\alpha_n+e^2) }},\qquad 
\widetilde{a}_n = \sqrt{\frac{\log(\alpha_n+1) }{2}} - \frac{\sqrt{2}\bigl(\log (2^{-\frac{3}{4}}\pi)+\frac{5}{4}\log \log(\alpha_n+e^{2^{3/5}\pi^{-4/5}})\bigr)}{\sqrt{\log(\alpha_n+e^2) }}.
\]
Define
\[
g_n(j, x, t)=\mathbb{P}\Bigl(\log Y_{n-j+1}\ge k_n\psi(n)+\frac{\widetilde{a}_n+\widetilde{b}_n x}{\sqrt{\alpha_n}}+t\Bigr),\quad t>0.
\]
Then by the law of total expectation,
\begin{equation}\label{widecn}
\widetilde{M}^{(n)}_{j, j}(x)=\frac{1}{\pi}\int_0^{+\infty} g_n(j, x, t) (e^{t}-1)^{-1/2}\, dt.
\end{equation}
Comparing the definitions of \({M}^{(n)}_{j, j}(x)\) and \(g_n(j, x, t),\) the only difference is that \(u_n(j, x)=\frac{j-1}{\sqrt{\alpha_n}}+a_n+b_n x\) is replaced by
\begin{equation}\label{defh}
h_n(j, x, t):=\frac{j-1}{\sqrt{\alpha_n}}+\widetilde{a}_n + \widetilde{b}_n x+\sqrt{\alpha_n} t.
\end{equation}
Note that when \(\alpha_n\gg 1,\) we have \(h_n(j, x, t)\gg 1\) just as \(u_n(j, x)\gg 1,\) which is the key to the estimates for \({M}^{(n)}_{j, j}(x).\)  
Therefore, following the same line of arguments as in Lemmas \ref{H} and \ref{333}, we similarly obtain the following lemma.
\begin{lem}\label{gnmxt}
	Let  $h_n(j, x, t)$ be defined as in \eqref{defh}. For $|x|\lesssim \log\log \alpha_n$ and $1\leq j\leq t_n,$ the following estimates hold.
			\begin{enumerate}
			\item
	If  $t> \alpha_n^{-\frac{9}{20}},$  then
	$$\aligned 
	g_n(j, x, t)\lesssim
	\frac{e^{-\frac{3}{8} h^{2}_n(j, x, t)}}{h_n(j, x, t)}+(nk_n)^{-1/2}h_n^{-3}(j, x, t)+n^{-4/5}. 
	\endaligned $$
	\item If $0<t\leq \alpha_n^{-\frac{9}{20}}$ and $1\leq j\leq j_n,$ then 
	\begin{equation}\label{z}
		\aligned 
	\frac{1+O(h_n^{-2}(j, x,t))}{\sqrt{2\pi}h_n(j, x,t)}e^{-\frac{h^{2}_n(j, x,t)}{2}}\leq	g_n(j,x,t)
		\leq \frac{1+O(h_n^{-2}(j, x,t))}{\sqrt{2\pi}h_n(j, x,t)}e^{-\frac{h^{2}_n(j, x,t)}{2}}+n^{-4/5}. 
		\endaligned 
	\end{equation}
	\item If $0<t\leq \alpha_n^{-\frac{9}{20}}$ and $j_n< j\leq t_n,$ then
	$$g_n(j,x,t)
	\lesssim \frac{1}{h_n(j, x, t)}e^{-\frac{3h^{2}_n(j, x,t)}{8}}+n^{-4/5}.  $$
	\item	In particular, when $j=t_n$, we have $$g_n(j,x,t)\leq 	\exp\big\{-\frac{(j-1)^2}{4\alpha_n}-\frac{(j-1)(\widetilde{a}_n+\widetilde{b}_n x)}{2\sqrt{\alpha_n}} -\frac{(j-1)t}{4}\big\}, \quad \text{for} \quad t\geq 0.$$
\end{enumerate}
\end{lem}

To capture the asymptotic behavior of \(\widetilde{M}^{(n)}_{j, j}(x),\) we examine the upper bounds of \(g_n(j, x, t)\) and \eqref{widecn}. This leads to the following specific integral, whose proof is deferred to the appendix.

\begin{lem}\label{int}
	Given \(h, v\) satisfying \(1\ll h\) and \(\frac{4}{h}\log h\le v,\) we have
	$$ \int_{0}^{v}\frac{1}{\sqrt{s}(h+s)}e^{-\frac{(h+s)^{2}}{2}}ds=\sqrt{\frac{\pi}{h^{3}}}e^{-\frac{h^{2}}{2}}(1+O(h^{-2})).$$
\end{lem}

We now present the asymptotic behavior of \(\widetilde{M}^{(n)}_{j, j}(x).\)
Set 
$$\widetilde{u}_n(j, x) = \frac{j-1}{\sqrt{\alpha_n}} + \widetilde{a}_n + \widetilde{b}_n x$$
and choose 
$$\widetilde{\ell}_{1,\infty}(n)=\frac{1}{2}\log\log\alpha_n,\qquad 
\widetilde{\ell}_{2,\infty}(n)=\frac{5}{4}\log \log(\alpha_n+e^{2^{3/5}\pi^{-4/5}}). $$

\begin{lem}\label{tildec}
	Uniformly on $x\in [-\widetilde{\ell}_{1,\infty}(n), \widetilde{\ell}_{2,\infty}(n)],$ we have the following estimates.
	\begin{enumerate}
		\item If $1\leq j\leq j_n,$ then
		\begin{equation}\label{cn1}
			\widetilde{M}^{(n)}_{j, j}(x)=\frac{(1+O((\log\alpha_n)^{-1}))}{\sqrt{2}\pi\alpha_n^{1/4}\tilde{u}^{3/2}_n(j, x)}e^{-\frac{\widetilde{u}^{2}_n(j,x)}{2}}.
		\end{equation} 
		\item If $j_n< j< t_n,$ then
		\begin{equation}\label{cn2}
			\widetilde{M}^{(n)}_{j, j}(x)\lesssim\frac{1}{\alpha_n^{1/4}\widetilde{u}^{3/2}_n(j,x)}e^{-\frac{3\widetilde{u}^{2}_n(j,x)}{8}}+n^{-4/5} .
		\end{equation} 
		\item For the endpoint $j=t_n,$
		\begin{equation}\label{cn3}
			\widetilde{M}^{(n)}_{t_n, t_n}(x)\leq	n^{-4}.
		\end{equation}
	\end{enumerate}

\end{lem}
\begin{proof}
We split the integral in \eqref{widecn} into two parts at the point $\alpha_n^{-9/20}$:
$$\aligned
\widetilde{M}^{(n)}_{j, j}(x)
&=\int_{0}^{\alpha_n^{-9/20}}\frac{1}{\pi\sqrt{e^{t}-1}}g_n(j,x,t)dt+\int_{\alpha_n^{-9/20}}^{+\infty}\frac{1}{\pi\sqrt{e^{t}-1}}g_n(j,x,t)dt\\
&=:\mathrm{J}_1+\mathrm{J}_2.
\endaligned$$
For $1\le j\le j_n$, the dominant contribution to $\widetilde{M}^{(n)}_{j, j}(x)$ comes from the first integral $\mathrm{J}_1$, while the second satisfies $\mathrm{J}_2\ll \mathrm{J}_1 (\log \alpha_n)^{-1}.$

Observe that
$$\frac{1}{\sqrt{e^{t}-1}}=\frac{1+O(\alpha_n^{-9/20})}{\sqrt{t}}, \quad\text{uniformly for }0< t\leq\alpha_n^{-9/20},$$
and
$$ h_n^{-2}(j, x, t)\lesssim \widetilde{u}^{-2}_n(j, x)\lesssim (\log\alpha_n)^{-1} $$
for all $1\le j\le t_n$ and $-\widetilde{\ell}_{1,\infty}(n)\leq x\leq \widetilde{\ell}_{2,\infty}(n).$ 
Lemma \ref{gnmxt} (ii) gives the two-sided bound
\begin{equation}\label{Ige}
	\frac{1+O((\log\alpha_n)^{-1})}{\sqrt{2}\pi^{3/2}}\int_{0}^{\alpha_n^{-\frac{9}{20}}}\frac{e^{-\frac{1}{2}h^{2}_n(j, x,t)}}{\sqrt{t}h_n(j, x,t)}dt
	\leq \mathrm{J}_1
\end{equation}
and
\begin{equation}\label{ILE}
		\mathrm{J}_1\leq 	\frac{1+O((\log\alpha_n)^{-1})}{\sqrt{2}\pi^{3/2}}\int_{0}^{\alpha_n^{-\frac{9}{20}}}\frac{e^{-\frac{1}{2} h^{2}_n(j, x,t)}}{\sqrt{t}h_n(j, x,t)}dt+n^{-4/5}, 
\end{equation}
where the term $n^{-4/5}$ in \eqref{ILE} appears because 
\begin{equation}\label{Densi}\int_{0}^{\alpha_n^{-9/20}}\frac{1}{\pi\sqrt{e^{t}-1}}dt\leq \int_{0}^{+\infty}\frac{1}{\pi\sqrt{e^{t}-1}}dt =1.\end{equation}
We now analyze the common integral in \eqref{Ige} and \eqref{ILE}. Using the substitution $s=\sqrt{\alpha_n}t$, we obtain
$$ \int_{0}^{\alpha_n^{-\frac{9}{20}}}\frac{\exp(-\frac{1}{2}h^{2}_n(j, x,t))}{\sqrt{t}h_n(j, x,t)}dt=\alpha_n^{-1/4}\int_{0}^{\alpha_n^{1/20}}\frac{\exp(-\frac{1}{2}(\widetilde{u}_n(j,x)+s)^{2})}{\sqrt{s}(\widetilde{u}_n(j, x)+s)} ds.$$
The conditions on $j$ and $x$ ensure $\widetilde{u}_n(j, x)\gg 1 $ and $\alpha_n^{1/20}\gg \widetilde{u}_n(j, x),$ which allows us to apply Lemma \ref{int}
with $h=\widetilde{u}_n(j, x)$ and $v=\alpha_n^{1/20}$. Consequently,
\begin{equation}\label{Imain}
		\int_{0}^{\alpha_n^{-\frac{9}{20}}}\frac{\exp(-\frac{1}{2}h^{2}_n(j, x,t))}{\sqrt{t}h_n(j, x,t)}dt=\frac{\sqrt{\pi}(1+O(\widetilde{u}^{-2}_n(j,x)))}{\alpha_n^{1/4}\widetilde{u}^{3/2}_n(j,x)}\exp(-\frac{1}{2}\widetilde{u}_n^2(j, x)).
\end{equation}
Next we compare the magnitude of the main term with the error $n^{-4/5}.$
Since $\widetilde{u}_n(j, x)$ is increasing in both $j$ and $x$ and $\frac{1}{\sqrt{2}}<\frac{5}{7}$, we have 
$$\widetilde{u}_n(j,x)\lesssim\widetilde{u}_n(j_n, \widetilde{\ell}_{2,\infty}(n))=\bigl(\frac{1}{5}+\frac{1}{\sqrt{2}}\bigr)\sqrt{\log \alpha_n}(1+o(1))\leq\frac{32}{35}\sqrt{\log\alpha_n}\leq\sqrt{\frac{10}{11}\log\alpha_n}.  $$ 
Using that $y^{-3/2} e^{-y^2/2}$ is decreasing, we deduce
$$\frac{\exp(-\frac{\widetilde{u}^{2}_n(j,x)}{2})}{\alpha_n^{1/4}\widetilde{u}^{3/2}_n(j,x)} \gtrsim\alpha_n^{-1/4}(\log\alpha_n)^{-3/4}e^{-\frac{5}{11}\log\alpha_n}=(\log\alpha_n)^{-3/4} \alpha_n^{-\frac{31}{44}}\ge(\alpha_n\log\alpha_n)^{-3/4}.$$
A direct comparison yields
 $$
 	\frac{(\alpha_n\log\alpha_n)^{-3/4}}{\alpha_n^{-4/5} \log \alpha_n} =(\alpha_n)^{\frac{1}{20}}(\log\alpha_n)^{-7/4}\gg 1,
 	 $$ 
which, recalling $n=\alpha_n k_n\ge \alpha_n$, implies
\begin{equation}\label{IGG}n^{-4/5}\le \alpha_n^{-4/5}\ll (\log \alpha_n)^{-1} \int_{0}^{\alpha_n^{-\frac{9}{20}}}\frac{\exp(-\frac{1}{2}h^{2}_n(j, x,t))}{\sqrt{t}h_n(j, x,t)}dt. \end{equation}  
Combining \eqref{Ige}, \eqref{ILE}, \eqref{Imain} and \eqref{IGG}, we obtain
$$\mathrm{J}_1= \frac{(1+O((\log\alpha_n)^{-1}))}{\sqrt{2}\pi\alpha_n^{1/4}\widetilde{u}^{3/2}_n(j,x)}\exp\bigl(-\frac{1}{2}\widetilde{u}^{2}_n(j, x)\bigr).$$
We now prove $\mathrm{J}_2\ll\mathrm{J}_1\,  (\log \alpha_n)^{-1} $.
Lemma \ref{gnmxt} gives
$$\mathrm{J}_2\lesssim\int_{\alpha_n^{-9/20}}^{+\infty}\frac{1}{\sqrt{e^{t}-1}}\bigl(	\frac{e^{-\frac{3}{8}h^{2}_n(j, x,t)}}{h_n(j,x,t)}+(nk_n)^{-1/2} h_n^{-3}(j,x,t)+n^{-4/5}\bigr) dt.$$
For $t\ge \alpha_n^{-\frac{9}{20}}$, we have $\sqrt{e^t-1}>\sqrt{t}\ge \alpha_n^{-\frac{9}{40}}$, and $h_n(j, x, t)>\sqrt{\alpha_n} \,t$. Hence,
$$\aligned \mathrm{J}_2&\lesssim \alpha_n^{9/40}\int_{\alpha_n^{-9/20}}^{+\infty}\frac{\exp(-\frac{3}{8}h^{2}_n(j, x,t))}{h_n(j,x,t)}dt+\alpha_n^{-2}\int_{\alpha_n^{-9/20}}^{+\infty}t^{-\frac{7}{2}}dt+n^{-4/5}.
\endaligned$$
For the first integral, set $y=h_n(j,x,t)=\widetilde{u}_n(j,x)+\sqrt{\alpha_n}\, t $. Then
$$\aligned \alpha_n^{9/40}\int_{\alpha_n^{-9/20}}^{+\infty}\frac{\exp(-\frac{3}{8}h^{2}_n(j, x,t))}{h_n(j,x,t)}dt&=\alpha_n^{-\frac{11}{40}}\int_{\tilde{u}_n(j,x)+\alpha_n^{1/20}}^{+\infty}y^{-1}e^{-\frac{3}{8}y^2}dy\\
&\lesssim\alpha_n^{-\frac{11}{40}}\frac{\exp\bigl(-\frac{3}{8}(\widetilde{u}_n(j,x)+\alpha_n^{1/20})^{2}\bigr)}{(\widetilde{u}_n(j,x)+\alpha_n^{1/20})^{2}}\\
&\le\alpha_n^{-3/8}\exp\Bigl(-\frac{3}{8}\alpha_n^{1/10}\Bigr),
\endaligned$$
where the last inequality uses $\widetilde{u}_n(j,x)>0$ for $-\widetilde{\ell}_{1,\infty}(n)\leq x\leq \widetilde{\ell}_{2,\infty}(n).$ 
Moreover,
$$\alpha_n^{-2}\int_{\alpha_n^{-9/20}}^{+\infty}t^{-\frac{7}{2}}dt\lesssim \alpha_n^{-2}(\alpha_n)^{\frac{9}{8}}=\alpha_n^{-7/8}.$$
Thus,
\begin{equation}\label{IIup}\mathrm{J}_2\le \alpha_n^{-3/8}\exp\Bigl(-\frac{3}{8}\alpha_n^{1/10}\Bigr)+\alpha_n^{-7/8}\lesssim \alpha_n^{-4/5},\end{equation}
and \eqref{IGG} again implies 
$$\mathrm{J}_2\ll\mathrm{J}_1(\log \alpha_n)^{-1}.$$
Consequently,	 
$$\widetilde{M}^{(n)}_{j, j}(x)=\frac{(1+O((\log\alpha_n)^{-1}))}{\sqrt{2}\pi\alpha_n^{1/4}\widetilde{u}^{3/2}_n(j,x)}\exp\bigl(-\frac{1}{2}\widetilde{u}^{2}_n(j, x)\bigr)$$ 
uniformly for $1\leq j\leq j_n$ and $-\widetilde{\ell}_{1,\infty}(n)\leq x\leq \widetilde{\ell}_{2,\infty}(n).$
Examining the proof of \eqref{cn1}, we see that \eqref{ILE}, \eqref{Imain} and \eqref{IIup} still hold for $j_n<j< t_n$, from which \eqref{cn2} follows. For the special case $j=t_n=[8\sqrt{\alpha_n\log n}]$,
Lemma \ref{gnmxt} together with the equality in \eqref{Densi} and the monotonicity of $g_n(j,x,t)$ in $t$ gives
$$\aligned \widetilde{M}^{(n)}_{j, j}(x)&\leq \int_{0}^{+\infty}\frac{1}{\pi\sqrt{e^{t}-1}} \exp\Bigl\{-\frac{(j-1)^2}{4\alpha_n}-\frac{(j-1)(\widetilde{a}_n+\widetilde{b}_n x)}{2\sqrt{\alpha_n}}\Bigr\}dt\\
&\leq \exp(-4\log n)\int_0^{+\infty}\frac{1}{\pi\sqrt{e^{t}-1}} dt =n^{-4}. 
\endaligned $$ 
This completes the proof.
\end{proof}  

Now, we are able to give the proofs of Theorems  for $\alpha=+\infty.$
\subsection{Proof of Theorem \ref{main} for $\alpha=+\infty$}

 We take $ \ell_{1,\infty}(n)=\frac{1}{2}\log\log \alpha_n$ and $\ell_{2,\infty}(n)=\log\left(\sqrt{2\pi}\log (\alpha_n+e^{\frac{1}{\sqrt{2\pi}}})\right).$

 	First, we explore the exact asymptotical expression of $\mathbb{P}(X_n\leq x),$ which is expressed as   
 	$$\mathbb{P}(X_n\leq x)=\prod_{j=1}^n(1-M^{(n)}_{j, j}(x))=\exp(\beta_n(x)), $$
 	where 
 	$$\beta_n(x):=\sum_{j=1}^{n}\log(1-M^{(n)}_{j, j}(x)).$$
 	The monotonicity of $M^{(n)}_{j, j}(x)$ in both $j$ and $x$ implies that
 	$$M^{(n)}_{j, j}(x)\leq M^{(n)}_{1, 1}(-\ell_{1,\infty}(n)),$$
 	uniformly
 	on  $ x\geq-\ell_{1,\infty}(n)$ and $ j\geq1.$ By definition, 
 	$$u_n(1, -\ell_{1, \infty}(n))=\sqrt{\log(\alpha_n+1)}-\frac{\ell_{1, \infty}(n)+\ell_{2, \infty}(n)}{\sqrt{\log(\alpha_n+e)}},$$
 	 and then Lemma \ref{333} entails 
 	$$M^{(n)}_{1, 1}(-\ell_{1,\infty}(n))=O(\alpha_n^{-1/2}\log\alpha_n).$$
 	Applying the Taylor expansion $\log(1 - t) = -t (1 + O(t))$ for sufficiently small $|t|,$ we get
 	 $$\beta_n(x)=-\sum_{j=1}^{n}M^{(n)}_{j, j}(x)(1+O(\alpha_n^{-1/2}\log\alpha_n))=-{\rm Tr}(M^{(n)}(x))(1+O(\alpha_n^{-1/2}\log\alpha_n)).$$

Recall $j_n=[\frac{1}{5}\sqrt{\alpha_n\log\alpha_n}]$ and $t_n=[8\sqrt{\alpha_n\log n}].$ From the monotonicity of $M^{(n)}_{j, j}(x)$ with respect to $ j $, we obtain
 	 \begin{equation}\label{32}
 	 	\sum_{j=1}^{j_n}M^{(n)}_{j, j}(x)\leq {\rm Tr}(M^{(n)}(x))\leq\sum_{j=1}^{t_n}M^{(n)}_{j, j}(x)+nM^{(n)}_{t_n, t_n}(x).
 	 \end{equation} 
 	 Lemma \ref{H} entails that 
 	 \begin{equation}\label{33}
 	 	nM^{(n)}_{t_n, t_n}(x)\leq  n\exp\{-4\log n-2\sqrt{\log n}(a_n+b_nx)\}=o(n^{-3}),
 	 \end{equation}
 	 uniformly
 	 on  $-\ell_{1,\infty}(n)\leq x\leq\ell_{2,\infty}(n).$ Leveraging Lemmas  \ref{sum} and \ref{333}, we get
 	 $$	\aligned\sum_{j=j_n+1}^{t_n}M^{(n)}_{j, j}(x)&\leq\sum_{j=j_n+1}^{t_n}(\frac{1}{u_n(j,x)}e^{-\frac{3u^{2}_n(j, x)}{8}}+n^{-4/5}k_n^{-1/5})\\
 	 &\lesssim\frac{\sqrt{\alpha_n}}{ u_{n}^{2}(j_n+1, -\ell_{1,\infty}(n))}\exp(- \frac{3u_{n}^{2}(j_n+1, -\ell_{1,\infty}(n))}{8})+n^{-3/10}\sqrt{\log n}.	\endaligned$$
 	Note that
 	 $$u_{n}(j_n+1, -\ell_{1,\infty}(n))=\frac{6}{5}\sqrt{\log \alpha_n}+o(1)$$  and $$u_{n}^{2}(j_n+1, -\ell_{1,\infty}(n))=\frac{36}{25}\log \alpha_n-\frac{12}{5}(\ell_{1, \infty}(n)+\ell_{2, \infty}(n))+O(1).$$
Thereby,
 	 \begin{equation}\label{34}
 	 	\sum_{j=j_n+1}^{t_n}M^{(n)}_{j, j}(x)\lesssim \alpha_n^{-\frac{1}{25}}(\log \alpha_n)^{\frac{7}{20}}.
 	 \end{equation}
 	 For $ 1\leq j\leq j_n,$ Lemmas \ref{sum} and \ref{333} immediately imply that
 	 \begin{equation}\label{35}
 	 	\sum_{j=1}^{j_n}M^{(n)}_{j, j}(x) =\sum_{j=1}^{j_n}\frac{1+O((\log \alpha_n)^{-1})}{\sqrt{2\pi}u_n(j, x)}e^{-\frac{u_n^{2}(j, x)}{2}}\\
 	 	=\frac{\sqrt{\alpha_n}(1+O((\log \alpha_n)^{-1}))}{\sqrt{2\pi } u_{n}^{2}(1, x)}e^{- \frac{u_{n}^{2}(1, x)}{2}},
 	 \end{equation}
 	 where for the last equality we use the fact$$\frac{1}{u_n^2(j_n+1, x)} e^{-\frac{u_n^2(j_n+1, x)}{2}} \ll \frac{1}{ u_n^2(1, x)\log \alpha_n} e^{-\frac{u_n^2(1, x)}{2}}.$$ 
 	This inequality is true because $$u_n(j_n+1, x)=\frac{6}{5}\sqrt{\log \alpha_n}(1+o(1))\quad\text{and}\quad u_n(1, x)=\sqrt{\log \alpha_n}(1+o(1)).$$
 	Furthermore,  
 	\begin{equation}\label{e4}
 	\aligned u_n(1,x)&=\sqrt{\log (\alpha_n+1)}+ \frac{x-\ell_{2,\infty}(n)}{\sqrt{\log (\alpha_n+e)}}\\
 	&=\sqrt{\log (\alpha_n+e)}+\frac{x-\ell_{2,\infty}(n)}{\sqrt{\log (\alpha_n+e)}}+O((\log \alpha_n)^{-1/2}(\alpha_n)^{-1}) , \endaligned 
 	\end{equation}
 	 and then 
 	 \begin{equation}\label{e3} \aligned 
 	 	u_n^2(1, x)&=\log (\alpha_n+e)-2\ell_{2,\infty}(n)+2 x+\frac{(x-\ell_{2,\infty}(n))^{2}}{\log (\alpha_n+e)}+O(\alpha_n^{-1});
 	 	\\
 	 	\exp(-\frac{1}{2}u_n^2(1, x))&=\sqrt{\frac{2\pi }{\alpha_n+e}}\log (\alpha_n+e^{\frac1{\sqrt{2\pi}}})\exp(-x-\frac{(\ell_{2,\infty}(n)-x)^{2}}{2\log(\alpha_n+e)})e^{O(\alpha_n^{-1})}\\
 	 	&=\sqrt{\frac{2\pi }{\alpha_n}}\log (\alpha_n)\exp(-x-\frac{(\ell_{2,\infty}(n)-x)^{2}}{2\log(\alpha_n+e)})(1+O(\alpha_n^{-1})).
 	 	\endaligned
 	 \end{equation}
 	 Putting \eqref{33}, \eqref{34}  \eqref{35} and \eqref{e3} back into \eqref{32}, we see
 	$${\rm Tr}(M^{(n)}(x))=\frac{1+O((\log  \alpha_n)^{-1})}{(1+\frac{x-\ell_{2,\infty}(n)}{\log (\alpha_n+e)})^{2}}\exp(-x-\frac{(x-\ell_{2,\infty}(n))^2}{2\log  (\alpha_n+e)})	
 	$$
 	and then 
 	 \begin{equation}\label{betanx}
 	 \beta_n(x)=-\frac{1+O((\log  \alpha_n)^{-1})}{(1+\frac{x-\ell_{2,\infty}(n)}{\log (\alpha_n+e)})^{2}}\exp(-x-\frac{(x-\ell_{2,\infty}(n))^2}{2\log  (\alpha_n+e)}).	
 	 \end{equation}
The expression \eqref{betanx} guarantees that 
$$\aligned |\beta_n(x)+e^{-x}|&=e^{-x}|1-\frac{1+O((\log \alpha_n)^{-1})}{(1+\frac{x-\ell_{2,\infty}(n)}{\log (\alpha_n+e)})^{2}} \exp(-\frac{(x-\ell_{2,\infty}(n))^2}{2\log  (\alpha_n+e)})|\\
&=\frac{e^{-x}(1+o(1))}{(1+\frac{x-\ell_{2,\infty}(n)}{\log (\alpha_n+e)})^{2}}|\frac{(x-\ell_{2,\infty}(n))^2}{2\log(\alpha_n+e)}+\frac{2(x-\ell_{2,\infty}(n))}{\log(\alpha_n+e)}|. \endaligned $$
The choices of $\ell_{i, \infty}(n)$ (i=1,2)  are designed to ensure $|\beta_n(x)+e^{-x}|=o(1),$ whence 
 	 \begin{equation}\label{aphagg1}\aligned |\mathbb{P}(X_n\leq x)-e^{-e^{-x}}|&=e^{-e^{-x}}|\exp(\beta_n(x)+e^{-x})-1| \\
 	 &=e^{-e^{-x}}(1+o(1))|\beta_n(x)+e^{-x}|\\
 	 &=e^{-x-e^{-x}} \frac{|(x-\ell_{2,\infty}(n))^2+4(x-\ell_{2,\infty}(n))|}{2\log(\alpha_n+e)}(1+o(1))
 	 \endaligned \end{equation}
 	 uniformly on $[-\ell_{1, \infty}(n), \ell_{2, \infty}(n)].$
We see clearly  $$\sup_{x\in\mathbb{R}} e^{-x-e^{-x}} |x^k|<+\infty, \quad k=1, 2$$ and $\sup_{x\in\mathbb{R}} e^{-x-e^{-x}} =e^{-1},$ which together with \eqref{aphagg1}, imply that \begin{equation}\label{midd}
\sup_{x\in [-\ell_{1, \infty}(n), \ell_{2, \infty}(n)]}|\mathbb{P}(X_n\leq x)-e^{-e^{-x}}|=\frac{(\log\log \alpha_n)^2}{2e\log\alpha_n}(1+o(1)).
\end{equation}
While for the two side intervals, we have 
	$$\aligned \sup\limits_{x\in(-\infty, -\ell_{1,\infty}(n)]}|\mathbb{P}(X_n\leq x)-e^{-e^{-x}}|&\leq\mathbb{P}(X_n\leq -\ell_{1,\infty}(n))+e^{-e^{\ell_{1,\infty}(n)}}\lesssim e^{-e^{\ell_{1,\infty}(n)}};\endaligned $$ 
	and analogously $$\sup\limits_{x\in[\ell_{2,\infty}(n),+\infty)}|\mathbb{P}(X_n\leq x)-e^{-e^{-x}}|\lesssim1-e^{-e^{-\ell_{2,\infty}(n)}}.  $$
	Thereby, 
	\begin{equation}\label{twoside}\sup\limits_{x\in(-\infty, -\ell_{1,\infty}(n)] \cup [\ell_{2,\infty}(n),+\infty) }|\mathbb{P}(X_n\leq x)-e^{-e^{-x}}|\lesssim \frac{1}{\log \alpha_n}.
	\end{equation}
Combining \eqref{midd} and \eqref{twoside}, we derive 
	 	$$\sup_{x\in \mathbb{R}}|\mathbb{P}(X_n\le x)-e^{-e^{-x}}|= \frac{(\log \log \alpha_n)^{2}}{2e\log \alpha_n}(1+o(1)).$$
	 	The proof of Theorem \ref{main} for $\alpha=+\infty$
is completed. 	 	

\subsection{Proof of Theorem \ref{thmrealpart} for $\alpha=+\infty$} By the classical inequality between the trace and the determinant for the matrix (see (7.11) in \cite{Gohberg}), we have  
\begin{equation}\label{E2}
\begin{aligned}
&\bigl| \det\bigl({\rm I}_n - \widetilde{M}^{(n)}(x)\bigr) - \exp\bigl(-\operatorname{Tr}(\widetilde{M}^{(n)}(x))\bigr) |\lesssim \|\widetilde{M}^{(n)}(x)\|_{\rm HS} \; 
        \exp\bigl\{ - \operatorname{Tr}(\widetilde{M}^{(n)}(x)) \bigr\}
\end{aligned}
\end{equation}
once $\|\widetilde{M}^{(n)}(x)\|_{\rm HS}<1.$ 

We first state a lemma on $\|\widetilde{M}^{(n)}(x)\|_{\rm HS},$ whose proof is postponed to the Appendix.  

\begin{lem}\label{Mhslem}
For the matrix $\widetilde{M}^{(n)}(x),$ we have 
\begin{equation}\label{condionM}
	\|\widetilde{M}^{(n)}(x)\|_{\rm HS}\ll  e^{-x} \frac{(\log\log \alpha_n)^2}{\log\alpha_n}
\end{equation}
uniformly for $-\widetilde{\ell}_{1,\infty}(n)\leq x\leq \widetilde{\ell}_{2,\infty}(n).$
\end{lem}

Examining the expressions from \eqref{32} to \eqref{35}, while using estimates for $\widetilde{M}^{(n)}_{j, j}(x)$ in Lemma \ref{tildec}, we have 
\begin{equation}\label{finalreal}\aligned \operatorname{Tr}(\widetilde{M}^{(n)}(x))&= \sum_{j=1}^{n}\widetilde{M}^{(n)}_{j, j}(x)=(1+O((\log\alpha_n)^{-1}))\sum_{j=1}^{j_n}\widetilde{M}^{(n)}_{j, j}(x)\\
&=(1+O((\log\alpha_n)^{-1}))\sum_{j=1}^{j_n}\frac{e^{-\frac{\widetilde{u}^{2}_n(j,x)}{2}}}{\sqrt{2}\pi\alpha_n^{1/4}\tilde{u}^{3/2}_n(j,x)}\\
&=\frac{\alpha_n^{1/4}(1+O((\log\alpha_n)^{-1}))}{\sqrt{2}\pi\widetilde{u}^{5/2}_n(1, x)}\exp(-\frac{1}{2}\widetilde{u}^{2}_n(1, x)). \endaligned \end{equation}
Review $\widetilde{\ell}_{2,\infty}(n)=\frac{5}{4}\log \log(\alpha_n+e^{2^{3/5}\pi^{-4/5}}), \;\widetilde{b}_n = \frac{\sqrt{2}}{\sqrt{\log(\alpha_n+e^2) }}$ and  $$
\widetilde{a}_n = \sqrt{\frac{\log(\alpha_n+1) }{2}} - \frac{\sqrt{2}(\log (2^{-\frac{3}{4}}\pi)+\frac{5}{4}\log \log(\alpha_n+e^{2^{3/5}\pi^{-4/5}}))}{\sqrt{\log(\alpha_n+e^2) }}
. $$
Now 
\begin{equation}\label{widetildeu}\aligned \widetilde{u}^{2}_n(1,x)&=(\widetilde{a}_n +\widetilde{b}_n x)^2\\
&=\big(\sqrt{\frac{\log(\alpha_n+1) }{2}}-\frac{\sqrt{2}(\log (2^{-\frac{3}{4}}\pi)+\widetilde{\ell}_{2,\infty}(n)-x)}{\sqrt{\log(\alpha_n+e^2) }}\big)^2\\
&=\frac{\log \alpha_n}{2}-2(\log (2^{-3/4}\pi)+\widetilde{\ell}_{2,\infty}(n))+2x+\frac{2(\widetilde{\ell}_{2,\infty}(n)-x)^{2}}{\log\alpha_n}+O(\frac{\log\log\alpha_n}{\log\alpha_n})\endaligned  \end{equation}
and
\begin{equation}\label{widetildeuu} \widetilde{u}^{5/2}_n(1,x)=2^{-5/4}(\log\alpha_n)^{5/4}(1+O(\frac{\log\log\alpha_n}{\log\alpha_n}))\end{equation}  uniformly on $-\widetilde{\ell}_{1,\infty}(n)\leq x\leq \widetilde{\ell}_{2,\infty}(n). $
Putting these two asymptotics back into \eqref{finalreal}, we have 
\begin{equation}\label{Trwidefinal}
	\operatorname{Tr}(\widetilde{M}^{(n)}(x))=\exp\{-x-\frac{(\widetilde{\ell}_{2,\infty}(n)-x)^{2}}{\log\alpha_n}\}(1+O(\frac{\log\log\alpha_n}{\log\alpha_n})) \end{equation}
 and then similarly as \eqref{aphagg1} 
 \begin{equation}\label{Trwidedif}
	|\exp(-\operatorname{Tr}(\widetilde{M}^{(n)}(x)))-\exp(-e^{-x})|=\exp(-e^{-x}-x)\frac{(\widetilde{\ell}_{2, \infty}(n)-x)^{2}}{\log\alpha_n}(1+o(1)). \end{equation} 
The asymptotic \eqref{Trwidefinal} and the fact $(e^{-x}\vee 1)(\widetilde{\ell}_{2, \infty}(n)-x)^{2}\ll \log\alpha_n$ also imply
$$\aligned \operatorname{Tr}(\widetilde{M}^{(n)}(x)))&=\exp(-x)(1+O(\frac{(\widetilde{\ell}_{2,\infty}(n)-x)^{2}}{\log\alpha_n}))\\
&=\exp(-x)+O(e^{-x}(\frac{\widetilde{\ell}_{2,\infty}(n)-x)^{2}}{\log\alpha_n})) =\exp(-x)+o(1),\endaligned $$
which helps us to write 
 \begin{equation}\label{Trwideexp}
	\exp(-\operatorname{Tr}(\widetilde{M}^{(n)}(x)))=\exp(-e^{-x})(1+o(1))\asymp \exp(-e^{-x}).\end{equation} 	
	The inequality \eqref{E2} and the expressions \eqref{Trwidedif} and \eqref{Trwideexp} yield that 
	$$\bigl| \det\bigl({\rm I}_n - \widetilde{M}^{(n)}(x)\bigr) - \exp(-e^{-x})\bigr|$$ lies in the interval 
	$$\big[|\exp(-\operatorname{Tr}(\widetilde{M}^{(n)}(x)))-\exp(-e^{-x})| \pm \|\widetilde{M}^{(n)}(x)\|_{\rm HS} \exp(-e^{-x})\big].$$ 
Examining the expressions \eqref{Trwidedif} and the upper bound \eqref{condionM}, we derive  
$$\bigl| \det\bigl({\rm I}_n -\widetilde{M}^{(n)}(x)\bigr) - \exp(-e^{-x})\bigr|=|\exp(-\operatorname{Tr}(\widetilde{M}^{(n)}(x)))-\exp(-e^{-x})|(1+o(1))$$ 
uniformly on $-\widetilde{\ell}_{1,\infty}(n)\leq x\leq \widetilde{\ell}_{2,\infty}(n).$ The proof is then completed. 
Following the same line of reasoning as \eqref{aphagg1}, \eqref{midd} and \eqref{twoside}, we derive 
$$\aligned
	\sup_{x\in \mathbb{R}}|\mathbb{P}(\widetilde{X}_n\le x)-\exp(-e^{-x})|=\frac{\widetilde{\ell}_{2,\infty}^{2}(n)}{e\log\alpha_n}(1+o(1))=\frac{25(\log\log\alpha_n)^{2}}{16e\log\alpha_n}(1+o(1)).\endaligned $$
This finishes the proof of Theorem \ref{thmrealpart} for $\alpha=+\infty.$

 \section{Proof of the Theorems for $ \alpha=0$}\label{sec:alphazero}
  In this section, we address the case \(\alpha = 0\). Briefly, when $\alpha=0,$ we are able to prove two things: 
  
  \begin{enumerate}
  \item $M^{(n)}_{j, j}(x)$ is small enough for $2\le j\le n$ such that
  $$\mathbb{P}(X_n\le x)=(1-M^{(n)}_{1, 1}(x))(1+o(1)).$$
  \item The matrix \(\widetilde{M}^{(n)}(x)\) satisfies  
\[
\sum_{j<k}\frac{\bigl(\widetilde{M}_{j,k}^{(n)}(x)\bigr)^2}{\bigl(1-\widetilde{M}_{j,j}^{(n)}(x)\bigr)\bigl(1-\widetilde{M}_{k,k}^{(n)}(x)\bigr)}\ll 1,
\]  
and the diagonal entries \(\widetilde{M}_{j,j}^{(n)}(x)\) are sufficiently small for \(2\le j\le n\). These two properties together yield
  $$\mathbb{P}(\widetilde{X}_n\le x)=(1-\widetilde{M}^{(n)}_{1, 1}(x))(1+o(1)).$$
	
  \end{enumerate}

Set \(z_n = \alpha_n^{-1/2} \wedge n.\) We are going to obtain precise asymtotics for \({M}^{(n)}_{j, j}(x)\) and $\widetilde{M}^{(n)}_{j, j}(x)$ uniformly on $|x|\le \sqrt{2\log z_n}.$

\subsection{Estimates on \({M}^{(n)}_{j, j}(x)\) and $\widetilde{M}^{(n)}_{j, j}(x)$} 
We begin this subsection with the following lemma on \({M}^{(n)}_{1, 1}(x),\) obtained by applying the Edgeworth expansion.
\begin{lem}\label{ed2} Let $\Phi$ and $\phi$ be the distribution function and the density function of standard normal, respectively. 
    	We have 
  	$${M}^{(n)}_{1, 1}(x)=1-\Phi(x)-(\sqrt{\alpha_n}-\frac{x}{4n})\phi(x)+O(z_n^{-19/10})$$ uniformly on $|x|\le \sqrt{2\log z_n}.$
  	    \end{lem}
    \begin{proof}
   Note that
\[
\begin{aligned}
{M}^{(n)}_{1, 1}(x) &= \mathbb{P}\left( \log Y_n > k_n \psi(n) + (k_n / n)^{1/2} (a_n + b_n x) \right) \\
&= \mathbb{P}( \frac{\log Y_n - k_n \psi(n)}{\sqrt{k_n \psi'(n)}} > (n \psi'(n))^{-1/2} (a_n + b_n x) ).
\end{aligned}
\]
Here, review 
\[
a_n = \sqrt{\log(\alpha_n + 1)} - \frac{\log\big(\sqrt{2\pi} \log\big(\alpha_n + e^{1/\sqrt{2\pi}}\big)\big)}{\sqrt{\log(\alpha_n + e)}}, \quad 
b_n = \frac{1}{\sqrt{\log(\alpha_n + e)}}.
\]
Under the condition \(\alpha = 0\), the parameters \(a_n\) and \(b_n\) satisfy the asymptotic relations
\begin{equation}\label{alpha0}
a_n = \sqrt{\alpha_n} + O(\alpha_n), \quad \text{and} \quad b_n = 1 + O(\alpha_n).
\end{equation}
Lemma \ref{diagammapro} leads
\[
\begin{aligned}
(n \psi'(n))^{-1/2} (a_n + b_n x) 
&= (1 - \frac{1}{4n} + O(n^{-2})) \left( \sqrt{\alpha_n} + x + O(\alpha_n (|x| + 1)) \right) \\
&= x + \sqrt{\alpha_n} - \frac{x}{4n} + O(z_n^{-19/10}).
\end{aligned}
\]
Define
\[
y_n(x) = \sqrt{\alpha_n} - \frac{x}{4n} + O(z_n^{-19/10}) = O(z_n^{-1}) \quad \text{and} \quad t_n(x) = x + y_n(x).
\]
It follows from Lemma \ref{ed} that 
\[
1 -{M}^{(n)}_{1, 1}(x) = \Phi(t_n(x)) - \frac{1 - t_n^2(x)}{6\sqrt{k_n n}} \phi(t_n(x)) + O\left(k_n^{-3/2} + k_n^{-1/2} n^{-3/2}\right).
\]
Since \(x^2 \phi(x)\) is bounded, and noting that
\[
\frac{1}{\sqrt{k_n n}} = \sqrt{\alpha_n} n^{-1} \lesssim z_n^{-2}, \quad k_n^{-3/2} + k_n^{-1/2} n^{-3/2} = (\alpha_n^{-1} n)^{-3/2} + \sqrt{\alpha_n} n^{-2} \lesssim z_n^{-3},
\]
we conclude that
\[
\frac{1 - t_n^2(x)}{6\sqrt{k_n n}} \phi(t_n(x)) + O\left(k_n^{-3/2} + k_n^{-1/2} n^{-3/2}\right) = O(z_n^{-2}).
\]
 On expanding \(\Phi(t_n(x))\) in a Taylor series about \(x\) (valid for \(|x|\le \sqrt{2\log z_n}\)), we obtain

\begin{equation}\label{Phitn}
\Phi(t_n(x)) = \Phi(x) + \phi(x) \left( y_n(x) + O(y_n^2(x) |x|) \right) = \Phi(x) + ( \sqrt{\alpha_n} - \frac{x}{4n} ) \phi(x) + O(z_n^{-19/10}).
\end{equation}
Therefore,
\[
{M}^{(n)}_{1, 1}(x)= 1 - \Phi(x) - ( \sqrt{\alpha_n} - \frac{x}{4n} ) \phi(x) + O(z_n^{-19/10}).
\] 
The proof is completed. 
    \end{proof}

Next, we present a similar result on $\widetilde{M}^{(n)}_{1, 1}(x).$  
    \begin{lem}\label{m0cn} For $|x|\leq \sqrt{2\log z_n}, $ we have
	$$\widetilde{M}^{(n)}_{1, 1}(x)=1-\Phi(x)-((\frac{\sqrt{2}+4\ln2}{2})\sqrt{\alpha_n}-\frac{x}{4n})\phi(x)+O(z_n^{-19/10})$$ 
	as $n\to+\infty.$
\end{lem}
 \begin{proof}
    We recall that
\[
\begin{aligned}
g_n(1, x, t) &= \mathbb{P}\Bigl(\log Y_{n}\ge k_n\psi(n)+\frac{\widetilde{a}_n+\widetilde{b}_n x}{\sqrt{\alpha_n}}+t\Bigr) \\
&= \mathbb{P}\Bigl( \frac{\log Y_n - k_n \psi(n)}{\sqrt{k_n \psi'(n)}} > (n \psi'(n))^{-1/2} (\widetilde{a}_n + \widetilde{b}_n x+\sqrt{\alpha_n}t) \Bigr),
\end{aligned}
\]
with
\[
\widetilde{b}_n = \frac{\sqrt{2}}{\sqrt{\log(\alpha_n +e^{2})}},\qquad 
\widetilde{a}_n = \sqrt{\frac{\log(\alpha_n+1) }{2}} - \frac{\sqrt{2}\bigl(\log (2^{-\frac{3}{4}}\pi)+\frac{5}{4}\log \log(\alpha_n+e^{2^{3/5}\pi^{-4/5}})\bigr)}{\sqrt{\log(\alpha_n +e^{2})}}.
\]
These parameters admit the expansions
\[
\widetilde{a}_n = \sqrt{\frac{\alpha_n}{2}} + O(\alpha_n), \qquad 
\widetilde{b}_n = 1 + O(\alpha_n).
\]
By Lemma \ref{diagammapro}, we have
\[
\begin{aligned}
 (n \psi'(n))^{-1/2} (\widetilde{a}_n + \widetilde{b}_n x+\sqrt{\alpha_n}t) 
&= \bigl(1 - \frac{1}{4n} + O(n^{-2})\bigr) \bigl( \sqrt{\frac{\alpha_n}{2}} + x +\sqrt{\alpha_n}t + O(\alpha_n (|x|+1)) \bigr) \\
&= \sqrt{\frac{\alpha_n}{2}} + x +\sqrt{\alpha_n}t - \frac{x}{4n} + O(z_n^{-19/10}+z_n^{-2}t),
\end{aligned}
\]
uniformly for \(|x|\le \sqrt{2\log z_n}\). For \(0\leq t\leq \alpha_n^{-1/4}\), we have
\[
\sqrt{\alpha_n}t \leq \alpha_n^{1/4} \ll 1,\qquad z_n^{-2}t \leq z_n^{-3/2}.
\]
Applying Lemma \ref{ed} and an argument analogous to Lemma \ref{ed2}, we rewrite \(y_n\) in \eqref{Phitn} as
\[
y_n = \sqrt{\frac{\alpha_n}{2}} + \sqrt{\alpha_n}\, t - \frac{x}{4n} + O(z_n^{-3/2}),
\]
which yields
\[
g_n(1,x,t) = 1 - \Phi(x) - \Bigl( \sqrt{\frac{\alpha_n}{2}} + \sqrt{\alpha_n}\, t - \frac{x}{4n} \Bigr) \phi(x) + O(z_n^{-3/2}).
\]
Substituting this into \eqref{widecn} gives
\[
\widetilde{M}^{(n)}_{1, 1}(x) = \Bigl( \int_{0}^{\alpha_n^{-1/4}} + \int_{\alpha_n^{-1/4}}^{+\infty} \Bigr) \frac{g_n(1,x,t)}{\pi \sqrt{e^t - 1}} \, dt.
\]

Since \(g_n(1,x,t) \le 1\), the tail integral is estimated via the change of variable \(t = -2\log u\):
\[
\begin{aligned}
\int_{\alpha_n^{-1/4}}^{+\infty} \frac{g_n(1,x,t)}{\pi \sqrt{e^t-1}} \, dt
&\leq \int_{\alpha_n^{-1/4}}^{+\infty} \frac{1}{\pi \sqrt{e^t-1}} \, dt
= \int_{0}^{e^{-\frac12 \alpha_n^{-1/4}}} \frac{2}{\pi\sqrt{1-u^2}} \, du \\
&= \frac{2}{\pi}\arcsin\!\bigl(e^{-\frac12 \alpha_n^{-1/4}}\bigr)
\lesssim e^{-\frac12 \alpha_n^{-1/4}}.
\end{aligned}
\]

We now focus on the integral over \([0,\alpha_n^{-1/4}]\). Inserting the expression of \(g_n(1,x,t)\) and expanding, we obtain
\[
g_n(1,x,t) = w_n(x) - t\sqrt{\alpha_n}\,\phi(x) + O(z_n^{-3/2}),
\]
where
\[
w_n(x) := 1 - \Phi(x) - \Bigl( \sqrt{\frac{\alpha_n}{2}} - \frac{x}{4n} \Bigr) \phi(x).
\]
Hence,
\[
\begin{aligned}
\int_{0}^{\alpha_n^{-1/4}} \frac{g_n(1,x,t)}{\pi \sqrt{e^t-1}} \, dt
&= \int_{0}^{\alpha_n^{-1/4}} \frac{w_n(x) - t\sqrt{\alpha_n}\,\phi(x)}{\pi \sqrt{e^t-1}} \, dt + O(z_n^{-3/2}) \\
&= w_n(x) \int_{0}^{\alpha_n^{-1/4}} \frac{dt}{\pi \sqrt{e^t-1}}
   - \sqrt{\alpha_n}\,\phi(x) \int_{0}^{\alpha_n^{-1/4}} \frac{t}{\pi \sqrt{e^t-1}} \, dt
   + O(z_n^{-3/2}).
\end{aligned}
\]
From the previous tail estimate,
\[
\int_{0}^{\alpha_n^{-1/4}} \frac{dt}{\pi \sqrt{e^t-1}}
= \Bigl( \int_{0}^{+\infty} - \int_{\alpha_n^{-1/4}}^{+\infty} \Bigr) \frac{dt}{\pi \sqrt{e^t-1}}
= 1 + O\!\bigl(e^{-\frac12 \alpha_n^{-1/4}}\bigr),
\]
and a standard computation gives
\[
\int_{0}^{\alpha_n^{-1/4}} \frac{t}{\pi \sqrt{e^t-1}} \, dt
= \Bigl( \int_{0}^{+\infty} - \int_{\alpha_n^{-1/4}}^{+\infty} \Bigr) \frac{t}{\pi \sqrt{e^t-1}} \, dt
= 2\ln 2 + O\!\bigl(\alpha_n^{-1/4} e^{-\frac12 \alpha_n^{-1/4}}\bigr).
\]
Thus,
\[
\begin{aligned}
\int_{0}^{\alpha_n^{-1/4}} \frac{g_n(1,x,t)}{\pi \sqrt{e^t-1}} \, dt
&= w_n(x)\bigl(1 + O(e^{-\frac12 \alpha_n^{-1/4}})\bigr) \\
&\quad - \sqrt{\alpha_n}\,\phi(x) \bigl( 2\ln 2 + O(\alpha_n^{-1/4} e^{-\frac12 \alpha_n^{-1/4}}) \bigr)
+ O(z_n^{-3/2}).
\end{aligned}
\]
Since \(\alpha_n^{-1/4} e^{-\frac12 \alpha_n^{-1/4}} \ll z_n^{-3/2}\), the errors simplify to
\[
\int_{0}^{\alpha_n^{-1/4}} \frac{g_n(1,x,t)}{\pi \sqrt{e^t-1}} \, dt
= w_n(x) - 2\ln 2\,\sqrt{\alpha_n}\,\phi(x) + O(z_n^{-3/2}).
\]
Finally, substituting the definition of \(w_n(x)\) yields
\[
\int_{0}^{\alpha_n^{-1/4}} \frac{g_n(1,x,t)}{\pi \sqrt{e^t-1}} \, dt
= 1 - \Phi(x) - \Bigl( \sqrt{\frac{\alpha_n}{2}} + 2\ln 2\,\sqrt{\alpha_n} - \frac{x}{4n} \Bigr) \phi(x) + O(z_n^{-3/2}).
\]
Adding the tail integral, which is of lower order, completes the proof.   
      \end{proof} 
        
  Now we present the asymptotic on $\prod_{j=2}^{n}(1-{M}^{(n)}_{j, j}(x))$ and $\prod_{j=2}^{n}(1-\widetilde{M}^{(n)}_{j, j}(x)), $ whose proof is postponed to the Appendix.  
    \begin{lem}\label{mgeq1cn}
     	For sufficiently large 
     $	n$
     	and all $|x|\leq \sqrt{2\log z_n},$  we have  
     	$$\prod_{j=2}^{n}(1-{M}^{(n)}_{j, j}(x))=1+O(z_n^{-9/2})$$
     	and analogously      	$$\prod_{j=2}^{n}(1-\widetilde{M}^{(n)}_{j, j}(x))=1+O(z_n^{-9/2}).$$ 
     \end{lem} 
             
 Now that the behaviors of \({M}^{(n)}_{j, j}(x)\) and \(\widetilde{M}^{(n)}_{j, j}(x)\) have been fully understood, we proceed to prove the theorems in the regime \(\alpha = 0\). 
 
 \subsection{Proof of Theorem \ref{main} for $\alpha=0$} 
  We now proceed with the proof for the case \(\alpha = 0\). Recall that \(z_n = \alpha_n^{-1/2} \wedge n\). By Lemmas \ref{ed2} and \ref{mgeq1cn}, we have uniformly for \(|x| \le \sqrt{2\log z_n}\) that
\begin{equation}\label{79}
\begin{aligned}
\bigl| \mathbb{P}(X_n \le x) - \Phi(x) \bigr|
&= \bigl| (1 - {M}^{(n)}_{1, 1}(x))(1 + O(z_n^{-9/2})) - \Phi(x) \bigr| \\
&= \bigl| 1 - {M}^{(n)}_{1, 1}(x) - \Phi(x) + O(z_n^{-9/2}) \bigr| \\
&= \bigl| \bigl( \sqrt{\alpha_n} - \frac{x}{4n} \bigr) \phi(x) + O(z_n^{-19/10}) \bigr|. 
\end{aligned}
\end{equation}

The leading term \(\bigl( \sqrt{\alpha_n} - \frac{x}{4n} \bigr) \phi(x)\) is of order \(z_n^{-1}\), since \(\sqrt{\alpha_n}\vee n^{-1} = z_n^{-1}\). Its supremum over \(\mathbb{R}\) is attained at a point independent of \(n\) and is of the same order. Consequently,
\begin{equation}\label{middlesup}
\sup_{|x| \le\sqrt{2\log z_n}} \bigl| \mathbb{P}(X_n \le x) - \Phi(x) \bigr|
= (1+o(1)) \sup_{x \in \mathbb{R}} \phi(x) \bigl| \sqrt{\alpha_n} - \frac{x}{4n} \bigr|. 
\end{equation} 
The two terms \(\sqrt{\alpha_n}\) and \(n^{-1}\) compete to determine the asymptotic order. To compare them, it is convenient to introduce the parameter
\[
\beta := \lim_{n \to \infty} \frac{n^3}{k_n},
\]
since \(\sqrt{\alpha_n} = \sqrt{n/k_n}= n^{-1}\sqrt{n^3/k_n}\). The value of \(\beta\) therefore determines whether \(\sqrt{\alpha_n}\) or \(n^{-1}\) dominates.

For any \(d_1, d_2 \ge 0\) with \(d_1^2 + d_2^2 > 0\),
\begin{equation}\label{suprecal} \sup_{x\in\mathbb{R}}|d_1 - d_2 x|\phi(x)=\frac{d_1 + \sqrt{d_1^2 + 4d_2^2}}{2\sqrt{2\pi}} \exp\Bigl\{ -\frac{(d_1-\sqrt{d_1^2+4d_2^{2}})^{2}}{8d_2^2} \Bigr\}.\end{equation} 
Applying this formula yields the following estimates.

- If \(\beta = +\infty\), i.e., \(\sqrt{\alpha_n}\gg n^{-1}\), then  
\[
\sup_{x\in\mathbb{R}}\phi(x)\bigl|\sqrt{\alpha_n}-\frac{x}{4n}\bigr| = \sqrt{\alpha_n}\sup_{x\in\mathbb{R}}\phi(x)=\frac{\sqrt{\alpha_n}}{\sqrt{2\pi}}.
\]

- If \(\beta = 0\), i.e., \(\sqrt{\alpha_n}\ll n^{-1}\), then  
\[
\sup_{x\in\mathbb{R}}\phi(x)\bigl|\sqrt{\alpha_n}-\frac{x}{4n}\bigr| = \frac{1}{4n}\sup_{x\in\mathbb{R}}|x|\phi(x)= \frac{1}{4\sqrt{2\pi e} \, n}.
\]

- If \(\beta \in (0, +\infty)\), then
\[
\begin{aligned}
\sup_{x\in\mathbb{R}}\phi(x)\bigl|\sqrt{\alpha_n}-\frac{x}{4n}\bigr|
&= \sqrt{\alpha_n}\sup_{x\in\mathbb{R}}\phi(x)\Bigl|1-\frac{ x}{ 4n \sqrt{\alpha_n} } \Bigr| \\
&=\frac{2n\sqrt{\alpha_n}+\sqrt{1+4n^2\alpha_n}}{4\sqrt{2\pi} n}\exp\!\bigl(-\frac12\bigl(\sqrt{1+4n^2\alpha_n}-2n\sqrt{\alpha_n}\bigr)^2\bigr)\\
&=\frac{(2\sqrt{\beta} + \sqrt{1+4\beta})(1+o(1))}{4\sqrt{2\pi e} \, n} \exp\!\bigl( \frac{2\sqrt{\beta}}{\sqrt{4\beta + 1} + 2\sqrt{\beta}} \bigr),
\end{aligned}
\]
where the last equality uses the fact that \(n^2\alpha_n = \beta(1+o(1))\).

We next establish the uniform tail bound
\begin{equation}\label{tailsup}
\sup_{|x| > \sqrt{2\log z_n}} \bigl| \mathbb{P}(X_n \le x) - \Phi(x) \bigr| \ll z_n^{-1}. 
\end{equation}
A straightforward computation using the expansion in \eqref{79} gives
\[
\mathbb{P}(|X_n| \ge \sqrt{2\log z_n}) \asymp \Psi(\sqrt{2\log z_n}) + O(z_n^{-19/10})
\lesssim (z_n\sqrt{\log z_n})^{-1} + O(z_n^{-19/10}) \ll z_n^{-1}.
\]
We now treat the two tails separately. For \(x < -\sqrt{2\log z_n}\), by the triangle inequality and monotonicity,
\begin{equation}\label{01}
\begin{aligned}
\sup_{x < -\sqrt{2\log z_n}} \bigl| \mathbb{P}(X_n \le x) - \Phi(x) \bigr|
&\le \mathbb{P}(X_n \le -\sqrt{2\log z_n}) + \Phi(-\sqrt{2\log z_n}) \\
&\lesssim \Phi(-\sqrt{2\log z_n}) + O(z_n^{-19/10}) \ll z_n^{-1}.
\end{aligned}
\end{equation}
Similarly, for \(x > \sqrt{2\log z_n}\),
\begin{equation}\label{02}
\begin{aligned}
\sup_{x > \sqrt{2\log z_n}} \bigl| \mathbb{P}(X_n \le x) - \Phi(x) \bigr|
&\le \mathbb{P}(X_n \ge \sqrt{2\log z_n}) + \Psi(\sqrt{2\log z_n}) \ll z_n^{-1}.
\end{aligned}
\end{equation}
Combining \eqref{01} and \eqref{02} yields \eqref{tailsup}.

With the tail estimate \eqref{tailsup} established, together with the central estimate \eqref{middlesup} and the uniform bound on \(\phi(x) \bigl| \sqrt{\alpha_n} - \frac{x}{4n} \bigr|\), the proof of Theorem \ref{main} for \(\alpha = 0\) is complete.
    
  \subsection{Proof of Theorem 2 for $\alpha=0$} \label{sec:alphainzero}
 When \(\alpha = +\infty\), the Hilbert-Schmidt norm of \(\widetilde{M}^{(n)}(x)\) becomes negligible, allowing the approximation
\[
\det\bigl({\rm I}_n - \widetilde{M}^{(n)}(x)\bigr) \approx \exp\bigl(-\operatorname{Tr} \widetilde{M}^{(n)}(x)\bigr).
\]
This is a standard technique in the literature and ultimately yields the Gumbel distribution.

For \(\alpha = 0\), however, the situation is fundamentally different: the norm no longer tends to zero. Indeed, Lemma \ref{m0cn} provides the lower bound
\[
\|\widetilde{M}^{(n)}(x)\|_{\mathrm{HS}} \ge \widetilde{M}^{(n)}_{1, 1}(x) \asymp 1-\Phi(x),
\]
which can be of constant order (e.g., for bounded \(|x|\)) and is not uniformly small over the relevant range \(|x| \le \sqrt{2\log z_n}\), where \(z_n = \alpha_n^{-1/2} \wedge n\). Consequently, the reduction to the trace alone is no longer valid in this regime, and a different approach is required to handle the determinant.

To this end, we first establish the following auxiliary lemma, whose proof is given in the Appendix.    

 \begin{lem}\label{traceB2}Let 
    $\widetilde{M}_{j,k}^{(n)}(x)$  be defined as above. Then
     	$$\sum_{j< k}\frac{(\widetilde{M}_{j, k}^{(n)}(x))^2}{(1-\widetilde{M}_{j, j}^{(n)}(x))(1-\widetilde{M}_{k, k}^{(n)}(x))}\ll z_n^{-2}$$ 
     	uniformly for $|x|\le \sqrt{2\log z_n}$ and sufficiently large $n.$
     \end{lem}     
We now set  
\[
d_j(x) = 1 - \widetilde{M}_{j,j}^{(n)}(x) > 0, \quad \Lambda(x) = \operatorname{diag}(d_1(x), \dots, d_n(x)),
\]  
and define  
\[
N_{j,k}(x) = \widetilde{M}_{j,k}^{(n)}(x) \mathbf{1}_{j \neq k},
\]  
so that  
\[
\mathrm{I}_n - \widetilde{M}^{(n)}(x) = \Lambda(x) - N(x).
\]  
Observe that  
\[
\det(\Lambda(x)) = \prod_{j=1}^{n} \bigl(1 - \widetilde{M}_{j,j}^{(n)}(x)\bigr),
\]  
and Lemmas  \ref{m0cn} and \ref{mgeq1cn}  guarantee that \(\det(\Lambda(x))>0.\)  

Introduce  
\[
B(x) = \Lambda^{-1}(x)N(x),
\]  
whose entries are  
\[
B_{j,k}(x) = \frac{N_{j,k}(x)}{d_j(x)} = \frac{\widetilde{M}_{j,k}^{(n)}(x)}{1 - \widetilde{M}_{j, j}^{(n)}(x)} \quad (k \neq j), \qquad B_{j, j}(x) = 0.
\]  
Then we have the factorization  
\begin{equation}\label{detMzerom}
\det\bigl(\mathrm{I}_n - \widetilde{M}^{(n)}(x)\bigr) = \det(\Lambda(x)) \det\bigl(\mathrm{I}_n - B(x)\bigr)
= \det\bigl(\mathrm{I}_n - B(x)\bigr) \prod_{j=1}^{n} \bigl(1 - \widetilde{M}_{j,j}^{(n)}(x)\bigr).
\end{equation}  

Let \(\lambda_1(x), \dots, \lambda_n(x)\) be the eigenvalues of \(B(x)\). From the definition of the trace,
\[
\operatorname{Tr}\bigl(B^2(x)\bigr) = \sum_{i=1}^{n} \lambda_i^2(x) = \sum_{j=1}^{n}\sum_{k=1}^{n} B_{j, k}(x)B_{k, j}(x).
\]  
Using the symmetry \(\widetilde{M}_{j, k}^{(n)}(x) = \widetilde{M}_{k, j}^{(n)}(x)\), the fact that \(B_{j, j}(x) = 0\), and Lemma \ref{traceB2}, we obtain  
\[
\operatorname{Tr}\bigl(B^2(x)\bigr) = 2\sum_{j<k} \frac{\bigl(\widetilde{M}_{j,k}^{(n)}(x)\bigr)^2}{\bigl(1-\widetilde{M}_{j, j}^{(n)}(x)\bigr)\bigl(1-\widetilde{M}_{k, k}^{(n)}(x)\bigr)} \ll z_n^{-2}.
\]  
This estimate yields  
\[
\rho_n(x):=\max_{1\le j\le n}|\lambda_j(x)| \leq \sqrt{\operatorname{Tr}\bigl(B^2(x)\bigr)} \ll 1.
\]  
Consequently, the following series expansion is valid:  
\[
\log\bigl(\mathrm{I}_n - B(x)\bigr) = -\sum_{m=1}^{\infty} \frac{1}{m} B^{m}(x).
\]  
Using \(\det(\mathrm{I}_n - B(x)) = \exp\bigl(\operatorname{Tr} \log (\mathrm{I}_n - B(x))\bigr)\), we get  
\[
\det\bigl(\mathrm{I}_n - B(x)\bigr) = \exp\bigl( -\sum_{m=1}^{\infty} \frac{1}{m} \operatorname{Tr}\bigl(B^{m}(x)\bigr) \bigr).
\]  
For \(m \geq 3\), we have the estimate  
\[
\bigl|\operatorname{Tr}\bigl(B^{m}(x)\bigr)\bigr| = \bigl| \sum_{i=1}^{n} \lambda_i^{m}(x) \bigr| \leq \sum_{i=1}^{n} |\lambda_i(x)|^{m} \leq \rho_{n}^{m-2}(x) \sum_{i=1}^{n} \lambda_i^{2}(x) \ll \operatorname{Tr}\bigl(B^{2}(x)\bigr).
\]  
Since \(\operatorname{Tr}(B(x)) = 0\) and \(\operatorname{Tr}(B^{2}(x)) \ll z_n^{-2}\), we conclude that  
\[
\det\bigl(\mathrm{I}_n - B(x)\bigr) = 1 + o(z_n^{-2}).
\]  
Substituting this asymptotic together with Lemmas \ref{m0cn} and \ref{mgeq1cn} into \eqref{detMzerom}, we obtain 
\[
\begin{aligned}
\mathbb{P}(\widetilde{X}_{n}\leq x) &= \det\bigl({\mathrm I}_n-\widetilde{M}^{(n)}(x)\bigr)\\
&= (1+o(z_n^{-2}))\bigl(1-\widetilde{M}^{(n)}_{1, 1}(x)\bigr) \\
&= \Phi(x) + \bigl( \frac{\sqrt{2}+4\ln2}{2}\sqrt{\alpha_n} - \frac{x}{4n} \bigr) \phi(x) + O(z_n^{-19/10}).
\end{aligned}
\]
Repeating the argument that led from \eqref{79} to \eqref{02} gives  
\[
\sup_{x\in \mathbb{R}} \bigl| \mathbb{P}(\widetilde{X}_n \leq x) - \Phi(x) \bigr| = (1+o(1)) \sup_{x\in\mathbb{R}} \phi(x) \bigl| \frac{\sqrt{2}+4\ln2}{2}\sqrt{\alpha_n} - \frac{x}{4n} \bigr|.
\]

Now recall the parameter \(\beta := \lim_{n \to \infty} n^3 / k_n\). The asymptotic behavior of the right-hand side depends on the value of \(\beta\):

- If \(\beta = +\infty\), i.e. \(\sqrt{\alpha_n} \gg n^{-1}\), the linear term in \(x\) is negligible and we obtain
\[
\sup_{x\in\mathbb{R}} \phi(x) \bigl| \frac{\sqrt{2}+4\ln2}{2}\sqrt{\alpha_n} - \frac{x}{4n} \bigr|
= \frac{\sqrt{2}+4\ln2}{2} \sqrt{\alpha_n} \sup_{x\in\mathbb{R}} \phi(x)
= \frac{(1+2\sqrt{2}\ln2)\sqrt{\alpha_n}}{2\sqrt{\pi}}.
\]

- If \(\beta = 0\), i.e. \(\sqrt{\alpha_n} \ll n^{-1}\), the \(\sqrt{\alpha_n}\) term is negligible and we get
\[
\sup_{x\in\mathbb{R}} \phi(x) \bigl| \frac{\sqrt{2}+4\ln2}{2}\sqrt{\alpha_n} - \frac{x}{4n} \bigr|
= \frac{1}{4n} \sup_{x\in\mathbb{R}} |x|\phi(x)
= \frac{1}{4\sqrt{2\pi e} \, n}.
\]

- If \(\beta \in (0, +\infty)\), we again use relation \eqref{suprecal}.  Set \(c := (\sqrt{2}+4\ln2)/2\) and note that \(n^2\alpha_n = \beta(1+o(1))\). Then
\[
\begin{aligned}
&\sup_{x\in\mathbb{R}} \phi(x) \bigl| \frac{\sqrt{2}+4\ln2}{2}\sqrt{\alpha_n} - \frac{x}{4n} \bigr| \\
&= \sqrt{\alpha_n}(1+o(1)) \sup_{x\in\mathbb{R}} \phi(x) \bigl| c - \frac{x}{4\sqrt{\beta}} \bigr| \\
&= \frac{\bigl(2c\sqrt{\beta} + \sqrt{1+4c^2\beta}\bigr)(1+o(1))}{4\sqrt{2\pi e} \, n} \exp\!\bigl( \frac{2c\sqrt{\beta}}{\sqrt{4c^2\beta + 1} + 2c\sqrt{\beta}} \bigr).
\end{aligned}
\]  

This completes the proof of Theorem \ref{thmrealpart} for the case \(\alpha = 0\).
   
 \section{Proof of the Theorems for $\alpha\in(0, +\infty)$}\label{sec:alphafinite}

This section addresses the intermediate scaling regime where $$\alpha = \lim_{n\to\infty} \frac{n}{k_n} \in (0, +\infty),$$ implying $k_n \asymp n$. We begin this section with the proof of Theorem \ref{main}. 

\subsection{Proof of Theorem \ref{main} for $\alpha\in(0,+\infty)$}

Recall the definitions

\[
M^{(n)}_{j, j}(x)=\mathbb{P}(\log Y_{n-j+1}>k_n\psi (n)+a_n+b_n x),
\]

with

\[
u_n(j, x)=\frac{j-1}{\sqrt{\alpha_n}}+a_n +b_n x,\qquad 
v_\alpha(j, x)=\frac{j-1}{\sqrt{\alpha}}+a+b x,
\]
and
\[
\Phi_{\alpha}(x)=\prod_{j=1}^{\infty}\Phi(v_{\alpha}(j, x)).
\]
We introduce the parameter \(s_n = |\alpha_n - \alpha|^{-1} \wedge n\) and select

\[
\ell_{1,\alpha}(n) = \Bigl(\frac{1}{10} \log s_n\Bigr)^{1/2}, \quad \ell_{2,\alpha}(n) = \frac{4\sqrt{\log s_n} - a_n}{b_n}, \quad r_n = \lfloor s_n^{1/10} \rfloor.
\]

Since \(k_n \gg 1\) and \(u_n(j,x)\) remains bounded for fixed \(j\) and \(x\), we may apply an Edgeworth expansion to \(\log Y_{n-j+1}\) to obtain a precise asymptotic expression for \(M^{(n)}_{j, j}(x)\). This approximation holds uniformly for \(x \in [-\ell_{1,\alpha}(n), \ell_{2,\alpha}(n)]\) and \(1 \le j \le r_n\).

Set \[\beta_n(x) := \log \mathbb{P}(X_n \le x)=\sum_{j=1}^{n} \log(1 - M^{(n)}_{j, j}(x)).\] Then
\begin{equation}\label{totalmiddle}
\bigl|\mathbb{P}(X_n \le x) - \Phi_\alpha(x)\bigr| = \Phi_\alpha(x) \big| \exp\bigl( \beta_n(x) - \sum_{j=1}^{\infty} \log \Phi(v_\alpha(j, x)) \bigr) - 1 \big|. 
\end{equation}
The core of the proof is to show that the exponent is asymptotically negligible, i.e.,
\begin{equation}\label{o1}
\beta_n(x) - \sum_{j=1}^{\infty} \log \Phi(v_\alpha(j, x)) = o(1). 
\end{equation} 
Once this holds, \eqref{totalmiddle} implies that 
\begin{equation}\label{totalalpha}
\bigl|\mathbb{P}(X_n \le x) - \Phi_\alpha(x)\bigr|=(1+o(1)) \Phi_\alpha(x) |\beta_n(x) - \sum_{j=1}^{\infty} \log \Phi(v_\alpha(j, x))|.	
\end{equation}
We decompose the difference as
\begin{equation}\label{o2}
\begin{aligned}
\beta_n(x) - \sum_{j=1}^{\infty} \log \Phi(v_\alpha(j, x))
&=  \sum_{j=1}^{r_n} \log\!\big( 1 + \frac{\Psi(v_\alpha(j, x)) - M^{(n)}_{j, j}(x) }{\Phi(v_\alpha(j, x))} \big) \\
&\quad + \sum_{j=r_n+1}^{n} \log(1 -M^{(n)}_{j, j}(x)) - \sum_{j=r_n+1}^{\infty} \log \Phi(v_\alpha(j, x)). \end{aligned}
\end{equation}
We now present two lemmas, the first one is for the first term of \eqref{o2}
and the second one is for the last two terms of \eqref{o2}. 

\begin{lem}\label{m1} Let \(1\leq j\leq r_n.\) Set  
\[
c_1=\frac{\sqrt{\log(\alpha+1)}}{w(\alpha+1)}+\frac{2}{\sqrt{\log(\alpha+e)}w(\alpha+e^{\frac{1}{\sqrt{2\pi}}})}-\frac{\log(\sqrt{2\pi}\log(\alpha+e^{\frac{1}{\sqrt{2\pi}}}))}{w(\alpha+e)\sqrt{\log(\alpha+e)}}
\]  
and  
\[
c_2=\frac{1}{2(\alpha+e)(\log(\alpha+e))^{3/2}},
\]  
where \(w(t)=2t\log t.\) Then, uniformly on \(-\ell_{1,\alpha}(n)\leq x\leq \ell_{2,\alpha}(n)\),
\[
M^{(n)}_{j, j}(x)=\Psi(v_\alpha(j, x))-\phi(v_\alpha(j, x))\Bigl(\frac{1}{n}q_{1}(j, x)+(\alpha_n-\alpha)q_2(j, x)\Bigr)+O(s_n^{-3/2}),
\]  
where  
\[
q_{1}(j, x)=\frac{2\alpha(v^{2}_\alpha(j, x)-1)-3\sqrt{\alpha}(2j-1)v_{\alpha}(j, x)+6j(j-1)}{12\sqrt{\alpha}}
\]  
and  
\[
q_{2}(j, x)= c_1-c_2x-\frac{j-1}{2\alpha^{3/2}}.
\]
\end{lem}

\begin{lem}\label{jn1} Uniformly for \(x\ge -\ell_{1, \alpha}(n)\) as \(n\to\infty\),
\[
\sum_{j=r_n+1}^{n}\log(1-M^{(n)}_{j, j}(x))=o\!\big(e^{-\frac{s_n^{1/5}}{3\alpha}}\big)\quad\text{and}\quad
\sum_{j=r_n+1}^{+\infty}\log\Phi(v_\alpha(j, x))=o(e^{-\frac{s_n^{1/5}}{3\alpha}}). \]
\end{lem}  
Lemma \ref{jn1} and \eqref{o2} give
\begin{equation}\label{betanxdif}
\beta_n(x) - \sum_{j=1}^{+\infty} \log \Phi(v_{\alpha}(j, x))
= \sum_{j=1}^{r_n } \log\!\big(1 + \frac{1 - M^{(n)}_{j, j}(x) - \Phi(v_{\alpha}(j, x))}{\Phi(v_{\alpha}(j, x))}\big) + o\!\big(e^{-\frac{s_n^{1/5}}{3\alpha}}\big). 
\end{equation}
For the term inside the logarithmic function of \eqref{betanxdif},  Lemma \ref{m1} indicates that 
\begin{equation}\label{dalpha}
\aligned
d_\alpha(j,x) : &= \Psi(v_\alpha(j,x))  - M^{(n)}_{j, j}(x)\\
&= \phi(v_\alpha(j,x)) \left( n^{-1} q_1(j,x) + (\alpha_n - \alpha) q_2(j,x) \right) + O(s_n^{-3/2}) 
\endaligned
\end{equation}
and then the facts that $\phi$ is bounded and $|q_1(j, x)+q_2(j, x)|\le \log s_n$ imply  
\[|d_\alpha(j, x)| \lesssim s_n^{-1} \log s_n \ll 1\] uniformly on \(1\le j\le r_n.\)
Thereby, it follows from \eqref{betanxdif} that
\begin{equation}\label{betanxdifnew}
\beta_n(x) - \sum_{j=1}^{+\infty} \log \Phi(v_{\alpha}(j, x))
= \sum_{j=1}^{r_n } \log\!\big(1 + \frac{d_\alpha(j, x)}{\Phi(v_{\alpha}(j, x))}\big) + o\big(e^{-\frac{s_n^{1/5}}{3\alpha}}\big). 
\end{equation}
 Using the monotonicity of the standard normal distribution function \(\Phi\), we obtain for all \(-\ell_{1,\alpha}(n) \le x \le \ell_{2,\alpha}(n)\) and \(1 \le j \le r_n\),
\[
\frac{1}{\Phi(v_\alpha(j,x))} \le \frac{1}{\Phi(v_\alpha(1,-\ell_{1,\alpha}(n)))} \lesssim v_\alpha(1,-\ell_{1,\alpha}(n)) e^{\frac12 v_\alpha^2(1,-\ell_{1,\alpha}(n))} \lesssim \sqrt{\log s_n} \, s_n^{\frac1{20}}.
\]
Consequently,
\begin{equation}\label{dalphaPhi}
\frac{|d_\alpha(j, x)|}{\Phi(v_\alpha(j, x))} \lesssim s_n^{-\frac{19}{20}} (\log s_n)^{\frac32}. 
\end{equation}
A lower bound is obtained by considering the term with \(j=1\):
\[
\sum_{j=1}^{r_n} \frac{|d_\alpha(j,x)|}{\Phi(v_\alpha(j,x))} \gtrsim \frac{|d_\alpha(1,x)|}{\Phi(v_\alpha(1,x))} \gtrsim s_n^{-\frac{19}{20}} \sqrt{\log s_n} \gtrsim \exp(-\frac{s_n^{\frac{1}{5}}}{3\alpha}).
\]
Comparing the upper and lower bounds, the term \(o\big(e^{-s_n^{1/5}/(3\alpha)}\big)\) in \eqref{betanxdifnew} is negligible and then 
\[
\bigl| \beta_n(x) - \sum_{j=1}^{\infty} \log \Phi(v_\alpha(j, x)) \bigr| = (1+o(1)) \sum_{j=1}^{r_n} \frac{|d_\alpha(j,x)|}{\Phi(v_\alpha(j,x))}.
\]
The upper bound \eqref{dalphaPhi} and the fact \(r_n \asymp s_n^{\frac{1}{10}}\) give
\[
\bigl| \beta_n(x) - \sum_{j=1}^{\infty} \log \Phi(v_\alpha(j, x)) \bigr| \lesssim r_n s_n^{-\frac{19}{20}} (\log s_n)^{\frac32} = s_n^{-\frac{17}{20}} (\log s_n)^{\frac32} = o(1),
\]
whence it follows from \eqref{totalalpha} and \eqref{dalpha} that
\begin{equation}\label{basicabove}
\bigl| \mathbb{P}(X_n \le x) - \Phi_{\alpha}(x) \bigr|
= (1 + o(1)) \, \Phi_{\alpha}(x) \bigl| \sum_{j=1}^{+\infty} \frac{\phi(v_{\alpha}(j, x))}{\Phi(v_{\alpha}(j, x))} \bigl( n^{-1} q_1(j, x) + (\alpha_n - \alpha) q_2(j, x) \bigr) \bigr| 
\end{equation}
uniformly for \(x \in [-\ell_{1,\alpha}(n), \ell_{2,\alpha}(n)]\). This establishes the desired convergence rate in the intermediate regime.

As will be shown in the Appendix,

\[
\sup_{x \in \mathbb{R}} \Phi_\alpha(x) \sum_{j=1}^{\infty} |q_k(j,x)| \frac{\phi(v_\alpha(j,x))}{\Phi(v_\alpha(j,x))} < +\infty \qquad (k=1,2).
\]

To complete the proof of Theorem \ref{main} for \(\alpha\in (0, +\infty)\), we justify replacing the supremum over the middle interval \([-\ell_{1,\alpha}(n), \ell_{2,\alpha}(n)]\) by the supremum over \(\mathbb{R}\). From \eqref{basicabove}, for \(-\ell_{1,\alpha}(n) \le x \le \ell_{2,\alpha}(n)\),

\begin{equation}\label{ignoring}
\bigl| \mathbb{P}(X_n \le x) - \Phi_\alpha(x) \bigr| \asymp \Phi_\alpha(x) s_n^{-1} \ll 1.
\end{equation}

Using \eqref{ignoring} together with monotonicity of the cumulative distribution function yields
\[
\sup_{x \in (-\infty, -\ell_{1,\alpha}(n)]} \bigl| \mathbb{P}(X_n \le x) - \Phi_\alpha(x) \bigr|
\le \mathbb{P}(X_n \le -\ell_{1,\alpha}(n)) + \Phi_\alpha(-\ell_{1,\alpha}(n))
\lesssim \Phi_\alpha(-\ell_{1,\alpha}(n)),
\]

and similarly,

\[
\sup_{x \in [\ell_{2,\alpha}(n), +\infty)} \bigl| \mathbb{P}(X_n \le x) - \Phi_\alpha(x) \bigr|
\lesssim  1-\Phi_\alpha(\ell_{2,\alpha}(n)).
\]
Choose \(m_1 = \bigl\lfloor \sqrt{\alpha} (b \ell_{1,\alpha}(n)/2 - a) \bigr\rfloor+1\). Under this choice,
\[
v_\alpha(m_1, x) =\frac{m_1-1}{\sqrt{\alpha}}+a+b x= -\frac{b \ell_{1,\alpha}(n)}{2}(1 + o(1)).
\]
Consequently,
\[
\Phi_\alpha(-\ell_{1,\alpha}(n)) \le \bigl(\Phi(v_\alpha(m_1, -\ell_{1,\alpha}(n))) \bigr)^{m_1}
\ll \ell_{1,\alpha}^{-m_1}(n) \exp(-\frac{\sqrt{\alpha} b^3 \ell_{1,\alpha}^3(n)}{20} )
\ll s_n^{-1}.
\]
On the other hand, the definition \(\ell_{2,\alpha}(n) = \frac{4\sqrt{\log s_n} - a_n}{b_n}\) gives \[v_\alpha(j, \ell_{2,\alpha}(n))\ge v_\alpha(1, \ell_{2,\alpha}(n)) = 4\sqrt{\log s_n} \gg 1,\] which together with Mills' ratio ensures
\[
\Psi(v_\alpha(j, \ell_{2,\alpha}(n)))=\frac{1 + o(1)}{\sqrt{2\pi}} v_\alpha^{-1}(j, \ell_{2,\alpha}(n)) \exp( -\frac{v_\alpha^2(j, \ell_{2,\alpha}(n))}{2} )\ll 1. 
\]
Hence, leveraging the elementary inequality \(1-e^{-t}\le t\) for \(t>0\), we have 

\[
\begin{aligned}
1 - \Phi_\alpha(\ell_{2,\alpha}(n))
&= 1 - \exp \{ \sum_{j=1}^{\infty} \log (1-\Psi(v_\alpha(j, \ell_{2,\alpha}(n)))) \} \\
&\lesssim \sum_{j=1}^{\infty} v_\alpha^{-1}(j, \ell_{2,\alpha}(n)) \exp\!( -\frac{v_\alpha^2(j, \ell_{2,\alpha}(n))}{2} ).
\end{aligned}
\]
Lemma \ref{sum} then gives
\[
1 - \Phi_\alpha(\ell_{2,\alpha}(n))\lesssim v_\alpha^{-1}(1, \ell_{2,\alpha}(n)) \exp( -\frac{v_\alpha^2(1, \ell_{2,\alpha}(n))}{2} )
= 16 s_n^{-7} (\log s_n)^{-1/2} \ll s_n^{-1}.
\]

Combining the estimates above, both tails are of order \(o(s_n^{-1})\), which is negligible compared to the bound on the central interval. Therefore,

\[
\sup_{x \in \mathbb{R}} \bigl| \mathbb{P}(X_n \le x) - \Phi_\alpha(x) \bigr|
= \sup_{-\ell_{1,\alpha}(n) \le x \le \ell_{2,\alpha}(n)} \bigl| \mathbb{P}(X_n \le x) - \Phi_\alpha(x) \bigr|.
\]

This completes the proof of Theorem \ref{main} for the case \(\alpha \in (0, +\infty)\).
 
 \subsection{Proof of Theorem 2 for $ \alpha\in(0,+\infty)$} 
 
For the largest real-part \(\max_{1\le j\le n} \Re Z_j\), a fundamental difficulty of a different nature arises. In this regime, the matrix \(\widetilde{M}^{(n)}(x)\) has a non-negligible number of off-diagonal entries of order one, rendering the previous two approximations invalid: we can no longer approximate the determinant by \(\exp(-\operatorname{Tr}(\widetilde{M}^{(n)}(x))\) (since \(\|\widetilde{M}^{(n)}\|_{\rm HS}\) is not small) nor by \(\prod_{j=1}^n(1-\widetilde{M}_{j, j}^{(n)}(x))\) (because off diagonal contributions are significant).

Thus, the analysis must confront the full nonlinear structure of \(\det({\rm I}_n - \widetilde{M}^{(n)}(x))\), a problem of considerable complexity. However, for each fixed \(\alpha \in (0, +\infty)\), we can study the limiting operator \(\widetilde{M}(x,\alpha)\) to which \(\widetilde{M}^{(n)}(x)\) converges in an appropriate sense. As the asymptotics become significantly more delicate, we forgo the convergence rate and instead focus on establishing the existence of the limit and proving that it defines a distribution function. To make this precise, set
\[
\widetilde{v}_{\alpha}(j,x)=\frac{j-1}{\sqrt{\alpha}}+\widetilde{a}+\widetilde{b}x,
\]
and recall that
\[
s_n = |\alpha_n - \alpha|^{-1} \wedge n, \qquad r_n = \lfloor s_n^{1/10} \rfloor.
\]
The following lemma, whose proof is deferred to the Appendix, confirms that \(\widetilde{M}(x, \alpha)\) is the entrywise limit of \(\widetilde{M}^{(n)}(x)\). 
 \begin{lem}\label{M0infty} 
 	Let $r_n$ and $s_n$ be  defined as above and  let $\Psi(t)=1-\Phi(t).$ 
 	 	\begin{itemize}\item[(1).] For any $1 \le j \le r_n$ and $|x| \le \sqrt{\log s_n}$, it holds 
 $$\aligned \widetilde{M}_{j, j}^{(n)}(x)&=\frac{2}{\pi}\int_{0}^{\pi/2}\Psi(\widetilde{v}_{\alpha}(j, x)-\sqrt{\alpha}\log\cos^2\theta) d\theta+O(s_n^{-1})
.\endaligned $$
 \item[(2).] For any $1\le j\neq k\le r_n$ with $j-k$ even and $|x| \le \sqrt{\log s_n},$ we have 
 \begin{equation}\label{mjklimit} 
 	\widetilde{M}_{j, k}^{(n)}(x)=e^{-\frac{(j-k)^2}{4\alpha}}\frac{2}{\pi}\int_{0}^{\pi/2}\cos((j-k)\theta)\Psi(\widetilde{v}_{\alpha}(\frac{j+k}{2},x)-\sqrt{\alpha}\log \cos^2\theta)d\theta+O(s_n^{-9/10}). 
 \end{equation}
\item[(3).] Whenever $r_n+1\le j\le n$ and $x \ge -\sqrt{\log s_n},$ we have 
 \begin{equation}\label{1-mjj}
 	\widetilde{M}_{j, j}^{(n)}(x)\ll s_n^{-1/10}\exp(-s_n^{1/6}). 
 \end{equation}
 \item[(4).] For $|x| \le \sqrt{\log s_n},$ we have 
 \begin{equation}\label{summjjj}
 	\sum_{j=r_n+1}^{n}\widetilde{M}_{j, j}^{(n)}(x)\ll \exp(-\frac{r^2_n}{3\alpha}).
 \end{equation}
 \end{itemize}
  \end{lem}
 
From Lemma \ref{M0infty}, we obtain  the entrywise limit 
  \begin{equation}\label{Mxalphajk} \widetilde{M}_{j, k}(x, \alpha)=\frac{2}{\pi}\exp(-\frac{(j-k)^2}{4\alpha})\int_0^{\pi/2}\cos((j-k)\theta)\Psi(\widetilde{v}_{\alpha}(\frac{j+k}{2}, x)-\sqrt{\alpha}\log\cos^2\theta)d\theta\end{equation}
for any $j, k\ge 1$ with $j-k$ even and $\widetilde{M}_{j, k}(x, \alpha)=0$ when $j-k$ is odd. 

Next, we will show \(\widetilde{M}(x,\alpha)\) is trace class, so that $$\widetilde{\Phi}_{\alpha}(x):=\det({\rm I}-\widetilde{M}(x, \alpha))$$ is well defined and  $\widetilde{\Phi}_{\alpha}$ is a distribution function. Moreover, we establish the convergence
\begin{equation}\label{lastverifi}\lim_{n\to\infty}\det ({\rm I}_n-\widetilde{M}^{(n)}(x))=\det({\rm I}-\widetilde{M}(x, \alpha)).\end{equation} 

Although \(\det({\rm I} - \widetilde{M}(x,\alpha))\) does not admit a closed form expression in general, its properties-such as continuity in \(\alpha\)-can be derived from the operator \(\widetilde{M}(x,\alpha)\). Using these properties, we characterize the continuous phase transition across \(\alpha \in (0,\infty)\) without an explicit formula. Specifically, by examining the limits \(\alpha \to +\infty\) and \(\alpha \to 0^+\), we connect the behavior in this intractable regime to the solvable Gaussian and Gumbel limits, thereby completing the proof of the continuous transition.

\subsubsection{$\widetilde{\Phi}_{\alpha}$ is a distribution function} 
 
Due to the decreasingness of $\Psi$ and the fact $|\cos|\le 1,$ we have 
$$\aligned |\widetilde{M}_{j, k}(x, \alpha)|&\lesssim \exp(-\frac{(j-k)^2}{4\alpha}) \Psi(\widetilde{v}_{\alpha}(\frac{j+k}{2}, x)) \\
&\lesssim \exp(-\frac{(j-k)^2}{4\alpha}) \widetilde{v}_{\alpha}^{-1}(\frac{j+k}{2}, x)\exp(-\frac12\widetilde{v}_{\alpha}^2(\frac{j+k}{2}, x)).
	\endaligned 
$$
Using the substitution $p=\frac{j+k}{2}$ and $q=\frac{j-k}{2},$ together with Lemma \ref{sum}, we derive 
\begin{equation}\label{M1norm}\aligned \sum_{1\le k\le j}|\widetilde{M}_{j, k}(x, \alpha)|&\lesssim\sum_{q=0}^{+\infty}\exp(-\frac{q^2}{\alpha}) \sum_{p=1}^{+\infty}\widetilde{v}_{\alpha}^{-1}(p, x)\exp(-\frac12\widetilde{v}_{\alpha}^2(p, x)) \\
&\lesssim\widetilde{v}_{\alpha}^{-1}(1, x)\exp(-\frac12\widetilde{v}_{\alpha}^2(1, x))<+\infty
\endaligned \end{equation} 
uniformly on $\alpha\in (0, N]$ for some finite $N.$
This guarantees that $M(x, \alpha)$ is trace class when $\alpha>0$ fixed and then 
$$\widetilde{\Phi}_\alpha(x):={\rm det}({\rm I}-\widetilde{M}(x, \alpha))$$ is well defined.
By \eqref{M1norm} and the continuity of each \(\widetilde{M}_{j,k}(x,\alpha)\) in \(x\), the map \(x \mapsto \widetilde{M}(x,\alpha)\) is continuous in the trace norm; consequently, \(\widetilde{\Phi}_\alpha\) is continuous. Next, using the expression \eqref{Mxalphajk} and the bound \(|\widetilde{M}_{j,k}(x,\alpha)|\le 1\), together with the dominated convergence theorem, we examine the limits as \(x\to\pm\infty\). As \(x\to -\infty\), \(\Psi(-\infty)=1\), so
\[
\lim_{x\to -\infty} \widetilde{M}_{j,k}(x,\alpha) = \frac{2}{\pi}\exp\!\big(-\frac{(j-k)^2}{4\alpha}\big)\int_0^{\pi/2}\!\cos((j-k)\theta)\,d\theta.
\]
For \(j=k\), the integral equals \(\pi/2\), giving \(\lim\limits_{x\to -\infty} \widetilde{M}_{j,j}(x,\alpha)=1\). For \(j\neq k\), note that \(\widetilde{M}_{j, k}(x,\alpha)=0\) when \(j-k\) is odd, and when \(j-k\) is even and nonzero the integral vanishes because \(\int_0^{\pi/2}\cos(n\theta)d\theta = \sin(n\pi/2)/n =0\) for even \(n\). Hence, \(\lim\limits_{x\to -\infty} \widetilde{M}_{j,k}(x,\alpha)=0\) for all \(j\neq k\). Thus, \(\widetilde{M}(-\infty,\alpha)={\mathrm I}\), the identity matrix, and consequently
\[
\widetilde{\Phi}_\alpha(-\infty)=\det({\mathrm I}-{\mathrm I})=0.
\]
On the other hand, as \(x\to +\infty\), \(\Psi(+\infty)=0\), so
\[
\lim_{x\to +\infty} \widetilde{M}_{j,k}(x,\alpha) = \frac{2}{\pi}\exp\big(-\frac{(j-k)^2}{4\alpha}\big)\int_0^{\pi/2}\!\cos((j-k)\theta)\cdot0\,d\theta = 0
\]
for all \(j,k\ge 1\). Therefore, \(\widetilde{M}(+\infty,\alpha)=\mathbf{0}\), and
\[
\widetilde{\Phi}_\alpha(+\infty)=\det({\mathrm I}-\mathbf{0})=1.
\]
Combined with the monotonicity of \(\widetilde{\Phi}_\alpha\) (which follows from the fact that \(\widetilde{M}(x,\alpha)\) is decreasing in \(x\)), we conclude that \(\widetilde{\Phi}_\alpha\) is a distribution function.

\subsubsection{Verification of \eqref{lastverifi}}
Since \(\widetilde{M}(x,\alpha)\) is trace class, its Fredholm determinant can be approximated by determinants of finite-dimensional truncations:
\[
\det({\rm I} - \widetilde{M}(x,\alpha)) = \lim_{n \to \infty} \det({\rm I}_n - \widehat{M}^{(n
)}(x)),
\]
where \(\widehat{M}^{(n)}(x)\) denotes the \(n \times n\) principal submatrix of \(\widetilde{M}(x,\alpha)\).  
Therefore, \eqref{lastverifi} is equivalent to 
\begin{equation}\label{alphalasttogo}
\det({\rm I}_n - \widetilde{M}^{(n)}(x)) = \det({\rm I}_n - \widehat{M}^{(n)}(x)) + o(1) \quad \text{as } n \to \infty.
\end{equation}

	We first outline the main ideas. By partitioning \(\widetilde{M}^{(n)}(x)\) and \(\widehat{M}^{(n)}(x)\) and focusing on their leading \(r_n \times r_n\) principal submatrices, denoted by \(\widetilde{M}^{(1, n)}(x)\) and \(\widehat{M}^{(1, n)}(x)\) respectively, we show that the complementary blocks are negligible. This follows from the estimates \eqref{1-mjj} and \eqref{CSIneq}. Then, using the block determinant formula, the problem asymptotically reduces to comparing \(\det({\rm I}_{r_n} - \widetilde{M}^{(1, n)}(x))\) and \(\det({\rm I}_{r_n} - \widehat{M}^{(1, n)}(x))\). A perturbation argument shows that their difference tends to zero, thereby establishing the lemma.
	
	We now proceed with the detailed estimates. 
	Set \[
	\widetilde{M}^{(n)}(x) = \begin{pmatrix} \widetilde{M}^{(1, n)}(x) & \widetilde{M}^{(2, n)}(x) \\ (\widetilde{M}^{(2, n)}(x))' & \widetilde{M}^{(3, n)}(x) \end{pmatrix},
	\]
	where $\widetilde{M}^{(1, n)}(x)$ and $\widetilde{M}^{(3, n)}(x)$ are $r_n\times r_n$  and $(n-r_n)\times (n-r_n)$  submatrices, respectively. 
	Applying the block determinant formula yields
	\begin{equation}\label{detprod}\aligned
		\det(\mathrm{I}_n - \widetilde{M}^{(n)}&(x))
		 = \det(\mathrm{I}_{n-r_n} - \widetilde{M}^{(3, n)}(x))\\
		 &\times\det((\mathrm{I}_{r_n} - \widetilde{M}^{(1, n)}(x)) - \widetilde{M}^{(2, n)}(x)(\mathrm{I}_{n-r_n} - \widetilde{M}^{(3, n)}(x))^{-1}(\widetilde{M}^{(2, n)}(x))'). \endaligned 
	\end{equation}
The inequality \eqref{summjjj} ensures 
	\begin{equation}\label{Trnign}\operatorname{Tr}(\widetilde{M}^{(3, n)}(x))=\sum_{j=r_n+1}^{n} \widetilde{M}_{j,j}^{(n)}(x)\ll \exp(-\frac{r^2_n}{3\alpha})\ll 1 \end{equation}
	and then the inequality \eqref{CSIneq} tells 
	$$\aligned
	\|\widetilde{M}^{(3, n)}(x)\|_{\text{HS}}^2\leq (\sum_{j=r_n+1}^{n}
	\widetilde{M}_{j,j}^{(n)}(x))^2\ll\exp\{-\frac{2r^2_{n}}{3\alpha}\}.
	\endaligned$$
	Then, \eqref{E2} helps us to get
	\begin{equation}\label{M3nx}
		\det(\mathrm{I}_{n-r_n} - \widetilde{M}^{(3, n)}(x)) = \exp(-\operatorname{Tr}(\widetilde{M}^{(3, n)}(x)))(1 + o(1)) = 1 + o(\exp\{-\frac{r_n^2}{3\alpha}\}).
	\end{equation}
	Define $$\widetilde{E}(x)=\widetilde{M}^{(2, n)}(x)(\mathrm{I}_{n-r_n} - \widetilde{M}^{(3, n)}(x))^{-1}(\widetilde{M}^{(2, n)}(x))',
	\quad\quad \widetilde{D}(x)=\mathrm{I}_{r_n} - \widetilde{M}^{(1, n)}(x).$$
	We now show that $\det(\widetilde{D}(x)-\widetilde{E}(x)) =\det(\widetilde{D}(x))+o(1)$. 
	
The upper bound \eqref{1-mjj} ensures that \[({\rm I}_{n-r_n} - \widetilde{M}^{(3,n)}(x))^{-1} ={\rm I}_{n-r_n}(1 + o(\exp\{-\frac{r_n^2}{3\alpha}\})).\] Consequently,
\[
\widetilde{E}(x) = \widetilde{M}^{(2,n)}(x) \bigl( \widetilde{M}^{(2,n)}(x) \bigr)'(1 + o(\exp\{-\frac{r_n^2}{3\alpha}\})).
\]
	Set \(\widehat{E}^{(2, n)}(x) = \widetilde{M}^{(2,n)}(x) (\widetilde{M}^{(2,n)}(x) )'\).  Then for $1 \leq j, k \leq r_n$,
	\[\widehat{E}^{(2, n)}_{j, k}(x) = \sum_{i=1}^{n-r_n}\widetilde{M}^{(2,n)}_{j, i}(x) \widetilde{M}^{(2,n)}_{k,i}(x) = \sum_{i=r_n+1}^{n}\widetilde{M}^{(n)}_{j, i}(x) \widetilde{M}^{(n)}_{k,i}(x).\]
	From \eqref{CSIneq}, \eqref{Trnign}, and the bound $|\widetilde{M}_{j,j}^{(n)} |\leq 1$, we obtain
	$$|\widehat{E}^{(2, n)}_{j, k}(x)|\leq  \sum_{i=r_n+1}^{n}| \sqrt{\widetilde{M}_{j,j}^{(n)}(x)\widetilde{M}_{k, k}^{(n)}(x)}\widetilde{M}_{i, i}^{(n)}(x)|\leq \sum_{i=r_n+1}^{n}\widetilde{M}_{i, i}(x)\ll \exp\{-\frac{r_n^2}{3\alpha}\}.$$
	Consequently,
	\begin{equation}\label{E}
		\|\widetilde{E}(x)\|_{\max}:=\max_{j, k}|\widetilde{E}_{j, k}(x)|\ll e^{-\frac{r_n^2}{3\alpha}}. 
	\end{equation}
	Since $\|\widetilde{M}^{(1, n)}(x)\|_{\max}\le 1$ and its diagonal entries are positive,
	\begin{equation}\label{D}
		\max\{\|\widetilde{D}(x)\|_{\max}, \|\widetilde{D}(x)-\widetilde{E}(x)\|_{\max}\}\leq  1.
	\end{equation}
	Combining \eqref{E} and \eqref{D} with the perturbation bound for determinants gives
	$$\aligned 	|\det(\widetilde{D}(x)-\widetilde{E}(x)) - \det(\widetilde{D}(x))| \ll (r_n)! \exp(-\frac{r_n^{2}}{3\alpha}) .
	\endaligned$$
	Stirling's formula leads 
	$$r_n!\exp(-\frac{r_n}{3\alpha})
	\lesssim \sqrt{r_n}\exp\{-\frac{r_n^2}{3\alpha}+r_n\log r_n-r_n\}\ll e^{-\frac{r_n^2}{4\alpha}}\ll1, $$
	which implies 
	\begin{equation}\label{DEdif}\det(\widetilde{D}(x)-\widetilde{E}(x))=\det(\widetilde{D}(x))+o(e^{-\frac{r_n}{4\alpha}}).\end{equation}
	Putting \eqref{M3nx} and \eqref{DEdif} into \eqref{detprod}, we derive 
	$$	\det(\mathrm{I}_n - \widetilde{M}^{(n)}(x))=\det(\mathrm{I}_{j_n} - \widetilde{M}^{(1, n)}(x))+o(e^{-\frac{r_n^2}{4\alpha}}). $$
	Partition \(\widehat{M}^{(n)}(x)\) analogously:
	\[
	\widehat{M}^{(n)}(x) = \begin{pmatrix} \widehat{M}^{(1, n)}(x) & \widehat{M}^{(2, n)}(x)\\ (\widehat{M}^{(2, n)}(x))' & \widehat{M}^{(3, n)}(x) \end{pmatrix},
	\]
	and apply the same analysis to obtain
	\[
	\det\bigl(\mathrm{I}_n - \widehat{M}^{(n)}(x)\bigr) = \det(\widehat{D}(x))+o(e^{-\frac{r_n^2}{4\alpha}}),
	\]		
	where $\widehat{D}(x)={\rm I}_{r_n}-\widehat{M}^{(1, n)}(x).$
	It remains to prove that \[|\det(\widetilde{D}(x))- \det(\widehat{D}(x))| \to 0\] as \(n \to \infty\).	
Setting  
	$$F(x):=\widetilde{D}(x)-\widehat{D}(x),$$
whose entries are all $O(s_n^{-9/10})$ guaranteed by Lemma \ref{M0infty} and then 
	\begin{equation}\label{Fhs}
		\|F(x)\|_{\text{HS}}^2=\sum_{j=1}^{r_n}\sum_{k=1}^{r_n}F_{j,k}^{2}\le r_n^2 \max_{1\leq j,k\leq j_n} F_{j,k}^{2}\lesssim\frac{r_n^2}{s_n^{9/5}}. 
	\end{equation}
	Let $(\widetilde{\lambda}_i)_{1\le i\le r_n}$ and $(\widehat{\lambda}_i)_{1\le i\le r_n}$ be the eigenvalues of $\widetilde{D}(x)$ and $\widehat{D}(x)$, respectively, and set $w_i=\widehat{\lambda}_i-\widetilde{\lambda}_i$ and then
	\begin{equation}\label{deterdiff}\det (\widehat{D}(x))-\det (\widetilde{D}(x))=\prod_{i=1}^{r_n}(\widetilde{\lambda}_{i}+w_i)-\prod_{i=1}^{r_n} \widetilde{\lambda}_{i}.\end{equation} 
	
	Recall that $\phi_{n-j}(z) $ is orthogonal on the complex plane, and
		\[
	\widetilde{M}_{j, k}^{(1,n)}(x) = \int_{\widetilde{A}(x)} \phi_{n-j}(z)\;\overline{\phi_{n-k}(z)}\,  \; d^2 z=\langle \phi_{n-j},\,\phi_{n-k}\rangle_{\widetilde{A}(x)}.
	\]
	Hence,  $\widetilde{M}_{j, k}^{(1,n)}(x)$ is a Gram matrix and thus  positive semidefinite, yielding
	$1-\widetilde{\lambda}_i\geq 0,$ and consequently 
	 $\widetilde{\lambda}_i\leq 1.$
	For any $c\in \mathbb{C}^{r_n},$ set $f(z)=\sum_{j=1}^{r_n} c_{j}\phi_{n-j}(z).$ Using the orthonormality of $\{\phi_{j}\}$, we have
	$$c'\widetilde{D}(x)c=c'({\rm I}_{r_n}-\widetilde{M}_{j, k}^{(1,n)}(x))c=||c||^2-\int_{ \widetilde{A}(x)}|f|^2 \; d^2 z=\int_{\mathbb{C}\backslash \widetilde{A}(x)}|f|^2\; d^2z\geq 0  .$$
	Thus, $\widetilde{D}(x)$ is positive semidefinite, which implies $\widetilde{\lambda}_i\geq0.$ 
Expanding the product \(\prod_{i=1}^{r_n}(\widetilde{\lambda}_i + w_i)\) as a sum over all subsets of \(\{1,\dots, r_n\}\) and applying \eqref{deterdiff} yields
 \begin{equation}\label{dd1}
 \begin{aligned}
 |\det ( \widehat{D}(x))-\det( \widetilde{ D}(x))| &= \Bigl| \sum_{i=1}^{r_n} w_i \prod_{j \neq i}\widetilde{\lambda}_j  + \sum_{i < j} w_i w_j \prod_{k \neq i,j} \widetilde{\lambda}_k + \cdots + \prod_{i=1}^{r_n} w_i \Bigr| \\
&\leq \sum_{i=1}^{r_n} |w_i| + \sum_{i < j} |w_i||w_j| + \cdots + \prod_{i=1}^{r_n} |w_i| \\
 &= \bigl(1+\sum_{i=1}^{r_n} |w_i|\bigr)^{r_n} - 1.
 \end{aligned}
\end{equation}
Both $\widetilde{D}(x)$ and $\widehat{D}(x)$ are symmetric, Weyl's inequality together with \eqref{Fhs} and  \(r_n = \lfloor s_n^{1/10}\rfloor\) gives
	$$|w_i|=|\widehat{\lambda}_i-\widetilde{\lambda}_{i}|\leq \|F(x)\|_{\text{HS}}\lesssim\frac{r_n}{s_n^{9/10}}\ll 1.$$
Consequently, 
	$$(1+\sum_{i=1}^{r_n} |w_i|)^{r_n}-1=e^{r_n\log (1+\sum_{i=1}^{r_n} |w_i|)}-1\asymp r_n \sum_{i=1}^{r_n} |w_i|\lesssim \frac{r_n^2}{s_n^{9/10}}=s_n^{-7/10}\ll 1,$$ which is putting back to \eqref{dd1} to ensure 
	\[
	|\det(\widehat{ D}(x)) - \det(\widetilde{D}(x))| \ll 1.
	\]
	This finishes the proof.

 \section{Verification of the continuous transition}\label{sec:continuous trans}    
 In this section, we provide the verification of the continuous transition of  $\Phi_{\alpha}$ and $\widetilde{\Phi}_{\alpha}$ for $\alpha\to\infty$ and $\alpha\to 0^+.$
 \subsection{Proof of the Continuous Transition of $\Phi_{\alpha}$} Recall 
  $$\Phi_{\alpha}(x)=\prod_{j=1}^{+\infty}\Phi(v_{\alpha}(j, x)),$$
  where $v_{\alpha}(j, x)=\frac{j-1}{\sqrt{\alpha}}+a+b x$ and 
  \begin{equation} \label{abform}a = \sqrt{\log(\alpha + 1)} - \frac{\log\big(\sqrt{2\pi} \log\big(\alpha + e^{1/\sqrt{2\pi}}\big)\big)}{\sqrt{\log(\alpha + e)}}, \quad 
b = \frac{1}{\sqrt{\log(\alpha + e)}}.\end{equation} To emphasize the dependence of $a$ and $b$ on $\alpha,$ denote $a=a(\alpha)$ and $b=b(\alpha).$ 
\subsubsection{The case where $\alpha\to 0^+$}

Given that $0<\alpha\ll 1,$ then when $|x|\leq \alpha^{-1/10}$, we have 
  $$v_{\alpha}(j, x)\gg 1, \quad \,\forall \;j\ge 2.$$ Thus, Mills's ratio again implies 
  $$1-\Phi(v_{\alpha}(j, x))=\frac{1+o(1)}{\sqrt{2\pi} v_{\alpha}(j, x)}\exp(-\frac{1}{2}v_{\alpha}^2(j, x)),$$ whence it follows from Lemma \ref{sum} that 
 
  $$\aligned \sum_{j=2}^{+\infty} (1-\Phi(v_{\alpha}(j, x)))
  &\lesssim \sum_{j=2}^{+\infty}\frac{1}{v_{\alpha}(j, x)}\exp(-\frac{1}{2}v_{\alpha}^2(j, x))\\
 & \lesssim \frac{1}{v_{\alpha}(2, x)}\exp(-\frac{1}{2}v_{\alpha}^2(2, x))\ll \sqrt{\alpha}e^{-\frac{1}{3\alpha}}
  .
  \endaligned $$
The last inequality holds because $ v_{\alpha}(2, x)=\alpha^{-1/2}(1+o(1))$ for $|x|\leq \alpha^{-1/10}.$ Thus,
  	\begin{equation}\label{56}
  	\prod_{j=2}^{+\infty}\Phi(v_{\alpha}(j, x))=\exp\{-\sum_{j=2}^{+\infty} (1-\Phi(v_{\alpha}(j, x)))
  	(1+o(1))\}=1+o(\sqrt{\alpha}e^{-\frac{1}{3\alpha}}).
  	\end{equation} 
  	 Similarly as \eqref{alpha0}, we have from \eqref{abform} that 	 
  $$	 a(\alpha) = \sqrt{\alpha}+O(\alpha)\quad \text{ and} \quad 
  	b(\alpha) =1+O(\alpha),$$
 which implies $$v_\alpha(1,x)=a(\alpha)+b(\alpha)\;x=x+\sqrt{\alpha}+O((|x|+1)\alpha)$$
  	  uniformly on $|x|\leq \alpha^{-1/10}.$ We apply again Taylor's expansion to obtain 
  	  $$\aligned\Phi(v_\alpha(1,x))
  	  &=\Phi(x)+\sqrt{\alpha}\phi(x)+O(\alpha),
  	  \endaligned$$
  	  where the last equality is due to the boundedness of $x\phi(x).$ Taking account of \eqref{56}, we get
  	  $$\Phi_{\alpha}(x)=\Phi(x)+\sqrt{\alpha}\phi(x)+O(\alpha). $$
  	By analogy with equations \eqref{01} and \eqref{02}, we  conclude that
  	 $$\sup_{|x|>\alpha^{-1/10}}|\Phi_{\alpha}(x)-\Phi(x)|\lesssim1-\Phi(\alpha^{-1/10})\lesssim \alpha^{1/10}\exp(-\frac{1}{2\alpha^{1/5}})\ll \sqrt{\alpha}.$$
  	 Therefore,
  	  $$\sup_{x\in\mathbb{R}}|\Phi_{\alpha}(x)-\Phi(x)|=\sup_{x\in\mathbb{R}}\sqrt{\alpha}\phi(x)(1+o(1))=\sqrt{\frac{\alpha}{2\pi}}(1+o(1)).$$
  \subsubsection{{The case $\alpha\to+\infty$ }}	  Since $a(\alpha)\gg 1,$ then for any
  	  $j\geq 1$ and $|x|\leq 2\log\log\alpha,$ we have
  	  $$v_\alpha(j, x)\ge  v_{\alpha}(1, x)\gg 1 ,$$ 
  	  which, together with $\log(1+t)=t+O(t^2)$ for $|t|$ small enough, implies
  	 	\begin{equation}\label{proPhi}
  	 \aligned	\prod_{j=1}^{+\infty}\Phi(v_{\alpha}(j, x))&=\exp\{-(1+O(\Psi(v_{\alpha}(1, x))))\sum_{j=1}^{+\infty} \Psi(v_{\alpha}(j, x))
  	 	\}\\
  	 	&= \exp\{-(1+O(\Psi(v_{\alpha}(1, x))))\sum_{j=1}^{+\infty}\frac{1}{\sqrt{2\pi}v_{\alpha}(j, x)}\exp(-\frac{1}{2}v_{\alpha}^2(j, x))\}.
  	 	\endaligned
  	 \end{equation} 
  	  Lemma \ref{sum} ensures that 
  	  $$ \sum_{j=1}^{+\infty}\frac{1}{\sqrt{2\pi}v_{\alpha}(j, x)}\exp(-\frac{1}{2}v_{\alpha}^2(j, x))=\frac{\sqrt{\alpha}(1+O((\log  \alpha)^{-1}))}{\sqrt{2\pi}v^{2}_{\alpha}(1, x)}\exp(-\frac{1}{2}v_{\alpha}^2(1, x)).$$
  	  
  	  Let $\ell_{\alpha}=\log(\sqrt{2\pi} \log(\alpha + e^{\frac1{\sqrt{2\pi}}})).$ Similar to the equations \eqref{e4} and \eqref{e3}, the following expressions for $v_\alpha(1,x)$ are given:
  	  	\begin{equation}
  	  	\aligned v_\alpha(1,x)&=\sqrt{\log (\alpha+e)}+\frac{x-\ell_{\alpha}}{\sqrt{\log (\alpha+e)}}+O((\log \alpha)^{-1/2}\alpha^{-1}) ;\\
  	  		  	v_\alpha^2(1, x)&=\log (\alpha+e)-2\ell_{\alpha}+2 x+\frac{(x-\ell_{\alpha})^{2}}{\log (\alpha+e)}+O(\alpha^{-1});
  	  	\\
  	  	\exp(-\frac{1}{2}v_\alpha^2(1, x))
  	  	&=\sqrt{\frac{2\pi  }{\alpha}} \log \alpha \exp(-x-\frac{(\ell_{\alpha}-x)^{2}}{2\log(\alpha+e)})(1+O(\alpha_n^{-1})). \endaligned 
  	  \end{equation}
Now the expression inside the exponential of \eqref{proPhi} denoted by $\beta_{\alpha}(x)$ satisfies  $$ \beta_{\alpha}(x)=\frac{1+O((\log  \alpha)^{-1})}{(1+\frac{x-\ell_{\alpha}}{\log (\alpha+e)})^{2}}\exp(-x-\frac{(x-\ell_{\alpha})^2}{2\log  (\alpha+e)}),$$
whence $$\aligned |\beta_\alpha(x)+e^{-x}|=\frac{e^{-x}(1+O((\log \alpha)^{-1}))}{(1+\frac{x-\ell_{\alpha}}{\log (\alpha+e)})^{2}}|\frac{(x-\ell_{\alpha})^2}{2\log(\alpha+e)}+\frac{2(x-\ell_{\alpha})}{\log(\alpha+e)}|=o(1) \endaligned $$
  	  uniformly on $|x|\leq 2\log\log \alpha.$
  	  The same calculus as in Theorem 1 for the case $\alpha=+\infty$ yields  $$\sup_{x\in \mathbb{R}}|\Phi_{\alpha}(x)-e^{-e^{-x}}|= \frac{(\log \log \alpha)^{2}}{2e\log \alpha}(1+o(1))\quad \text{for $\alpha\gg 1$}.$$  
  	  The proof is then completed.

 \subsection{Proof of the Continuous Transition of $\widetilde{\Phi}_{\alpha}$}
 Recall $$\widetilde{\Phi}_{\alpha}(x)={\rm det}(\mathrm{I}-\widetilde{M}(x, \alpha)),$$
 where $\widetilde{M}(x, \alpha)=(\widetilde{M}_{j, k}(x, \alpha))_{j, \, k\ge 1}$ with 
 \begin{equation}\label{Mjkxalpha}
 	\widetilde{M}_{j, k}(x, \alpha)=\exp(-\frac{(j-k)^2}{4\alpha})\frac{2}{\pi}\int_0^{\frac{\pi}{2}}\cos((j-k)\theta)\Psi(\widetilde{v}_{\alpha}(\frac{j+k}2, x)-\sqrt{\alpha}\log \cos^2\theta)d\theta\end{equation} for $j, k\ge 1$ with $j-k$ even and zero if $j-k$ odd. 
  Here, $\widetilde{v}_{\alpha}(j, x)=\frac{j-1}{\sqrt{\alpha}}+\widetilde{a}(\alpha)+\widetilde{b}(\alpha) x$ with $\widetilde{b}(\alpha) = \frac{\sqrt{2}}{\sqrt{\log(\alpha+e^2) }}$ and 
$$
\widetilde{a}(\alpha) =\sqrt{\frac{\log(\alpha+1) }{2}} - \frac{\sqrt{2}\big(\log (2^{-\frac{3}{4}}\pi)+\frac{5}{4}\log \log(\alpha+e^{2^{3/5}\pi^{-4/5}})\big)}{\sqrt{\log(\alpha +e^{2})}}.
$$
   
\subsubsection{{For the case $\alpha\to 0^+$}} 
First, the expression \eqref{M1norm} guarantees that $\widetilde{M}(x, \alpha)$ is trace class uniformly for $0< \alpha\le 1.$ 
Hence, the elementary property of Fredholm determinant, together with the fact $\widetilde{M}_{j, k}(x, \alpha)$ is continuous on $\alpha,$  implies 
$$\lim_{\alpha\to 0^+}{\rm det}(\mathrm{I}-\widetilde{M}(x, \alpha))={\rm det}(\mathrm{I}-\lim_{\alpha\to 0^+}\widetilde{M}(x, \alpha)).$$   
Now, we check the infinite dimensional matrix $\widetilde{M}(x, 0):=\lim_{\alpha\to 0^+}\widetilde{M}(x, \alpha).$ 
We see clearly from \eqref{Mjkxalpha}  that 
$\lim_{\alpha\to 0^+}\widetilde{M}_{j, k}(x, \alpha)=0$ if $j\neq 1$ or $k\neq 1$ and $$\lim_{\alpha\to 0^+}\widetilde{M}_{1, 1}(x, \alpha)=\lim_{\alpha\to 0^+}\Psi(\widetilde{a}(\alpha)+\widetilde{b}(\alpha) x)=\Psi(x).$$
Here, we use the fact that $\lim_{\alpha\to 0^+}\widetilde{a}(\alpha)=0$ and $\lim_{\alpha\to 0^+}\widetilde{b}(\alpha)=1.$ Thus,
 $${\rm det}(\mathrm{I}-\widetilde{M}(x, 0))=1-\widetilde{M}_{1, 1}(x, 0)=1-\Psi(x)=\Phi(x).$$ 
 That is  
 $$\lim_{\alpha\to 0^+}{\rm det}(\mathrm{I}-\widetilde{M}(x, \alpha))=\Phi(x)$$ 
 for any $x\in\mathbb{R}.$ 
  \subsubsection{{The case $\alpha\to+\infty$}}  

Using the substitution $t=-\log\cos^2\theta$ in  \eqref{Mjkxalpha} gives  
 	\begin{equation}\label{1-M}
 		\aligned
 		\widetilde{M}_{j, j}(x, \alpha)&=\frac{1}{\pi}\int_{0}^{+\infty}\frac{\Psi(\widetilde{v}_{\alpha}(j,x)+\sqrt{\alpha}t)}{\sqrt{e^t-1}}dt.
 		\endaligned
 	\end{equation}
Comparing this expression with the expression \eqref{widecn}, one sees that the term $g_n(j, x, t)$ in \eqref{widecn} is replaced by  $\Psi(\widetilde{v}_{\alpha}(j,x)+\sqrt{\alpha}t).$ While examining the proof of \eqref{cn1}, we know the key ingredient is approximating $g_n(j, x, t)$ by
$\frac{e^{-\frac{h^{2}_n(j, x,t)}{2}}}{\sqrt{2\pi}h_n(j, x,t)},$ which indeed could be regarded as $\Psi(h_n(j, x,t)).$ Therefore, following the same reasoning line for \eqref{cn1}, we derive similarly 
  \begin{equation}\label{Mjjinf}\aligned
 \widetilde{M}_{j, j}(x, \alpha)
 &=\frac{\exp(-\frac12\widetilde{v}_{\alpha}^2(j, x))}{\sqrt{2}\pi \, \alpha^{1/4} \, \widetilde{v}_{\alpha}^{\,3/2}(j, x)}(1+o(1)).	\endaligned 
 \end{equation}
 The asymptotic identities \eqref{finalreal} and \eqref{Trwidefinal} work for $\alpha $ large enough to guide \begin{equation}\label{sumMjjinf} \sum_{j=1}^{\infty}\frac{e^{-\frac{\widetilde{v}^{2}_{\alpha}(j, x)}{2}}}{\sqrt{2}\pi \alpha^{1/4}\widetilde{v}^{3/2}_{\alpha}(j, x)}=\frac{\alpha^{1/4}(1+o(1))}{\sqrt{2}\pi\widetilde{v}^{5/2}_{\alpha}(1, x)}\exp(-\frac{1}{2}\widetilde{v}^{2}_{\alpha}(1, x))=e^{-x}+o(1),\end{equation} 
which is  put into \eqref{Mjjinf} to bring   \begin{equation}\label{tracealphainf} {\rm Tr}(\widetilde{M}(x, \alpha))=\sum_{j=1}^{+\infty}\widetilde{M}_{j, j}(x, \alpha)= (1+o(1)) \, e^{-x}\end{equation} 
 as $\alpha\to+\infty, $  which eventually leads 
 $$\lim_{\alpha\to+\infty} {\rm det}({\mathrm I}-\widetilde{M}(x, \alpha))=\exp(-\lim_{n\to\infty}{\rm Tr}(\widetilde{M}(x, \alpha)))=\exp(-e^{-x})$$ 
 once $$\|\widetilde{M}(x, \alpha)\|_{\rm HS}^2=\sum_{j, k}\widetilde{M}_{j, k}^2(x, \alpha)\ll 1.$$ 
  
In fact, Lemmas \ref{Mjjlem} and \ref{M0infty} give 
  
$$\widetilde{M}_{j, k}^2(x, \alpha)\le \exp(-\frac{(j-k)^2}{2\alpha}) \widetilde{M}^2_{\frac{j+k}{2}, \frac{j+k}{2}}(x, \alpha)$$ and then \eqref{Mjjinf} enhances this upper bound as  
$$
\widetilde{M}_{j, k}^2(x, \alpha)\lesssim  \frac{1}{\sqrt{\alpha}\widetilde{v}_{\alpha}^3(\frac{j+k}{2}, x)}\exp(-\widetilde{v}_{\alpha}^2(\frac{j+k}{2}, x))\exp(-\frac{(j-k)^2}{2\alpha}).$$
Using the substitution $p=\frac{j+k}2$ and $q=\frac{j-k}{2}$ again, we have
$$\aligned \sum_{j, k} \widetilde{M}_{j, k}^2(x, \alpha)&\lesssim \frac{1}{\sqrt{\alpha}\widetilde{v}_{\alpha}^3(1, x)}\sum_{q=0}^{+\infty}\exp(-\frac{q^2}{2\alpha})\sum_{p=1}^{+\infty}\exp(-\widetilde{v}_{\alpha}^2(p, x)).\endaligned $$ 
Applying Lemma \ref{sum} on the two summations, we see 
\begin{equation}\aligned
\sum_{j, k}\widetilde{M}_{j, k}^2(x, \alpha)&\lesssim \frac{\sqrt{\alpha}}{\widetilde{v}_{\alpha}^3(1, x)}\exp(-\widetilde{v}_{\alpha}^2(1, x))\ll 1,
	\endaligned
\end{equation}
where we use the fact that $$\exp(-\widetilde{v}_{\alpha}^2(1, x))\asymp \exp(-\widetilde{a}^2 (\alpha))\asymp \exp(-\frac{\log \alpha}{2})=\alpha^{-1/2}.$$
The proof is completed now. 
\subsubsection{The Convergence Rate of $\widetilde{\Phi}_{\alpha}$} In verifying the continuous transition of \(\widetilde{\Phi}_{\alpha}\), we did not focus on the details needed to capture the convergence rate. We now briefly explain the approach and state the result directly.

For the case \(\alpha \to +\infty\), \eqref{sumMjjinf}, together with the definition of \(\widetilde{v}^{2}_{\alpha}(1, x)\), yields
\[
\operatorname{Tr}(\widetilde{M}(x, \alpha)) = \exp\Bigl\{-x - \frac{(\widetilde{\ell}_{1}(\alpha)-x)^{2}}{\log\alpha}\Bigr\}\Bigl(1 + O\Bigl(\frac{\log\log\alpha}{\log\alpha}\Bigr)\Bigr),
\]
where \(\widetilde{\ell}_{1}(\alpha) = \frac{5}{4}\log \log(\alpha_n + e^{2^{3/5}\pi^{-4/5}})\). Moreover, a more accurate upper bound,
\[
\|\widetilde{M}(x, \alpha)\|_{\rm HS} \ll e^{-x} \frac{(\log\log \alpha)^2}{\log \alpha},
\]
holds uniformly on some interval, which can be established by an argument analogous to that used in the Appendix for \(\|\widetilde{M}^{(n)}(x)\|_{\rm HS}\). Consequently, by following the same reasoning as in Section 3 for \(\mathbb{P}(\widetilde{X}_n \le x)\) in the regime \(\alpha_n \to +\infty\), we obtain the convergence rate
\[
\lim_{\alpha\to+\infty}\frac{\log\alpha}{(\log\log\alpha)^2}\sup_{x\in\mathbb{R}}\bigl|\widetilde{\Phi}_{\alpha}(x)-\exp(-e^{-x})\bigr| = \frac{25}{16e}.
\]

For the case \(\alpha \to 0^+\), a more refined analysis based on \eqref{1-M} yields
\[
\widetilde{M}_{1, 1}(x, \alpha) = \Psi(x) - \frac{\sqrt{2}+4\ln2}{2}\sqrt{\alpha}\,\phi(x) + O(\alpha),
\]
along with the estimates
\[
\begin{aligned}
&\sum_{j\neq k\ge 1}\frac{\widetilde{M}_{j, k}^2(x, \alpha)}{(1-\widetilde{M}_{j, j}(x, \alpha))(1-\widetilde{M}_{k, k}(x, \alpha))} \lesssim \alpha^{\frac32}, \\
&\prod_{j=2}^{+\infty}(1-\widetilde{M}_{j, j}(x, \alpha)) = 1 + O(\alpha).
\end{aligned}
\]
Therefore,
\[
\widetilde{\Phi}_{\alpha}(x) = \det\bigl({\rm I} - \widetilde{M}(x, \alpha)\bigr) = \Phi(x) + \frac{\sqrt{2}+4\ln2}{2}\sqrt{\alpha}\,\phi(x) + O(\alpha).
\]
While the above argument is presented for fixed \(x\in\mathbb{R}\), it remains valid on a suitable central interval, leading to the uniform convergence rate
\[
\lim_{\alpha\to 0^+}\frac{1}{\sqrt{\alpha}}\sup_{x\in\mathbb{R}}\bigl|\widetilde{\Phi}_{\alpha}(x) - \Phi(x)\bigr| = \frac{\sqrt{2}+4\ln2}{2\sqrt{2\pi}}.
\]
The whole proof is completed. 

\section{Appendix}\label{sec:appendix} 
In this section, we provide proofs of some key equations, lemmas and remarks. 
\subsection{Proof of the Lemmas in section \ref{sec:alphainfinity}} We first give the proofs of the lemmas in the third section. 
\subsubsection{Proof of \eqref{equaforsumn} and \eqref{a83}}
Recall that
	$$u_n(j, x) = \frac{j-1}{\sqrt{\alpha_n}} + a_n + b_n x,$$
$\alpha_n\gg 1$	and with $\ell := n-j+1$,
	$$A_{\pm \epsilon}:=\{\sum_{r=1}^{k_n}\frac{S_{\ell,r}-\ell}{\ell}>k_n(\psi (n)-\log \ell)+\frac{a_n+b_nx}{\sqrt{\alpha_n}}\pm\varepsilon\}.$$
We now prove the following estimates. 
	For any $x$ and $j$ such that $1 \ll u_n(j, x) \ll n^{1/6},$ 
	$$\mathbb{P} (A_{\pm\alpha_n^{-3/5}})\le \frac{1}{u_n(j,x)}e^{-\frac{3u^{2}_n(j, x)}{8}}. 
 $$
	Furthermore, for $x$ and $j$ such that $1\ll u_n(j, x)\lesssim \sqrt{\log \alpha_n},$  
	\begin{equation}\label{z}
	\mathbb{P} (A_{\pm\alpha_n^{-3/5}})
		=\frac{1+O(u_n^{-2}(j, x))}{\sqrt{2\pi}u_n(j, x)}e^{-\frac{u^{2}_n(j, x)}{2}}. 
	\end{equation}
Once these are established, equations \eqref{equaforsumn} and \eqref{a83} follow directly under their respective conditions.

	Let $\{\xi_i\}$ be i.i.d. random variables obeying an exponential distribution with parameter  $1$ and then $$\sum_{r=1}^{k_n} S_{\ell, \,r}\stackrel{d}{=}\sum_{i=1}^{\ell k_n}\xi_i.$$   
	It follows that
	$$\aligned 
	&\mathbb{P} ( \sum_{r=1}^{k_n} (\frac{S_{\ell,r}}{\ell}-1) >k_n(\psi (n)-\log \ell)+\frac{a_n+b_n x}{\sqrt{\alpha_n}}\pm \alpha_n^{-\frac{3}{5}})=\mathbb{P} (\frac{\sum_{i=1}^{\ell k_n} (\xi_i-1)}{\sqrt{\ell k_n}}  >\widehat{u}_n(j,x)),
	\endaligned $$
	where $$ \widehat{u}_n(j,x)=\sqrt{\ell k_n} (\psi (n)-\log \ell) + \sqrt{\frac{\ell}{n}}(a_n+b_n x) \pm\sqrt{\frac{\ell}{n}} \alpha_n^{-\frac1{10}}.$$ 
	Since $j\ll n^{2/3}$ and $k_n\ll n,$ we know $\sqrt{\ell}=\sqrt{n}(1+O((j-1) n^{-1}))$ and then  we have 
	\begin{equation}\label{tildeuu}\aligned 
		\widehat{u}_n(j, x)&=(\frac{j-\frac{3}{2}}{\sqrt{\alpha_n}} + a_n+b_nx\pm \alpha_n^{-\frac1{10}})(1+O(\frac{m}{n}))\\
		&=u_n(j, x)(1+O(\alpha_n^{-\frac{1}{10}}u_n^{-1}(j,x)+\frac{j-1}{n}))\\
		&=u_n(j, x)(1+o(1)).
		\endaligned \end{equation} 
	For $1\ll u_n(j, x)\ll n^{\frac{1}{6}},$ the same holds for $\widehat{u}_n(j, x).$ The 
	Theorem 1  from \cite{Petrov1975}  entails that
	\begin{equation}\label{prec}\mathbb{P} (\frac{\sum_{i=1}^{\ell k_n} (\xi_i-1)}{\sqrt{\ell k_n}}  >\widehat{u}_n(j, x))=(1-\Phi(\widehat{u}_n(j, x)))(1+O((nk_n)^{-1/2}u_n^{3}(j, x))).\end{equation}
	Recall the Mills ratio
	$$	 1 - \Phi(t) = \frac{1}{\sqrt{2\pi}\,t} e^{-t^2/2} \left(1 + O(t^{-2})\right)$$ for $t\gg 1.$
	Eventually, uniformly on $1 \ll u_n(j, x) \ll n^{1/6},$ we have from \eqref{tildeuu} and \eqref{prec}  that 
	$$\aligned \mathbb{P} (\frac{\sum_{i=1}^{\ell k_n} (\xi_i-1)}{\sqrt{\ell k_n}}  >\widehat{u}_n(j,x))&=\frac{1+O(u_n^{-2}(j, x)+(nk_n)^{-1/2}u_n^{3}(j,x))}{\sqrt{2\pi}\widehat{u}_n(j,x)}e^{-\frac{\widehat{u}^{2}_n(j, x)}{2}}\\
	&\leq \frac{e^{-\frac{3u^{2}_n(j, x)}{8}}}{u_n(j,x)}.\endaligned$$
	Particularly, if $1\ll u_n(j, x)\lesssim \sqrt{\log \alpha_n},$ which implies $j\lesssim \sqrt{\alpha_n\log\alpha_n},$ then $$u_n^{2}(j,x)(\alpha_n^{-\frac{1}{10}}u_n^{-1}(j,x)+\frac{j-1}{n})\lesssim u_n(j, x) \alpha_n^{-1/10}=o(1).$$ Thereby, 
	$$\exp(-\frac{1}{2}\widehat{u}_n^2(j, x))=\exp(-\frac12 u_n^2(j, x))(1+O(u_n(j, x)\alpha_n^{-1/10}))$$ and 
	$$\frac{1+O(u_n^{-2}(j, x)+(nk_n)^{-1/2}u_n^{3}(j,x))}{\widehat{u}_n(j,x)}=\frac{1+O(u_n^{-2}(j, x))}{u_n(j, x)}.$$ 
	Thus, we have 
	$$\mathbb{P}(\frac{\sum_{i=1}^{\ell k_n} (\xi_i-1)}{\sqrt{\ell k_n}}  >\widehat{u}_n(j,x))=\frac{1+O(u_n^{-2}(j, x))}{\sqrt{2\pi} u_n(j, x)} e^{-\frac{1}{2}u_n^2(j, x)}.$$
	The proof is then completed.  
	
	\subsubsection{Proof of Lemma \ref{int}}
We are going to prove that  
		$$ \int_{0}^{v}\frac{1}{\sqrt{s}(h+s)}e^{-\frac{(h+s)^{2}}{2}}ds=\sqrt{\frac{\pi}{ h^{3}}}e^{-\frac{ h^{2}}{2}}(1+O(h^{-2}))$$
for any $h, v$ satisfying  $1\ll h$ and $\frac{4}{ h}\log h\le v.$ 

Dominating the integrand by $v^{-3/2} \exp(-\frac{(h+s)^2}{2})$ when $s\ge v,$ it is ready to see from $1\ll h$ that 
\begin{equation}
\label{righte}	
\int_{v}^{+\infty}\frac{1}{\sqrt{s}(h+s)}e^{-\frac{(h+s)^{2}}{2}}ds\le v^{-3/2} \int_{v+h}^{+\infty} e^{-\frac{t^2}{2}}dt\lesssim v^{-3/2}(v+h)^{-1} e^{-(v+h)^2/2}.\end{equation}
Now we consider the integral on $[0, +\infty),$ which can be rewritten as $$ \aligned \int_{0}^{+\infty}\frac{1}{\sqrt{s}(h+s)}e^{-\frac{(h+s)^{2}}{2}}ds&=h^{-1/2} e^{-h^2/2}\int_{0}^{+\infty}\frac{1}{\sqrt{t}(1+t)}e^{-h^2t} e^{-h^2t^2/2}dt. \endaligned $$
The elementary inequality tells 
$$1-\frac{h^2t^2}{2}-t \le \frac{1}{1+t} e^{-h^2t^2/2}\le 1$$ for any $t>0,$ whence 
$$\int_{0}^{+\infty}\frac{1}{\sqrt{t}}e^{-h^2t}(1-t-\frac{h^2t^2}{2})dt\le \int_{0}^{+\infty}\frac{1}{\sqrt{t}(1+t)}e^{-h^2t} e^{-t^2/2}dt\le \int_{0}^{+\infty}\frac{1}{\sqrt{t}}e^{-h^2t} dt.$$
The property  of Gamma function derives 
	$$\int_{0}^{+\infty}\frac{1}{\sqrt{t}}e^{-h^2 t}dt=h^{-1}\Gamma(\frac{1}{2})=\frac{\sqrt{\pi}}{h}$$
	and similarly 
	$$\int_{0}^{+\infty}(\frac12 h^2t^{3/2}+\sqrt{t})e^{-h^2 t} dt=\frac{7\sqrt{\pi}}{8h^3}. $$
Thus, we have 
$$\int_{0}^{+\infty}\frac{1}{\sqrt{t}(1+t)}e^{-h^2t} e^{-t^2h^2/2}dt=\frac{\sqrt{\pi}}{h}(1+O(h^{-2})),$$ which ensures 
\begin{equation}
\label{lefte}	
\int_{0}^{+\infty}\frac{1}{\sqrt{s}(h+s)}e^{-\frac{(h+s)^{2}}{2}}ds=\frac{\sqrt{\pi}}{h^{3/2}} e^{-\frac{h^2}{2}}(1+O(h^{-2})).\end{equation}
Leveraging \eqref{righte} and \eqref{lefte}, once \begin{equation}
\label{conditionvh} v^{-3/2}(v+h)^{-1} e^{-(v+h)^2/2}\lesssim (h^{-3/2}e^{-h^2/2}) h^{-2},\end{equation} the proof is completed. The requirement  \eqref{conditionvh} is verified because the conditions $\frac{4}{h}\log h\le v$ and $1\ll h$ indicate
$$\log \frac{v^{-3/2}(v+h)^{-1} e^{-(v+h)^2/2}}{h^{-7/2}e^{-h^2/2}}\lesssim \frac{3}{2}\log (v^{-1})+\frac{5}{2}\log h-h v\le 4\log h-\frac32\log\log h-h v,$$ 
which tends to $-\infty$. The proof is complete. 

\subsubsection{Proof of Lemma \ref{Mhslem}} 
Recall the context of Lemma \ref{Mhslem}:  
\begin{equation}\label{wideMHSup}
\|\widetilde{M}^{(n)}(x)\|_{\rm HS}^2 = \sum_{1\le j,k\le n} \bigl(\widetilde{M}_{j,k}^{(n)}(x)\bigr)^2 \ll e^{-x} \frac{(\log\log \alpha_n)^2}{\log\alpha_n}
\end{equation}
holds uniformly for \(-\widetilde{\ell}_{1,\infty}(n) \le x \le \widetilde{\ell}_{2,\infty}(n)\), where  
\[
\widetilde{\ell}_{1,\infty}(n) = \frac12 \log\log\alpha_n,\qquad 
\widetilde{\ell}_{2,\infty}(n) = \frac54 \log\log\bigl(\alpha_n + e^{2^{3/5}\pi^{-4/5}}\bigr),
\]  
and \(\alpha_n \gg 1\).

With
\[
j_n = \Bigl\lfloor \frac15 \sqrt{\alpha_n \log \alpha_n}\Bigr\rfloor,\qquad 
t_n = \bigl\lfloor 8\sqrt{\alpha_n \log \alpha_n}\bigr\rfloor,
\]  
we cut the summation range in \eqref{wideMHSup} into several parts to obtain the upper bound.

Use the substitution $q:=\frac{j-k}{2}$ and $p:=\frac{j+k}{2}.$ Let $S_1$ be the sum in \eqref{wideMHSup} for $p\ge t_n,$ and the monotonicity of $\widetilde{M}^{(n)}_{j, j}(x)$ on $j$ and the estimate in \eqref{cn3} give
\begin{equation}\label{s1}\aligned 
	{\rm S}_{1}\le n^2( \widetilde{M}^{(n)}_{t_n, t_n}(x))^2\leq n^{-6}\ll e^{-2x}\frac{(\log\log \alpha_n)^4}{(\log \alpha_n)^2}. 
	\endaligned 
\end{equation}
Now $S_2$ is the corresponding sum when $j_n+1\le p<t_n,$ and we have   
by \eqref{cn2} and Lemma \ref{sum} that  
$$\aligned 
	{\rm S}_{2}&\lesssim \sum_{p=j_n+1}^{t_n}p \big(\frac{e^{-\frac{3\widetilde{u}^{2}_n(p, x)}{4}}}{\alpha_n^{1/2}\widetilde{u}^{3}_n(p, x)}+n^{-8/5} \big)\lesssim \sum_{p=j_n+1}^{t_n}\frac{1}{\widetilde{u}^{2}_n(p, x)}e^{-\frac{3\widetilde{u}^{2}_n(p, x)}{4}}+t_nn^{-8/5}\\
	&\lesssim\frac{\sqrt{\alpha_n}}{\widetilde{u}^{2}_n(j_n+1, x)}e^{-\frac{3\widetilde{u}^{2}_n(j_n+1, x)}{4}}+\sqrt{\log\alpha_n}\,\alpha_n^{-11/10}.
	\endaligned 
$$
Here, the second inequality is due to $\widetilde{u}_n(j, x)=\frac{j-1}{\sqrt{\alpha_n}}+\widetilde{a}_n+\widetilde{b}_n x\ge \frac{j-1}{\sqrt{\alpha_n}}.$ Now 
$$\aligned \widetilde{u}_n(j_n+1,x)&=\frac{5\sqrt{2}+2}{10}\sqrt{\log \alpha_n}(1+o(1)); \\
\widetilde{u}^{2}_n(j_n+1, x)&=\frac{54+20\sqrt{2}}{100}\log\alpha_n(1+o(1))\geq \frac{4}{5}\log\alpha_n,
\endaligned  $$
which implies $$\frac{\sqrt{\alpha_n}}{\widetilde{u}^{2}_n(j_n+1, x)}e^{-\frac{3\widetilde{u}^{2}_n(j_n+1,x)}{4}}\le \alpha_n^{-1/10}(\log \alpha_n)^{-1}.$$ 
And also, $\widetilde{\ell}_{2,\infty}(n)=\frac{5}{4}\log \log(\alpha_n+e^{2^{3/5}\pi^{-4/5}})$ helps us to derive 
$$\frac{\alpha_n^{-1/10}(\log \alpha_n)^{-1}}{e^{-2x} \frac{(\log\log \alpha_n)^4}{(\log \alpha_n)^2}}\le \exp(2\widetilde{\ell}_{2,\infty}(n))\frac{\log \alpha_n}{\alpha_n^{1/10}(\log\log\alpha_n)^4}\ll 1.$$
As a direct consequence, it follows 
\begin{equation}\label{s2}
S_2\ll e^{-2x} 	\frac{(\log\log \alpha_n)^4}{(\log \alpha_n)^2}.
\end{equation}
We continue to tear an easier part  defined by 
$$S_3:=\sum_{0\le q< \alpha_n^{2/5}\wedge p, 1\le p\le j_n} (\widetilde{M}_{j, k}^{(n)}(x))^2\le \alpha_n^{2/5}\sum_{p=1}^{j_n} (\widetilde{M}_{p, p}^{(n)}(x))^2.$$
Similarly as \eqref{finalreal}, we have 
$$
	\aligned 
\sum_{p=1}^{j_n} (\widetilde{M}_{p, p}^{(n)}(x))^2\asymp\sum_{j=1}^{j_n}\frac{e^{-\widetilde{u}^{2}_n(j,x)}}{\alpha_n^{1/2}\tilde{u}^{3}_n(j,x)}\asymp
\frac{\exp(-\widetilde{u}^{2}_n(1, x))}{\widetilde{u}^{4}_n(1, x)}, \endaligned 
$$
which guarantees $$\sum_{p=1}^{j_n} (\widetilde{M}_{p, p}^{(n)}(x))^2\asymp e^{-2x}\frac{\sqrt{\log \alpha_n}}{\sqrt{\alpha_n}}. $$
Hence, \begin{equation}\label{s3}S_3\lesssim  e^{-2x}\frac{\alpha_n^{2/5}\sqrt{\log \alpha_n}}{\sqrt{\alpha_n}}\ll e^{-2x}\frac{(\log\log \alpha_n)^4}{(\log\alpha_n)^2}.\end{equation}
The most delicate regime is
\[
\alpha_n^{2/5} \le q <p \le j_n,
\]
where Lemma \ref{Mjjlem} ceases to be effective. In this range, the off-diagonal correlations decay too slowly to be neglected.
To overcome this, we need to be more careful on the integral  for $M_{j, k}^{(n)}(x)$ in Lemma \ref{Mjjlem}. 

Setting $t = -\log \cos^{2}\theta,$ similarly as for $\widetilde{M}^{(n)}_{j, j}(x),$ we have  
 $$\aligned |\widetilde{M}_{j, k}^{(n)}(x)|&\leq |\int_{0}^{+\infty} \frac{\cos(2q\arccos(e^{-t/2}))}{\sqrt{e^{t}-1}}g_n(p, x, t)dt|\\
&\le| \int_{0}^{\alpha_n^{-9/20}}\frac{\cos(2q\arccos(e^{-t/2}))}{\sqrt{e^{t}-1}}g_n(p, x, t)dt|+\int_{\alpha_n^{-9/20}}^{+\infty}\frac{g_n(p, x, t)}{\sqrt{e^{t}-1}}dt.
\endaligned$$ Under the conditions $1\le j \leq j_n$, the second integral is proportional to $\mathrm{I\!I}$ in the proof of \eqref{cn2} and then 
$$\int_{\alpha_n^{-9/20}}^{+\infty}\frac{g_n(p,x,t)}{\sqrt{e^{t}-1}}dt\lesssim \alpha_n^{-4/5}.$$ It suffices to estimate the first integral, denoted by ${\rm I}_1$. 
$$\aligned &\quad |{\rm I}_1-\int_{0}^{\alpha_n^{-9/20}}\frac{\cos(2 q\arccos(e^{-t/2}))}{h_n(p, x, t)\sqrt{e^{t}-1}}\exp(-\frac{1}{2}h_n^2(p, x, t))dt|\\
&\le n^{-4/5}\int_{0}^{\alpha_n^{-9/20}}\frac{|\cos(2q\arccos(e^{-t/2}))|}{\sqrt{e^{t}-1}}dt\le n^{-4/5}\lesssim \alpha_n^{-4/5}.\endaligned $$
Now we have the asymptotic 
$$\arccos(e^{-t/2})=\sqrt{t}+O(t^{3/2})\quad \text{and}\quad \frac{1}{\sqrt{e^{t}-1}}=\frac{1+O(\alpha_n^{-9/20})}{\sqrt{t}}$$ for $0\le t \le \alpha_n^{-9/20}.$
Now, $$q t^{3/2} \le j_n \alpha_n^{-27/40}\lesssim \alpha_n^{-7/40}\sqrt{\log\alpha_n} \ll 1,$$ whence
$$\cos(2q\arccos(e^{-t/2}))=\cos(2q\sqrt{t}+O(q\, t^{3/2}))=\cos(2q\sqrt{t})+O(\alpha_n^{-7/40}\sqrt{\log\alpha_n}). $$
Thus,
\begin{equation}\label{conc}\aligned &\quad|\int_{0}^{\alpha_n^{-9/20}}\frac{\cos(2q\arccos(e^{-t/2}))}{h_n(p, x, t)\sqrt{e^{t}-1}}\exp(-\frac{1}{2}h_n^2(p, x, t))dt|\\
&=|\int_{0}^{\alpha_n^{-9/20}}\frac{\cos(2q \sqrt{t})(1+O(\alpha_n^{-7/40}\sqrt{\log\alpha_n}))}{h_n(p, x, t)\sqrt{t}}\exp(-\frac{1}{2}h_n^2(p, x, t))dt|\\
&\le |\int_{0}^{\alpha_n^{-9/20}}\frac{\cos(2q \sqrt{t})}{h_n(p, x, t)\sqrt{t}}\exp(-\frac{1}{2}h_n^2(p, x, t))dt|\\
&+|O(\alpha_n^{-7/40}\sqrt{\log\alpha_n})|\int_{0}^{\alpha_n^{-9/20}} \frac{\exp(-\frac 12 h_n^2(p, x, t))}{\sqrt{t} h_n(p, x, t)}dt.
\endaligned\end{equation} The fact $  h_n(p, x, t) = \widetilde{u}_n(p, x)+\sqrt{\alpha_n}t\ge\widetilde{u}_n(1, x)$ and \eqref{Imain}  
work together to bring   
\begin{equation}\label{hexph}\aligned \int_{0}^{\alpha_n^{-9/20}} \frac{\exp(-\frac 12 h_n^2(p, x, t))}{\sqrt{t} h_n(p, x, t)}dt\lesssim  \frac{\exp(-\frac{1}{2}\widetilde{u}_n^2(1, x))}{\alpha_n^{1/4}\widetilde{u}_n^3(1, x)}\lesssim \frac{e^{-x}}{\sqrt{\alpha_n} (\log \alpha_n)^{1/4}} \endaligned \end{equation} and then the error term in the last line of  \eqref{conc} is bounded by 
$$\alpha_n^{-7/40}\sqrt{\log \alpha_n} \times \frac{e^{-x}}{\sqrt{\alpha_n} (\log \alpha_n)^{1/4}}=(\log\alpha_n)^{1/4}(\alpha_n)^{-27/40}e^{-x}.$$  
We now bound the integral in the last but second line of \eqref{conc}, denoted by $\mathrm{I}_2.$  
Using the substitution $s = \sqrt{t},$ we have 
$$\aligned \mathrm{I}_2=2\int_{0}^{\alpha_n^{-9/40}}\frac{\cos(2qs)}{h_n(p, x, s^2)}\exp(-\frac{1}{2}h_n^2(p, x, s^2))ds.\endaligned$$
Setting $$	f_n(s) = \frac{1}{h_n(p, x, s^2)}\, \exp(-\frac{1}{2}h_n^2(p, x, s^2))
	$$ and using once the integration by parts formula, 
we derive 
$$	\mathrm{I}_2= \frac{\sin(2qs) f_n(\alpha_n^{-9/40})}{q}  - \frac{1}{q} \int_{0}^{\alpha_n^{-9/40}} \sin(2qs) f_n'(s)\, ds.$$
	Now $$f'_n(s)=-2\sqrt{\alpha_n}s \exp(-\frac{1}{2} h_n^2(p, x, s^2))(1+\frac{1}{h_n^2(p, x, s^2)}) $$ and then 
$$	|f'(s)| \lesssim \sqrt{\alpha_n}\, s \, \exp(-\frac{1}{2} h_n^2(p, x, s^2)).$$
Thus, the fact $h_n(p, x, s^2)=\widetilde{u}_n(p, x)+\sqrt{\alpha_n} s^2$ again helps us to obtain 
$$\aligned	|\mathrm{I}_2|& \lesssim \frac{|f(\alpha_n^{-9/40})|}{q} + \frac{1}{q} \int_0^{\alpha_n^{-9/40}} |f'(s)|\, ds\\
&\lesssim\frac{1}{q \widetilde{u}_n(p, x) }\exp(-\frac12 (\widetilde{u}_n(p, x)+\alpha_n^{\frac{1}{20}})^2)+\frac{1}{q}\int_{0}^{\alpha_n^{-9/40}} \sqrt{\alpha_n} s \exp(-\frac{1}{2} h_n^2(p, x, s^2))ds.
\endaligned	$$
By the variable of change  $y = \widetilde{u}_n(p, x) + \sqrt{\alpha_n} s^2,$ the second integral right above becomes  
$$	\int_{\widetilde{u}_n(p, x)}^{\widetilde{u}_n(p, x) + \alpha_n^{1/20}} e^{-y^2/2}\, dy \le \int_{\widetilde{u}_n(p, x)}^{\infty} e^{-y^2/2}\, dy\lesssim \frac{1}{\widetilde{u}_n(p, x)} e^{-\frac12\widetilde{u}_n^2(p, x)}.$$		
	Therefore,
$$	|\mathrm{I}_2| \lesssim \frac{1}{q\widetilde{u}_n(p, x)} \exp(-\frac12\widetilde{u}_n^2(p, x)).$$
Thus, picking up the two error terms we see 
\begin{equation}\label{s40}
	\aligned S_4:&=\sum_{(j, k)\in D: \alpha_n^{2/5}\le q<p\le j_n} (\widetilde{M}_{j, k}^{(n)}(x))^2\\
	&\lesssim  \sum_{p=\alpha_n^{2/5}}^{j_n}\frac{\exp(-\widetilde{u}_n^2(p, x))}{\widetilde{u}_n^2(p, x)} \sum_{\alpha_n^{2/5}\le q<p } \frac{1}{q^2}+j_n^2 O(\alpha_n^{-8/5}+(\log\alpha_n)^{1/2}(\alpha_n)^{-27/20}e^{-2x}). \endaligned \end{equation} 
The fact $\sum_{\alpha_n^{2/5}\le q<p} q^{-2}\lesssim \alpha_n^{-2/5}$ and Lemma \ref{sum} ensure the first part in the second line of \eqref{s40} is bounded by $$\alpha_n^{-2/5}\sqrt{\alpha_n} \widetilde{u}_n^{-3}(1, x)\exp(-\widetilde{u}_n^2(1, x))\asymp \alpha_n^{-2/5} (\log\log\alpha_n)^{5/2}(\log\alpha_n)^{-3/2}\exp(-2x).$$ Putting this upper bound back into \eqref{s40} and comparing the orders, we get 
\begin{equation}\label{s4}
	\aligned S_4\lesssim  \alpha_n^{-2/5} (\log\log\alpha_n)^{5/2}(\log\alpha_n)^{-3/2}\exp(-2x)\ll \frac{(\log\log \alpha_n)^4}{(\log \alpha_n)^2} e^{-2x}. \endaligned \end{equation} 
Combining \eqref{s1}, \eqref{s2}, \eqref{s3} and \eqref{s4} together, we complete the proof of Lemma \ref{Mhslem}. 	

\subsection{Proof of Lemmas \ref{mgeq1cn} and \ref{traceB2}} 
We provide two important lemmas in the fourth section here. 

\subsubsection{Proof of Lemma \ref{mgeq1cn}} 

     The fact $0<\alpha_n\ll 1$ implies that  $$u_n(j,x):=\frac{j-1}{\sqrt{\alpha_n}}+a_n +b_n x\gg1$$  uniformly on $2\leq j\ll n$ and $x\geq -\sqrt{2\log z_n}.$ Then, we apply the Berry-Esseen bound for the sum of i.i.d random sequence (\cite{Chen2011}) under the condition $k_n\gg 1$ to derive   
    	$${M}^{(n)}_{j, j}(x)\leq u_n^{-1}(j, x)e^{-\frac{u_n^{2}(j, x)}{3}}+O(k_n^{-\frac{1}{2}}u^{-3}_n(j, x)).$$ 
    	First, it follows from the monotonicity of ${M}^{(n)}_{j, j}(x)$ that
    		$$\aligned {M}^{(n)}_{j, j}(x)&\leq {M}^{(n)}_{2, 2}(-\sqrt{2\log z_n})\\
    	&\leq u_n^{-1}(2,-\sqrt{2\log z_n})\exp(-\frac13 u^{2}_n(2,-\sqrt{2\log z_n}))+O(k_n^{-\frac{1}{2}}u^{-3}_n(2,-\sqrt{2\log z_n}))\\
    	&\ll1\endaligned$$
    	uniformly on  $|x|\leq \sqrt{2\log z_n}$ and $2\leq j\leq n.$ Hence,
    	\begin{equation}\label{511}
    		\prod_{j=2}^{n}(1-{M}^{(n)}_{j, j}(x))=\exp\{-(1+o(1))\sum_{j=2}^{n}{M}^{(n)}_{j, j}(x)\}.
    	\end{equation}
    	Next, we prove that $\sum_{j=2}^{n}{M}^{(n)}_{j, j}(x)=o(1).$ Note that 
    $u_n(2,x)=\alpha_n^{-1/2}(1+o(1)),$ which implies together with Lemma \ref{sum} with $\gamma_n=\sqrt{\alpha_n}$ bounded that 
    \begin{equation}\label{67} \aligned 
    \sum_{j=2}^{+\infty}u_n^{-1}(j, x)e^{-\frac{u_n^{2}(j,x)}{3}}&\lesssim u_n^{-1}(2,x)e^{-\frac{u_n^{2}(2,x)}{3}}\lesssim  \sqrt{\alpha_n}\exp(-\frac{\alpha_n^{-1}}{4});
    \\
    \sum_{j=2}^{+\infty}u^{-3}_n(j,x)&\lesssim \sqrt{\alpha_n} u_{n}^{-2}(2,x)\lesssim \alpha_n^{3/2}.\endaligned 
    \end{equation}
   Based on the decay of \({M}^{(n)}_{j, j}(x)\) in \(j\), we separate the sum at \(j = \lfloor n^{5/6} \rfloor\):
   \begin{equation}\label{0c1}
   	\sum_{j=2}^{n} {M}^{(n)}_{j, j}(x) \leq \sum_{j=2}^{\lfloor n^{5/6} \rfloor} {M}^{(n)}_{j, j}(x) \;+\; n \, {M}^{(n)}_{\lfloor n^{5/6} \rfloor, \lfloor n^{5/6} \rfloor}(x).
   \end{equation}
 We then proceed to bound each segment from above.

From the equation \eqref{67}, we know that
\begin{equation}\label{0c2}
	\aligned \sum_{j=2}^{[n^{5/6}]}{M}^{(n)}_{j, j}(x)&\leq\sum_{j=2}^{n^{5/6}}(u_n^{-1}(j, x)e^{-\frac{u_n^{2}(j,x)}{3}}+O(k_n^{-\frac{1}{2}}u^{-3}_n(j,x)))\\
	&\lesssim \alpha_n^{1/2} e^{-\frac{1}{4\alpha_n}}+k_n^{-\frac{1}{2}}\alpha_n^{\frac{3}{2}}.\endaligned
\end{equation}
   	The choice $z_n=n\wedge \alpha_n^{-1/2}$ ($\alpha_n=n/k_n$) ensures $\alpha_n^{1/2}\vee n^{-1}\le z_n^{-1}\ll 1,$ whence  
   	\begin{equation}\label{0c3}
   		\sum_{j=2}^{[n^{5/6}]}{M}^{(n)}_{j, j}(x)\lesssim z_n^{-1}\exp(-\frac{z_n^2}4)+n^{-1/2}\alpha_n^2 \lesssim z_n^{-9/2}.
   	\end{equation}
   	 Now $$u_{n}([n^{5/6}],x)\geq u_{n}([n^{5/6}],-\sqrt{2\log z_n})=n^{1/3}\sqrt{k_n}(1+o(1))$$ and then 
   	 \begin{equation}\label{0c4}
   	 	\aligned n \, {M}^{(n)}_{\lfloor n^{5/6} \rfloor, \lfloor n^{5/6} \rfloor}(x)&\leq  \frac{n^{2/3}}{\sqrt{k_n}}\exp(-\frac{1}{4}n^{2/3}k_n)+O(nk_n^{-\frac{1}{2}}(n^{1/3}\sqrt{k_n})^{-3})\\
   	 	&\lesssim k_n^{-2}=n^{-2}\alpha_n^{2}\le z_n^{-6}.\endaligned
   	 \end{equation}
   	 Consequently, $\sum_{j=2}^{n}{M}^{(n)}_{j, j}(x)=O(z_n^{-9/2})$ and, by \eqref{511},
   	$$\prod_{j=2}^{n}(1-{M}^{(n)}_{j, j}(x))=1+O(z_n^{-9/2}).$$ 
   Analogous to $(1)$ in Lemma \ref{gnmxt}, for \(n\gg j\ge 2\) and \(t\ge0\) we have  
\[
g_n(j, x, t)\lesssim \frac{e^{-\frac{3}{8} h^{2}_n(j, x,t)}}{h_n(j, x,t)}+k_n^{-1/2}h_n^{-3}(j,x,t),
\]  
where \(h_n(j,x,t)=\widetilde{u}_n(j,x)+\sqrt{\alpha_n}t.\)  Since \(\log \cos^2\theta\leq 0\), it follows that  
\[
\widetilde{M}^{(n)}_{j, j}(x)=\mathbb{P}\bigl(\log Y_{n-j+1}+\log \cos^2\Theta\ge k_n\psi(n)+\frac{\widetilde{a}_n+\widetilde{b}_n x}{\sqrt{\alpha_n}}\bigr)\leq g_n(j, x, 0).
\]  
Setting \(t=0\) gives \(h_n(j, x, 0)=\widetilde{u}_n(j, x)\), and hence for \(n\gg j\geq 2\),  
\begin{equation}\label{widecnup}
\widetilde{M}^{(n)}_{j, j}(x)\lesssim\frac{e^{-\frac{3}{8} \widetilde{u}^{2}_n(j, x)}}{\widetilde{u}_n(j,x)}+k_n^{-1/2}\widetilde{u}_n^{-3}(j,x)\ll 1.
\end{equation}  
By an argument similar to the one leading to \eqref{511}, we obtain the asymptotic product representation  
\[
\prod_{j=2}^{n}(1-\widetilde{M}^{(n)}_{j, j}(x))=\exp\bigl\{-(1+o(1))\sum_{j=2}^{n}\widetilde{M}^{(n)}_{j, j}(x)\bigr\}.
\]
Turning to the estimation of the sum, we recall from the definition that  
\[
\widetilde{u}_n(j,x)=u_n(j,x)(1+o(1))=(j-1)\alpha_n^{-1/2}(1+o(1)).
\]  
Now, performing the same summation procedure as in \eqref{0c1}-\eqref{0c4} but with \(\widetilde{u}_n(j, x)\) in place of \(u_n(j, x)\), we obtain  
\(
\sum_{j=2}^{n} \widetilde{M}^{(n)}_{j, j}(x) = O\bigl(z_n^{-9/2}\bigr).
\)  
Inserting this into the exponential representation and using \(e^{-y}=1+O(y)\) for \(y=O(z_n^{-9/2})\to0\), we conclude  
\[
\prod_{j=2}^{n}(1-\widetilde{M}^{(n)}_{j, j}(x))=1+O(z_n^{-9/2}).
\] 
	
	\subsubsection{Proof of Lemma \ref{traceB2}} 
	Recall our task is to prove 

     	$$\sum_{j< k}\frac{(\widetilde{M}_{j,k}^{(n)}(x))^2}{(1-\widetilde{M}_{j,j}^{(n)}(x))(1-\widetilde{M}_{k,k}^{(n)}(x))}\ll z_n^{-2}$$ 
     	uniformly on $|x|\le \sqrt{2\log z_n}.$ Here, $z_n=\alpha_n^{-1/2} \wedge n$ and $0<\alpha_n\ll 1.$

    The inequality \eqref{CSIneq} gives
\begin{equation}\label{CSIneqn}  
\frac{(\widetilde{M}_{j,k}^{(n)}(x))^2}{(1-\widetilde{M}_{j,j}^{(n)}(x))(1-\widetilde{M}_{k,k}^{(n)}(x))}\le \frac{\widetilde{M}_{j,j}^{(n)}(x)\widetilde{M}_{k,k}^{(n)}(x)}{(1-\widetilde{M}_{j,j}^{(n)}(x))(1-\widetilde{M}_{k,k}^{(n)}(x))}.
\end{equation}
The function \(t\mapsto \frac{t}{1-t}\) is increasing on \((0,1)\); together with the monotonicity of \(\widetilde{M}_{j,j}^{(n)}(x)\) in $j$, which allows us to bound the two factor in \eqref{CSIneqn} by its value at the smallest index in any given range. To proceed, we split the double sum over \(j<k\) into two parts according to whether \(j<\lfloor n^{2/3}\rfloor=i_n\) or \(j\ge i_n\). Estimating each part using the above observations, we obtain
\begin{equation}\label{Mij}
\begin{aligned}
& \sum_{1\le k<j\le n}\frac{(\widetilde{M}_{j,k}^{(n)}(x))^2}{(1-\widetilde{M}_{j,j}^{(n)}(x))(1-\widetilde{M}_{k, k}^{(n)}(x))}\\
&\le \frac{n^{2}\,\widetilde{M}^{(n)}_{i_n, i_n}(x)\,\widetilde{M}_{1, 1}^{(n)}(x)}{(1-\widetilde{M}^{(n)}_{i_n, i_n}(x))(1-\widetilde{M}_{1, 1}^{(n)}(x))}+\sum_{j=2}^{i_n}\sum_{k=1}^{j-1}\frac{\widetilde{M}_{k, k}^{(n)}(x)\,\widetilde{M}_{j, j}^{(n)}(x)}{(1-\widetilde{M}_{j, j}^{(n)}(x))(1-\widetilde{M}_{k, k}^{(n)}(x))}.
\end{aligned}
\end{equation}
 Note
     $$ \Phi(-\sqrt{2\log z_n})=\frac{1}{2\sqrt{\pi}z_n\sqrt{\log z_n}}(1+o(1))$$
     and 
     $$\widetilde{u}_n(i_n, x)=n^{2/3}\alpha_n^{-1/2}(1+o(1)) .$$
     Now Lemma \ref{m0cn} leads  
     	\begin{equation}\label{widecn0xlb} 1-\widetilde{M}_{1, 1}^{(n)}(x)\gtrsim  \Phi(-\sqrt{2\log z_n})\gtrsim \frac{1}{z_n\sqrt{\log z_n}} \end{equation} 
     	as well as \eqref{widecnup} gives 
     	$$\widetilde{M}_{i_n, i_n}^{(n)}(x)\lesssim\frac{e^{-\frac{3}{8} \widetilde{u}^{2}_n(i_n, x)}}{\widetilde{u}_n(i_n,x)}+z_n^{-3/2}\widetilde{u}_n^{-3}(i_n,x)\lesssim z_n^{-3/2}n^{-2}\alpha_{n}^{3/2}\ll 1.$$
     	This, together with $z_n=\alpha_n^{-1/2} \wedge n,$ implies that
     	\begin{equation}\label{Mij1}
     		\frac{n^2 \widetilde{M}_{i_n, i_n}^{(n)}(x)\widetilde{M}_{1, 1}^{(n)}(x)}{(1-\widetilde{M}_{i_n, i_n}^{(n)}(x))(1-\widetilde{M}_{1, 1}^{(n)}(x))}\lesssim\frac{n^2 \widetilde{M}_{i_n, i_n}^{(n)}(x)}{1-\widetilde{M}_{1, 1}^{(n)}(x)}\lesssim z_n^{-1/2}\alpha_n^{3/2}\sqrt{\log z_n}\ll z_n^{-2}.
     	\end{equation}
       Next, we estimate the double sum in \eqref{Mij}. 
     Note that $$\widetilde{u}_n(2,x)=\alpha_n^{-1/2}(1+o(1))\gtrsim z_n, $$ and thus the monotonicity and \eqref{widecnup} ensures 
     	$$\widetilde{M}_{j, j}^{(n)}(x)\le \widetilde{M}_{2, 2}^{(n)}(x)\lesssim \frac{e^{-\frac{3}{8} \widetilde{u}^{2}_n(2, x)}}{\widetilde{u}_n(2,x)}+z_n^{-3/2}\widetilde{u}_n^{-3}(2,x)\lesssim z_n^{-7/2}\ll 1. $$
 This indicates     
     \(1-\widetilde{M}^{(n)}_{j, j}(x)\asymp 1\) uniformly for \(2\le j\le n\) and \(|x|\le \sqrt{2\log z_n}\) and hence the factors \((1-\widetilde{M}^{(n)}_{j, j}(x))^{-1}\) are bounded and can be treated as constants in the estimates.

Consider first the terms with \(k=1\). Using \eqref{widecn0xlb}, \eqref{widecnup} and then Lemma \ref{sum}, we derive 
\[\aligned
\sum_{j=2}^{i_n}\frac{\widetilde{M}^{(n)}_{j, j}(x)\widetilde{M}^{(n)}_{1, 1}(x)}{(1-\widetilde{M}^{(n)}_{j, j}(x))(1-\widetilde{M}^{(n)}_{1, 1}(x))}
&\lesssim z_n\sqrt{\log z_n}\sum_{j=2}^{i_n} \widetilde{M}^{(n)}_{j, j}(x)\\
&\lesssim z_n\sqrt{\log z_n}\sum_{j=2}^{\infty}\bigl(\frac{e^{-\frac{3}{8}\widetilde{u}_n^2(j,x)}}{\widetilde{u}_n(j,x)}+z_n^{-3/2}\widetilde{u}_n^{-3}(j,x)\bigr)\\
&\lesssim z_n\sqrt{\log z_n}(\frac{e^{-\frac{3}{8} \widetilde{u}^{2}_n(2, x)}}{\widetilde{u}_n(2,x)}+z_n^{-3/2}\widetilde{u}_n^{-3}(2, x) )\\
&\lesssim \sqrt{\log z_n} z_n^{-5/2}\ll z_n^{-2}.\endaligned 
\]
For the remaining terms ($k\ge 2$), removing the bounded denominators and do similar estimate on the infinite sum, we have \[
\sum_{j=3}^{i_n}\sum_{k=2}^{j-1}\frac{\widetilde{M}^{(n)}_{j, j}(x)\widetilde{M}^{(n)}_{k, k}(x)}{(1-\widetilde{M}^{(n)}_{j, j}(x))(1-\widetilde{M}^{(n)}_{k, k}(x))}
\lesssim \bigl(\sum_{j=3}^{i_n}\widetilde{M}^{(n)}_{j, j}(x)\bigr)^2\asymp z_n^{-3} \widetilde{u}_n^{-6}(2,x)\ll z_n^{-2}.
\]
Inserting these estimates together with \eqref{Mij1} into \eqref{Mij} yields
\[
\sum_{j<k}\frac{(\widetilde{M}^{(n)}_{j,k}(x))^2}{(1-\widetilde{M}^{(n)}_{j, j}(x))(1-\widetilde{M}^{(n)}_{k,k}(x))}\ll z_n^{-2}.
\]

\subsection{Proofs of Lemmas \ref{m1}, \ref{jn1} and \ref{M0infty}} 
This subsection is devoted to the proofs of the lemmas in section \ref{sec:alphafinite}.

\subsubsection{Proof of Lemma \ref{m1}} 
We start by writing \(\log Y_{n-j+1} = \sum_{r=1}^{k_n} \log S_{n-j+1, \,r}\). This expression represents a sum of i.i.d. random variables \(\{\log S_{n-j+1, r}\}_{1 \le r \le k_n}\). Substituting \(\ell = n - j+1\), we reformulate \(M^{(n)}_{j, j}(x)\) as follows:
\[
M^{(n)}_{j, j}(x) = \mathbb{P}\big( \frac{\log Y_{\ell} - k_n \psi(\ell)}{\sqrt{k_n \psi'(\ell)}} > \frac{k_n (\psi(n) - \psi(\ell))}{\sqrt{k_n \psi'(\ell)}} + \frac{a_n + b_n x}{\sqrt{n \psi'(\ell)}} \big).
 \]
 Now, define the function
 \begin{equation}\label{vndef}
v_n(j, x) = \frac{k_n (\psi(n) - \psi(\ell))}{\sqrt{k_n \psi'(\ell)}} + \frac{a_n + b_n x}{\sqrt{n \psi'(\ell)}}.
\end{equation}
An application of Lemma \ref{diagammapro} yields the approximation $$v_n(j, x) = u_n(j, x)(1 + O((j-1) n^{-1})).$$
Review \(s_n = |\alpha_n - \alpha|^{-1} \wedge n\) and 
 \[
 \ell_{1,\alpha}(n) = (\frac{1}{10} \log s_n)^{1/2}, \quad \ell_{2,\alpha}(n) = \frac{4\sqrt{\log s_n} - a_n}{b_n}, \quad r_n= \lfloor s_n^{1/10} \rfloor.
 \]

 Given that \(a_n\), \(b_n\), and \(\alpha_n\) are bounded, we observe that \(u_n(j, x) \gg 1\) whenever \(j\ge r_n\) and $x\in[-\ell_{1, \alpha}(n),\ell_{2, \alpha}(n)]$. Therefore, the asymptotic behavior of \(M^{(n)}_{j, j}(x)\) in this regime mirrors the case where \(\alpha = +\infty\). In contrast, for \(1 \le j \le r_n \) and finite $x$, the quantity \(u_n(j, x)\) varies from a constant to positive infinity. This range of values introduces significant complexity into obtaining a precise asymptotic description of \(M^{(n)}_{j, j}(x)\).

 By Lemma \ref{ed}, we obtain
\begin{equation}\label{ccc}
M^{(n)}_{j, j}(x)= 1 - \Phi(v_n(j, x)) + \frac{1 - v_n^2(j, x)}{6\sqrt{n k_n}} \phi(v_n(j, x)) + O(n^{-3/2}).
\end{equation}
We next reduce the expression for \(v_n(j, x)\) to \(v_{\alpha}(j, x)\). Recall that
\[
a_n = \sqrt{\log(\alpha_n + 1)} - \frac{\log(\sqrt{2\pi} \log(\alpha_n + e^{\frac{1}{\sqrt{2\pi}}}))}{\sqrt{\log(\alpha_n + e)}}, \quad
b_n = \frac{1}{\sqrt{\log(\alpha_n + e)}}.
\]
Using Taylor expansion and the fact that \(\alpha_n - \alpha = o(1)\), we have for any fixed \(r > 0\),
\[
\sqrt{\log(\alpha_n + r)} = \sqrt{\log(\alpha + r)} + \frac{\alpha_n - \alpha}{2(\alpha + r)\sqrt{\log(\alpha + r)}} + O((\alpha_n - \alpha)^2),
\]
\[
\frac{1}{\sqrt{\log(\alpha_n + r)}} = \frac{1}{\sqrt{\log(\alpha + r)}} - \frac{\alpha_n - \alpha}{2(\alpha + r)(\log(\alpha + r))^{3/2}} + O((\alpha_n - \alpha)^2),
\]
as well as
\[
a_n = a + c_1(\alpha_n - \alpha) + O((\alpha_n - \alpha)^2), \quad
b_n = b - c_2(\alpha_n - \alpha) + O((\alpha_n - \alpha)^2).
\]
Similarly,
\[
\frac{j-1}{\sqrt{\alpha_n}} = \frac{j-1}{\sqrt{\alpha}} (1 - \frac{\alpha_n - \alpha}{2\alpha})\left(1 + O((\alpha_n - \alpha)^2)\right),
\]
and by Lemma \ref{diagammapro},
\[
\frac{1}{\sqrt{n \psi'(n-j+1)}} = 1 - \frac{2j - 1}{4n} + O(\frac{j^2}{n^2}),
\]
\[
\psi(n) - \psi(n - j+1) = \frac{j-1}{n} + \frac{j(j-1)}{2n^2} + O(\frac{j^3}{n^3}).
\]
Therefore,
\begin{equation}\label{vneq1}
\begin{aligned}
\frac{k_n (\psi(n) - \psi(n - j+1))}{\sqrt{k_n \psi'(n - j+1)}} 
&= \frac{j-1}{\sqrt{\alpha_n}} (1 - \frac{2j - 1}{4n} + \frac{j}{2n})(1 + O(\frac{j^2}{n^2})) \\
&= \frac{j-1}{\sqrt{\alpha}} (1 + \frac{1}{4n} - \frac{\alpha_n - \alpha}{2\alpha})(1 + O(j^2 n^{-2} + (\alpha_n - \alpha)^2)),
\end{aligned}
\end{equation}
and
\begin{equation}\label{vneq2}
\begin{aligned}
\frac{a_n + b_n x}{\sqrt{n \psi'(n - j+1)}}
&= (a + b x - \frac{(a + b x)(2j - 1)}{4n} + (c_1 - c_2 x)(\alpha_n - \alpha)) \\
&\quad \times (1 + O(\frac{j^2}{n^2} + (\alpha_n - \alpha)^2)).
\end{aligned}
\end{equation}
The choices of \(s_n\), \(j\) and \(\ell_{1, \alpha}(n), \ell_{2, \alpha}(n) \) ensure that
\[
\frac{j^2 (j+ |x|)}{n^2} + (\alpha_n - \alpha)^2 (j + |x|) \lesssim s_n^{-\frac{17}{10}}.
\]
Substituting \eqref{vneq1} and \eqref{vneq2} into \eqref{vndef} yields
\begin{equation}\label{asymforun0}
\begin{aligned}
v_n(j, x) &= v_{\alpha}(j, x) - \frac{(2j - 1)v_{\alpha}(j, x)}{4n} + \frac{j(j-1 )}{2\sqrt{\alpha} \, n} \\
&\quad + (\alpha_n - \alpha)(c_1 - c_2 x - \frac{j-1}{2\alpha^{3/2}}) + O(s_n^{-\frac{17}{10}}).
\end{aligned}
\end{equation}
Define
\[\aligned
\zeta_n(j, x) :&=v_n(j, x) - v_{\alpha}(j, x)\\
&= -\frac{(2j- 1)v_{\alpha}(j, x)}{4n} + \frac{j(j -1)}{2\sqrt{\alpha} \, n} + (\alpha_n - \alpha)\big(c_1 - c_2 x - \frac{j-1}{2\alpha^{3/2}}\big) + O(s_n^{-\frac{17}{10}}),\\
\endaligned 
\]
so that \(|\zeta_n(j, x)| \lesssim s_n^{-\frac45}\). 

Applying Taylor expansions to \(\Phi\) and \(\phi\), we obtain
\[
\Phi(v_n(j, x)) = \Phi(v_\alpha(j, x)) + \phi(v_\alpha(j, x)) \left(\zeta_n(j, x) + O(v_{\alpha}(j, x) \zeta_n^2(j, x))\right),
\]
and
\[
\phi(v_n(j, x)) = \phi(v_\alpha(j, x)) \left(1 + O(v_{\alpha}(j, x) |\zeta_n(j, x)|)\right).
\]
  Note that for $1\leq j\leq r_n$ and $-\ell_{1,\alpha}(n)\leq x\leq \ell_{2,\alpha}(n), $ 
 	$$v_{\alpha}(j, x)\zeta_n^{2}(j,x)\lesssim s_n^{-\frac32} \quad \text{ and} \quad \frac{|v_\alpha(j,x)\zeta_n(j,x)|}{n}\lesssim s_n^{-\frac{17}{10}}.$$
 	Therefore, it follows from \eqref{ccc} that 
 			$$\aligned 
 			&M^{(n)}_{j, j}(x)\\
 			=&1-\Phi(v_\alpha(j, x))-\phi(v_\alpha(j, x))\big(\frac{\sqrt{\alpha}(v^{2}_\alpha(j, x)-1)}{6n}+\zeta_n(j,x)+O(s_n^{-\frac{3}{2}})\big)+O(n^{-\frac{3}{2}})\\
 			=&1-\Phi(v_\alpha(j, x))-\phi(v_\alpha(j, x))(n^{-1} q_{1}(j, x)+(\alpha_n-\alpha)q_2(j, x))+O(s_n^{-\frac32})
 			.\endaligned $$
  The proof is then completed. 

\subsubsection{Proof of Lemma \ref{jn1}}
 Since both \(1 - M^{(n)}_{j, j}(x)\) and \(\Phi(v_{\alpha}(j, x))\) are increasing in \(x\) and bounded above by 1, it suffices to prove Lemma \ref{jn1} for \(x = -\ell_{1,\alpha}(n)\).

By Lemma \ref{diagammapro}, we have
\[
v_n(j, -\ell_{1,\alpha}(n)) = u_n(j, -\ell_{1,\alpha}(n))(1 + O(j n^{-1})) = O(j).
\]
Applying the central limit theorem to the sequence \((\log S_{n - j+1, r})_{1 \le r \le k_n}\), we obtain
\begin{equation}\label{c1}
M^{(n)}_{j, j}(-\ell_{1,\alpha}(n)) = \frac{1 + o(1)}{\sqrt{2\pi} \, u_n(j, -\ell_{1,\alpha}(n))} \exp\big(-\frac{1}{2} u_n^2(j, -\ell_{1,\alpha}(n))\big)
\end{equation}
uniformly for \(r_n<j\ll n^{1/6}\). The monotonicity of \(M^{(n)}_{j, j}(x)\) in \(j\) implies
\[\aligned
&\quad M^{(n)}_{j, j}(-\ell_{1,\alpha}(n))\le M^{(n)}_{r_n+1, r_n+1}(-\ell_{1,\alpha}(n)) \lesssim \frac{\exp\big(-\frac{1}{2} u_n^2(r_n+1, -\ell_{1,\alpha}(n))\big)}{u_n(r_n+1, -\ell_{1,\alpha}(n))}  = o(1)
\endaligned 
\]
uniformly for \(j \ge r_n+1\), noting that
\begin{equation}\label{256}
u_n(r_n+1, -\ell_{1,\alpha}(n)) = \frac{r_n}{\sqrt{\alpha}}(1 + o(1)) \gg 1.
\end{equation}
Consequently,
\begin{equation}\label{cm}
\sum_{j= r_n+1}^{n} \log(1 - M^{(n)}_{j, j}(-\ell_{1,\alpha}(n))) = -(1 + o(1))\sum_{j = r_n+1}^{n } M^{(n)}_{j, j}(-\ell_{1,\alpha}(n)).
\end{equation}

We now consider two cases.

{\bf Case 1}: \((\alpha_n - \alpha)^{-1} \ll n\), i.e., \(r_n = \lfloor (\alpha_n - \alpha)^{-\frac1{10}} \rfloor\ll n^{1/10}\). Then by \eqref{c1},
\[
\begin{aligned}
\sum_{j = r_n+1}^{n } M^{(n)}_{j, j}(-\ell_{1,\alpha}(n)))
&= \sum_{j=r_n+1}^{\lfloor n^{\frac1{10}}-1 \rfloor} M^{(n)}_{j, j}(-\ell_{1,\alpha}(n)))+ \sum_{j= \lfloor n^{\frac1{10}} \rfloor}^{n }M^{(n)}_{j, j}(-\ell_{1,\alpha}(n))) \\
&\lesssim \sum_{j =r_n+1}^{\lfloor n^{\frac1{10}} \rfloor} \frac{\exp\left(-\frac{1}{2} u_n^2(j, -\ell_{1,\alpha}(n))\right)}{u_n(j, -\ell_{1,\alpha}(n))} + \frac{n \exp\left(-\frac{1}{2} u_n^2(\lfloor n^{\frac1{10}} \rfloor, -\ell_{1,\alpha}(n))\right)}{u_n(\lfloor n^{\frac1{10}} \rfloor, -\ell_{1,\alpha}(n))}.
\end{aligned}
\]
Applying Lemma \ref{sum} with \(\gamma_n = \sqrt{\alpha_n}\) bounded to the first sum yields
\[
\sum_{j = r_n+1}^{n } M^{(n)}_{j, j}(-\ell_{1,\alpha}(n)))
\lesssim \frac{\exp\left(-\frac{1}{2} u_n^2(r_n+1, -\ell_{1,\alpha}(n))\right)}{u_n(r_n+1, -\ell_{1,\alpha}(n))} + \frac{n \exp\left(-\frac{1}{2} u_n^2(\lfloor n^{\frac1{10}} \rfloor, -\ell_{1,\alpha}(n))\right)}{u_n(\lfloor n^{\frac1{10}} \rfloor, -\ell_{1,\alpha}(n))}.
\]
Note that
\[
u_n(\lfloor n^{\frac1{10}} \rfloor, -\ell_{1,\alpha}(n)) = \frac{n^{\frac1{10}}}{\sqrt{\alpha}}(1 + o(1)),
\]
which implies
\[
\frac{n \exp\left(-\frac{1}{2} u_n^2(\lfloor n^{\frac1{10}} \rfloor, -\ell_{1,\alpha}(n))\right)}{u_n(\lfloor n^{\frac1{10}} \rfloor, -\ell_{1,\alpha}(n))} \lesssim \exp\{ -\frac{n^{\frac1{5}}(1 + o(1))}{2\alpha} + \frac{9}{10} \log n\} \ll \exp\{ -\frac{n^{\frac1{5}}}{3\alpha} \}.
\]
It follows from \eqref{256} that
\[
\frac{\exp\left(-\frac{1}{2} u_n^2(r_n+1, -\ell_{1,\alpha}(n))\right)}{u_n(r_n+1, -\ell_{1,\alpha}(n))} + \frac{n \exp\left(-\frac{1}{2} u_n^2(\lfloor n^{\frac1{10}} \rfloor, -\ell_{1,\alpha}(n))\right)}{u_n(\lfloor n^{\frac1{10}} \rfloor, -\ell_{1,\alpha}(n))} = o(\exp(-\frac{1}{3\alpha} s_n^{\frac15})).
\]
{\bf Case 2}: \((\alpha_n - \alpha)^{-1} \gtrsim n\), under which  \(r_n = \lfloor n^{\frac1{10}} \rfloor\). Then
\[
\sum_{j = r_n+1}^{n } M^{(n)}_{j, j}(-\ell_{1,\alpha}(n)))\lesssim \frac{n \exp\left(-\frac{1}{2} u_n^2(r_n+1, -\ell_{1,\alpha}(n))\right)}{u_n(r_n+1, -\ell_{1,\alpha}(n))} \ll \exp(-\frac{1}{3\alpha}r_n^2).
\]
In both cases, we conclude
\[
\sum_{j = r_n+1}^{n } M^{(n)}_{j, j}(-\ell_{1,\alpha}(n)) =o(\exp(-\frac{1}{3\alpha}s_n^{\frac15})),
\]
and hence \eqref{cm} gives 
\[
\sum_{j = r_n+1}^{n} \log(1 - M^{(n)}_{j, j}(-\ell_{1,\alpha}(n))) = o(\exp(-\frac{1}{3\alpha}s_n^{\frac15})).
\]
Next, we establish the second asymptotic relation of Lemma \ref{jn1}. For  \(j \geq r_n+1\), we have
\[v_\alpha(j, -\ell_{1,\alpha}(n))\geq
v_\alpha(r_n+1, -\ell_{1,\alpha}(n)) \gg 1.
\]
The Mills ratio gives
\[
1 - \Phi(v_\alpha(j, -\ell_{1,\alpha}(n))) = \frac{1 + o(1)}{\sqrt{2\pi}}  \frac{\exp\left(-\frac{1}{2} v^2_\alpha(j, -\ell_{1,\alpha}(n))\right)}{v_\alpha(j, -\ell_{1,\alpha}(n))}.
\]
Consequently,
\[
\sum_{j = r_n+1}^{+\infty} \log \Phi(v_\alpha(j, -\ell_{1,\alpha}(n))) = -\frac{1 + o(1)}{\sqrt{2\pi}} \sum_{j = r_n+1}^{+\infty} \frac{\exp\left(-\frac{1}{2} v^2_\alpha(j, -\ell_{1,\alpha}(n))\right)}{v_\alpha(j, -\ell_{1,\alpha}(n))}.
\]
Applying Lemma \ref{sum} once more, and using the estimate
\[
v_\alpha(r_n+1, -\ell_{1,\alpha}(n)) = \frac{r_n}{\sqrt{\alpha}}(1 + o(1)),
\]
we obtain
\[
\sum_{j= r_n+1}^{+\infty} \frac{\exp\left(-\frac{1}{2} v^2_\alpha(j, -\ell_{1,\alpha}(n))\right)}{v_\alpha(j, -\ell_{1,\alpha}(n))} \lesssim \frac{\exp\left(-\frac{1}{2} v^2_\alpha(r_n+1, -\ell_{1,\alpha}(n))\right)}{v_\alpha(r_n+1, -\ell_{1,\alpha}(n))} \ll \exp\left(-\frac{r_n^2}{3\alpha}\right).
\]
Hence,
\[
\sum_{j = r_n+1}^{+\infty} \log \Phi(v_\alpha(j, -\ell_{1,\alpha}(n))) = o\big(\exp\big(-\frac{s_n^{1/5}}{3\alpha}\big)\big).
\]
This completes the proof.	

\subsubsection{Proof of Lemma \ref{M0infty}} For $1 \le j \le r_n$, the expression \eqref{Mjjfor} in Lemma \ref{Mjjlem} tells
 $$\aligned \widetilde{M}_{j, j}^{(n)}(x)&=\mathbb{P}(\log Y_{n+1-j}+\log \cos^2\Theta\ge k_n\psi(n)+\frac{\widetilde{a}_n+\widetilde{b}_n x}{\sqrt{\alpha_n}})\\
 &=\frac{2}{\pi}\int_0^{\pi/2}\mathbb{P}(\log Y_{n+1-j}\ge k_n\psi(n)+\frac{\widetilde{a}_n+\widetilde{b}_n x}{\sqrt{\alpha_n}}-\log\cos^2\theta)d\theta.
 \endaligned $$ 
 
 Similarly as Lemma \ref{m1}, one gets 
 \begin{equation}\label{Yn-m}
 	\mathbb{P}(\log Y_{n-j+1}\ge k_n\psi(n)+\frac{\widetilde{a}_n+\widetilde{b}_n x}{\sqrt{\alpha_n}}-\log\cos^2\theta)=\Psi(\widetilde{v}_{\alpha}(j,x)-\sqrt{\alpha}\log\cos^2\theta)+O(s_n^{-1}),
 \end{equation}
  whence 
  $$\widetilde{M}_{j, j}^{(n)}(x)=\frac{2}{\pi}\int_{0}^{\pi/2}\Psi(\widetilde{v}_{\alpha}(j, x)-\sqrt{\alpha}\log\cos^2\theta) d\theta+O(s_n^{-1}).$$
  
 Similarly, the expression \eqref{Mjkfor}, together with the expression \eqref{Yn-m} for $n-\frac{j+k}{2}+1,$ gives 
$$ \begin{aligned}
\widetilde{M}_{j, k}^{(n)}(x) &= \frac{2(\big( n-\frac{j+k}{2} \big)!)^{k_n}}{\pi \bigl( (n-j)! (n-k)! \bigr)^{k_n/2}} \\
&\quad \times \int_{0}^{\pi/2} \cos\bigl( (j-k)\theta \bigr) \,(\Psi(\widetilde{v}_{\alpha}(\frac{j+k}{2},x)-\sqrt{\alpha}\log \cos^2\theta)+O(s_n^{-1}))d\theta.
\end{aligned}
 $$  
 Noting that $j+k\lesssim r_n\ll n$ and $2(n-\frac{j+k}{2})=n-j+n-k,$ Stirling's formula and some simple calculus give 
 $$\aligned \log&\frac{(\big( n-\frac{j+k}{2} \big)!)^2}{ (n-j)! (n-k)! }
=2(n+\frac12-\frac{j+k}{2})\log(n+1-\frac{j+k}{2})\\
 &-(n+\frac12-j)\log(n+1-j)-(n+\frac12-k)\log(n+1-k)+O(\frac{(j-k)^2}{n^3})
 \endaligned$$
 and then by the corresponding Taylor's expansion for logarithmic  function, we continue to write 
 $$\aligned 
 \log&\frac{(\big( n-\frac{j+k}{2} \big)!)^2}{ (n-j)! (n-k)! }
=-\frac{(j-k)^2}{4n}+O(\frac{(j-k)^2(j+k)}{n^2}),
 \endaligned $$
 which, together with $k_n\asymp n,$ tells \begin{equation}\label{coefMjk}\frac{(\big( n-\frac{j+k}{2} \big)!)^{k_n}}{\bigl( (n-j)! (n-k)! \bigr)^{k_n/2}}=\exp(-\frac{(j-k)^2}{4\alpha_n})(1+O(\frac{(j-k)^2(j+k)}{n})).\end{equation}
Now $$\alpha_n^{-1}=(\alpha(1+\frac{\alpha_n-\alpha}{\alpha}))^{-1}=\alpha(1+O(s_n^{-1}))$$
and then  we have
  $$ \exp(-\frac{(j-k)^2}{4\alpha_n})=\exp(-\frac{(j-k)^2}{4\alpha})(1+O((j-k)^2s_n^{-1})).$$
Inserting this asymptotic into \eqref{coefMjk} yields
$$\frac{(\big( n-\frac{j+k}{2} \big)!)^{k_n}}{\bigl( (n-j)! (n-k)! \bigr)^{k_n/2}}=\exp(-\frac{(j-k)^2}{4\alpha})(1+O(\frac{j_n(j-k)^2}{n}))=\exp(-\frac{(j-k)^2}{4\alpha})+O(s_n^{-9/10}).$$
 Consequently, for 
$1\leq j,k\leq r_n$, together with the boundedness of $\Psi$, we obtain the representation $$ \widetilde{M}_{j, k}^{(n)}(x) = \frac{2\exp(-\frac{(j-k)^2}{4\alpha})}{\pi}  \int_{0}^{\pi/2} \cos\bigl( (j-k)\theta \bigr) \Psi(\widetilde{v}_{\alpha}(\frac{j+k}{2},x)-\sqrt{\alpha}\log \cos^2\theta)d\theta+O(s_n^{-9/10}).$$
 The second item is verified. 
  
  The monotonicity of $\widetilde{M}_{j, j}^{(n)}(x)$ in $j$ leads  
  \begin{equation}\label{mjj1}
  	\widetilde{M}_{j, j}^{(n)}(x) \le \widetilde{M}_{j_n+1, j_n+1}^{(n)}(x)
  \end{equation}
  for $r_n+1\le j\le n$ and $x \ge -\sqrt{\log s_n}.$ Note that $\log \cos^2\Theta < 0$ and then 
 $$
  	\aligned \widetilde{M}_{r_n+1, r_n+1}^{(n)}(x)\leq\mathbb{P}(\log Y_{n-j_n}\ge k_n\psi(n)+\frac{\widetilde{a}_n+\widetilde{b}_n x}{\sqrt{\alpha_n}}),\endaligned 
  $$ 
  whence similarly as \eqref{c1} 
$$
	\widetilde{M}_{r_n+1, r_n+1}^{(n)}(x)\lesssim \frac{1}{\widetilde{u}_n(r_n+1, x)}\exp(-\frac{\widetilde{u}_n^{2}(r_n+1,x)}{2}).
$$
 With $\widetilde{u}_n(r_n+1, -\sqrt{\log s_n})=\frac{s_n^{1/10}}{\sqrt{\alpha}}(1+o(1)),$ we have
 $$\widetilde{M}_{r_n+1, r_n+1}^{(n)}(x)\ll s_n^{-1/10}\exp(-s_n^{1/6}),$$
 uniformly on $x\geq -\sqrt{\log s_n}.$ The proof of the third item is completed. 
 Since $a_n, b_n,$
 $\widetilde{a}_n$ and $\widetilde{b}_n$
  are constants when $\alpha_n\in(0, +\infty)$, we have 
  $$\mathbb{P}(\log Y_{n-j}\ge k_n\psi(n)+\frac{\widetilde{a}_n+\widetilde{b}_n x}{\sqrt{\alpha_n}})\asymp \mathbb{P}(\log Y_{n-j}\ge k_n\psi(n)+\frac{a_n+b_n x}{\sqrt{\alpha_n}}),$$
  for $r_n+1\leq j\leq n.$ Therefore, with the help of Lemma \ref{jn1}, we obtain 
  $$ \sum_{j=r_n+1}^{n}\widetilde{M}_{j, j}^{(n)}(x)\asymp \sum_{j=r_n+1}^{n}M_{j, j}^{(n)}(x)\ll \exp(-\frac{r_n^2}{3\alpha}) ,$$
   which implies \eqref{summjjj}. The proof is finished now.

 \subsection{Proof of Remark \ref{upperbound}} For simplicity, use $v_{\alpha}$ to replace $v_{\alpha}(j, x)$ and rewrite $q_{1}(j, x)$ as 
  $$\aligned q_{1}(j, x)&=\frac{\sqrt{\alpha}}{6}v_{\alpha}^2-\frac{1}{2}\big(\sqrt{\alpha}(a+b x)-\frac12\big) v_{\alpha}+\frac{\sqrt{\alpha}}{2}(a+bx)^2-\frac12(a+bx)-\frac{\sqrt{\alpha}}{6}\\
  &=:f_1 v_{\alpha}^2-f_2(x) v_{\alpha}+f_3(x) \endaligned $$   with $f_1(x)=\frac{\sqrt{\alpha}}{6},$
  $$f_2(x)=\frac{1}{2}\big(\sqrt{\alpha}(a+b x)-\frac12\big) \quad \text{and} \quad f_3(x)=\frac{\sqrt{\alpha}}{2}(a+bx)^2-\frac12(a+bx)-\frac{\sqrt{\alpha}}{6}.$$
  Then 
  $$\aligned |\sum_{j=1}^{+\infty} q_{1}(j, x)\frac{\phi(v_{\alpha})}{\Phi(v_{\alpha})}|&\le f_1\sum_{j=1}^{+\infty} v_{\alpha}^2\frac{\phi(v_{\alpha})}{\Phi(v_{\alpha})}+|f_2(x)|\sum_{j=1}^{+\infty} |v_{\alpha}|\frac{\phi(v_{\alpha})}{\Phi(v_{\alpha})}+|f_3(x)|\sum_{j=1}^{+\infty} \frac{\phi(v_{\alpha})}{\Phi(v_{\alpha})}, \endaligned $$ 
  which in further ensures that   
 \begin{equation}\label{sumcontrol} \aligned \Phi_{\alpha}(x)\sum_{j=1}^{+\infty} \frac{|v^{k}_{\alpha}|\phi(v_{\alpha})}{\Phi(v_{\alpha})}&\le \sum_{j=1}^{+\infty} |v_{\alpha}^k|\phi(v_{\alpha})\le \sqrt{\alpha},
  \endaligned \end{equation}
   where the last inequality holds because the second term is controlled by the following integral 
   $$\sqrt{\alpha}\int_{-\infty}^{+\infty} |t|^k \phi(t) dt$$ for $k=0, 1, 2.$ This  observation tells us that 
    \begin{equation}\label{sumtotal} \aligned &\quad \Phi_{\alpha}(x)\sum_{j=1}^{+\infty} |q_{1}(j, x)|\frac{\phi(v_{\alpha})}{\Phi(v_{\alpha})}\le \frac{\alpha}{6}+(|f_2(x)|+|f_3(x)|)\Phi_{\alpha}(x)\sum_{j=1}^{+\infty} (|v_{\alpha}|\vee 1)\frac{\phi(v_{\alpha})}{\Phi(v_{\alpha})}.
  \endaligned \end{equation} 
  As a direct consequence, for $x$ such that $|v_{\alpha}(1, x)|\le 1$ ensuring the boundedness of $|f_2(x)|+|f_3(x)|,$ we have  
   $$\aligned \sup_{x: \; |v_{\alpha}(1, x)|\le 1}\Phi_{\alpha}(x)|\sum_{j=1}^{+\infty} q_{1}(j, x)\frac{\phi(v_{\alpha})}{\Phi(v_{\alpha})}|
   &\le \frac{\alpha}{6}+\sqrt{\alpha}\sup_{x: \; |v_{\alpha}(1, x)|\le 1} (|f_2(x)|+|f_3(x)|)\\
   &\le \frac{4(\alpha+\sqrt{\alpha})}{3}.\endaligned $$ 
 For $x$ such that $v_{\alpha}(1, x)>1,$ it follows from the monotonicity of $t^k \phi(t)$ on $t>1$ that 
 \begin{equation}\label{sumkbound}\aligned \sum_{j=1}^{+\infty} |v_{\alpha}^k|\phi(v_{\alpha})&\le v_{\alpha}^k(1, x)\phi(v_{\alpha}(1, x))+\sqrt{\alpha}\int_{v_{\alpha}(1, x)}^{+\infty} t^k \phi(t) dt\\
 &\le v_{\alpha}^{k-1}(1, x)(v_{\alpha}(1, x)+\sqrt{\alpha})\phi(v_{\alpha}(1, x))\endaligned\end{equation}
 for $k=0, 1.$
 Therefore,  we have from \eqref{sumtotal} and \eqref{sumkbound} that 
 $$\aligned &\quad \Phi_{\alpha}(x)|\sum_{j=1}^{+\infty} q_{1}(j, x)\frac{\phi(v_{\alpha})}{\Phi(v_{\alpha})}|\le \frac{\alpha}{6}+\phi(v_{\alpha}(1, x))(v_{\alpha}(1, x)+\sqrt{\alpha})(|f_2(x)|+|f_3(x)|v_{\alpha}^{-1}(1, x)).\endaligned$$  
Hence, use the substitution $t=v_{\alpha}(1, x)$ to have 
$$\aligned  \sup_{x: \; v_{\alpha}(1, x)>1}\Phi_{\alpha}(x)|\sum_{j=1}^{+\infty} q_{1}(j, x)\frac{\phi(v_{\alpha})}{\Phi(v_{\alpha})}|
 &\le \frac{\alpha}{6}+\sup_{t>1}\phi(t)(t+\sqrt{\alpha})(\sqrt{\alpha} t+\frac34+\frac{\sqrt{\alpha}}{6t})\\
 &\le \alpha+\frac{2\sqrt{\alpha}+1}{5}.\endaligned$$
 While for $x$ such that $v_{\alpha}(1, x)<-1,$ define 
  $$j_0:=\inf\{j: v_{\alpha}(j, x)\ge 0\},$$ which 
  implies with the monotonicity of $\Phi$ and $\Phi\le 1$ that $$\frac{\Phi_{\alpha}(x)}{\Phi(v_{\alpha}(j, x))}\le \Phi^{j_0-1}(v_{\alpha}(j_0, x))\le 2^{-j_0+1}.$$
   Thus, using the second inequality in \eqref{sumcontrol}, we have 
    $$\aligned  \Phi_{\alpha}(x)\sum_{j=1}^{+\infty} \frac{|v_{\alpha}|^k\phi(v_{\alpha})}{\Phi(v_{\alpha})}&\le \sqrt{\alpha} 2^{-j_0+1}\leq\sqrt{\alpha} 2^{\sqrt{\alpha} v_{\alpha}(1, x)+1},\endaligned $$ 
 where we use the fact that $j_0\geq -\sqrt{\alpha} v_{\alpha}(1, x).$
 Consequently, we have 
  $$\aligned \sup_{x: \; v_{\alpha}(1, x)<-1}\Phi_{\alpha}(x)|\sum_{j=1}^{+\infty} q_{1}(j, x)\frac{\phi(v_{\alpha})}{\Phi(v_{\alpha})}|&\le \frac{\alpha}{6}+2\sqrt{\alpha}\sup_{x: \; v_{\alpha}(1, x)<-1} 2^{\sqrt{\alpha} v_{\alpha}(1, x)}(|f_2(x)|+|f_3(x)|)\\
  &=\frac{\alpha}{6}+\sup_{t\ge \sqrt{\alpha}} 2^{-t}(t^2 +(1+\sqrt{\alpha})t+\frac{\sqrt{\alpha}}{2}+\frac{\alpha}{3})\\
  &\le \frac{\alpha}{6}+\frac{4}{3}.\endaligned$$ 
  When comparing these three suprema, we have  
   $$\aligned \sup_{x\in\mathbb{R}}\Phi_{\alpha}(x)|\sum_{j=1}^{+\infty} q_{1}(j, x)\frac{\phi(v_{\alpha})}{\Phi(v_{\alpha})}|&\le \frac{4}{3}(\alpha+\sqrt{\alpha}+1).\endaligned$$ 
   
We now estimate $\sup_{x\in\mathbb{R}}\Phi_{\alpha}(x)|\sum_{j=1}^{+\infty} q_{2}(j, x)\frac{\phi(v_{\alpha})}{\Phi(v_{\alpha})}|.$
A simple calculus indicates 
$$|q_2(j, x)|\le \frac{1}{2\alpha}|v_{\alpha}(j, x)|+(\frac1{2\alpha}-\frac{c_2}{b})|v_{\alpha}(1, x)|+c_1+\frac{c_2\alpha}{b}=:\frac{1}{2\alpha}|v_{\alpha}(j, x)|+f_4(x),$$
whence 
\begin{equation}\label{sumcontrol1}\Phi_{\alpha}(x)|\sum_{j=1}^{+\infty} q_{2}(j, x)\frac{\phi(v_{\alpha})}{\Phi(v_{\alpha})}|\le f_4(x)\Phi_{\alpha}(x)\sum_{j=1}^{+\infty}\frac{\phi(v_{\alpha})}{\Phi(v_{\alpha})}+\frac{1}{2\alpha}\Phi_{\alpha}(x)\sum_{j=1}^{+\infty} \frac{|v_{\alpha}|\phi(v_{\alpha})}{\Phi(v_{\alpha})}.\end{equation} Thereby, it follows from the second inequality of \eqref{sumcontrol} that 
 $$\aligned \sup_{x: \; |v_{\alpha}(1, x)|\le 1}\Phi_{\alpha}(x)|\sum_{j=1}^{+\infty} q_{2}(1, x)\frac{\phi(v_{\alpha})}{\Phi(v_{\alpha})}|
   &\le \sqrt{\alpha}\sup_{x: \; |v_{\alpha}(j, x)|\le 1} f_4(x)+\frac{1}{2\sqrt{\alpha}}\\
   &=\sqrt{\alpha}(\frac1{\alpha}+\frac{c_2(\alpha-1)}{b}+c_1).\endaligned $$ 
 and similarly 
   $$\aligned \sup_{x: \; v_{\alpha}(1, x)\le -1}\Phi_{\alpha}(x)|\sum_{j=1}^{+\infty} q_{2}(j, x)\frac{\phi(v_{\alpha})}{\Phi(v_{\alpha})}|
   &\le \frac{1}{2\sqrt{\alpha}}+2 \sqrt{\alpha} \sup_{x: \; v_{\alpha}(1, x)\le -1} 2^{\sqrt{\alpha}v_{\alpha}(1, x) } f_4(x)\\
   &\le \frac{1}{2\sqrt{\alpha}}+\frac{2}{e\ln 2}(c_1+\frac{c_2(\alpha-1)}{b}+\frac1{\alpha}).\endaligned $$  
   Here, for the last inequality we use $\sup_{t>0} 2^{-t} t=(e\ln 2)^{-1}.$

 We also have from \eqref{sumkbound} that 
   $$\aligned &\quad \sup_{x: \; v_{\alpha}(1, x)\ge 1}\Phi_{\alpha}(x)|\sum_{j=1}^{+\infty} q_{2}(j, x)\frac{\phi(v_{\alpha})}{\Phi(v_{\alpha})}|\\
 &\le \sup_{x: \; v_{\alpha}(1, x)\ge 1} \phi(v_{\alpha}(1, x))(v_{\alpha}(1, x)+\sqrt{\alpha} )(v_{\alpha}(1, x)^{-1}f_4(x)+\frac{1}{2\alpha})\\
 &=\sup_{t>1} \phi(t)\big((\frac{1}{\alpha}-\frac{c_2}{b}) t+\frac{c_1b+c_2\alpha}{b}+(\frac{1}{\alpha}-\frac{c_2}{b})\sqrt{\alpha}+\frac{c_1b+c_2\alpha}{b t} \sqrt{\alpha}\big)\\
 &=e^{-\frac12}(\sqrt{\alpha}+1)(\frac{1}{\alpha}+c_1+\frac{c_2(\alpha-1)}{b}).
 \endaligned$$ 
 Comparing these three upper bounds, we have 
$$\aligned \sup_{x\in\mathbb{R}}\Phi_{\alpha}(x)|\sum_{j=1}^{+\infty} q_{2}(j, x)\frac{\phi(v_{\alpha})}{\Phi(v_{\alpha})}|&\le \frac{2}{e\ln 2}(c_1+\frac{c_2(\alpha-1)}{b}+\frac1{\alpha})(1+\sqrt{\alpha}).\endaligned$$ 

\subsection*{Acknowledgment}  Yutao Ma acknowledges partial support from the National Natural Science Foundation of China (Grants No. 12171038 and 12571149) and the 985 Project. The authors also wish to thank Professor Forrester for pointing out several relevant references, and Professor Liming Wu and Professor Zhenqi Wang for their encouragement and valuable help throughout the course of this research.


\begin{thebibliography}{SOSL90} 
\bibitem{Abramowitz1968} Abramowitz M, Stegun I A. Handbook of mathematical functions with formulas, graphs, and mathematical tables. US Government printing office, 1968.

\bibitem{Adhi} Adhikari K, Reddy N K, Reddy T R, Saha K. Determinantal point processes in
the plane from products of random matrices. \emph{Ann. Inst. H. Poincare Probab. Statist.},
2016, {\bf 52}: 16-46.

\bibitem{AB}
 Akemann G, Bender M. Interpolation between Airy and Poisson statistics for unitary chiral non-Hermitian random matrix theory. \emph{J. Math. Phys.}, 2010, \textbf{51}, 103524. 
 

\bibitem{AkBur} Akemann G, Burda Z. Universal microscopic correlation functions for products of independent Ginibre matrices. \emph{ J. Phys. A: Math. Theor.}, 2012, {\bf 45}(46): 465201.

\bibitem{AkIp15} Akemann G, Ipsen J R. Recent exact and asymptotic results for products of independent random matrices. \emph{Acta Phys. Pol. B,} 2015, {\bf 46}(9): 1747-1784.

\bibitem{AnG}
 Anderson G, Guionnet A, Zeitouni O. \emph{An Introduction to Random Matrices}. Cambridge University Press, Cambridge, 2010.


\bibitem{BaiY98} Bai Z D, Yin Y Q. Limiting behavior of the norm of products of random matrices and two problems of Geman-Hwang. \emph{Probab. Th. Relat. Fields}, 1998, {\bf 73}: 555-569.

\bibitem{bellman} Bellman R. Limit theorems for non-commutative operations. \emph{I. Duke Math. J.}, 1954, {\bf 21}: 491-500. 

\bibitem{Benderellip} Bender M. Edge scaling limits for a family of non-Hermitian random matrix ensembles. \emph{Probab. Th. Relat. Fields}, 2010, {\bf 147}: 241-271.

\bibitem{Bordenave} Bordenave C. On the spectrum of sum and product of non-Hermitian random matrices. \emph {Electron. Commun. Probab.}, 2011, {\bf 16}: 104-113.

\bibitem{Circular law} Bordenave C, Chafai D. Around the circular law. \emph{Probab. Surv.}, 2012, {\bf 9}: 1-89.  

\bibitem{BouLa} Bougerol P, Lacroix J. Products of random matrices with applications to Schrodinger operators. Progress in Probability and Statistics {\bf 8}. Birkhauser Boston, Inc, Boston, 1985. 

\bibitem{BurdaJW10} Burda Z, Janik R A, Waclaw B. Spectrum of the product of independent random Gaussian matrices. \emph{Phys. Rev. E,} 2010, {\bf 81}: 041132.

\bibitem{BJLN10} Burda Z, Jarosz A, Livan G, Nowak M A,  Swiech A.  Eigenvalues and singular values of products of rectangular Gaussian random matrices. \emph{ 
Phys. Rev. E}, 2010, {\bf 82}(6), 061114.

\bibitem{Burda13} Burda Z. Free products of large random matrices-a short review of recent developments. \emph{J. Phys.: Conf. Ser.}, 2013, {\bf 473}: 012002.
\bibitem{Byun2025}
Byun S S, Forrester P J. \emph{Progress on the Study of the Ginibre Ensembles}, 1st ed., Springer Singapore, 2025.

\bibitem{CJQ25} Chang S, Jiang T, Qi Y. Eigenvalues of product of Ginibre ensembles and their inverses and that of truncated Haar Unitary Matrices and their inverses. \emph{J. Math. Phys.}, 2025, {\bf 66}, 063301. 

\bibitem{ChangQ17} Chang S, Qi Y. Empirical distribution of scaled eigenvalues for product of matrices from the spherical ensemble. \emph{Stat. Probab. Lett.}, 2017, {\bf 128}: 8-13.

\bibitem{ChangLQ} Chang S, Li D, Qi Y. Limiting distributions of spectral radii for product of matrices from the spherical ensemble. \emph{J. Math. Anal. Appl.}, 2018, {\bf 461}: 1165-1176.

\bibitem{Chen2011} Chen L H Y, Goldstein L, Shao Q M. Normal approximation by Stein's method. Springer, 2011. 

\bibitem{Cipolloni22Directional} Cipolloni G, Erd\"{o}s L, Schr\"{o}der D, Xu Y. Directional extremal statistics for Ginibre eigenvalues. \emph{J. Math. Phys.}, 2022, \textbf{63}(10), 103303.

\bibitem{Cipolloni22rightmost}
		Cipolloni G, Erd\"{o}s L, Schr\"{o}der D, Xu Y. On the rightmost eigenvalue of non-Hermitian random matrices. \emph{Ann. Probab.}, 2022, \textbf{51}(6): 2192-2242.

\bibitem{CipoErXu} Cipolloni G, Erd$\ddot{o}$s L, Xu Y. Universality of extremal eigenvalues of large random matrices. arXiv:2312.08325. 


\bibitem{Crisanti12} Crisanti A, Paladin G, Vulpiani A. Products of random matrices in statistical physics. Springer, 2012.

\bibitem{DasGupta2008} DasGupta A. Asymptotic theory of statistics and probability. Springer, 2008.


\bibitem{Esseen1945} Esseen C G. Fourier analysis of distribution functions. A mathematical study of the Laplace-Gaussian law. \emph{Acta Math.}, 1945, {\bf 77}: 1-125.

\bibitem{Forrester15} Forrester P J. Asymptotics of finite system Lyapunov exponents for some random matrix ensembles.  \emph{J. Phys. A: Math. Theoret.}, 2015, {\bf 48}, 215205. 

\bibitem{ForIpsen16} Forrester P J, Ipsen, I R.  Real eigenvalue statistics for products of asymmetric real Gaussian matrices. \emph{Linear Algebra Appl.}, 2016, {\bf 510}: 259-290. 

\bibitem{FurKesten} Furstenberg H, Kesten H. Products of random matrices. \emph{Ann. Math. Stat.}, 1960, {\bf 31}: 457-469. 

\bibitem{Gohberg} Gohberg I, Goldberg S,  Krupnik N. \emph{Traces and determinants of linear operators. Operator Theory: Advances and Applications.} {\bf 116} Birkhauser, Basel, 2000.

\bibitem{Gorin22} Gorin V, Sun, Y. Gaussian fluctuations for products of random matrices. \emph{Amer. J. Math.}, 2022, {\bf  144} (2): 287-393.

\bibitem{GT} G$\ddot{o}$tze F,  Tikhomirov A. On the asymptotic spectrum of products of independent random matrices. arXiv:1012.2710, 2010.

\bibitem{Gotze2015} G$\ddot{o}$tze F, Kosters H, Tikhomirov A. Asymptotic spectra of matrix-valued functions of independent random matrices and free probability. \emph{Random Matrices: Th. Appl.}, 2015, {\bf 4}: 1550005.

\bibitem{Gradshteyn2007} Gradshteyn I S, Ryzhik I M, Jeffrey A. Table of integrals, series, and products. 7th ed. Academic Press, 2007.
\bibitem{HKPV2006} Hough J B, Krishnapur M, Peres Y, Vir\'ag B. Determinantal processes and independence. \emph{Probab. Surveys}, 2006, {\bf 3}: 206-229. 
\bibitem{HKPV2009} Hough J, Krishnapur M, Peres Y, Vir\'ag B. Zeros of Gaussian analytic functions and determinantal point Processes. American Mathematical Society, Providence, RI 2009.

\bibitem{HuMa25} Hu X, Ma Y. Convergence rate of extreme eigenvalue of Ginibre ensembles to Gumbel distribution. arXiv:2506.04560, 2025. 

\bibitem{Hwang} Hwang C R. A brief survey on the spectral radius and the spectral distribution of large random matrices with iid entries. \emph{Random Matrices  Appl.}, 1986, 145-152.

\bibitem{Ipsen} Ipsen J R. Products of independent Gaussian random matrices. Bielefeld University, 2015.
\bibitem{Ipsen2015} Ipsen J R. Lyapunov exponents for products of rectangular real, complex and quaternionic Ginibre matrices. \emph{J. Phys. A: Math. Theor.}, 2015, {\bf  48} (15), 155204.

\bibitem{IpsenKie14} Ipsen J R, Kieburg M. Weak commutation relations and eigenvalue statistics for products of rectangular random matrices. \emph{Phys. Rev. E}, 2014, {\bf 89}(3): 032106. 


\bibitem{JQ17} Jiang T, Qi Y. Spectral radii of large non-Hermitian random matrices. \emph{J. Theor. Probab.,} 2017, {\bf 30}(1): 326-364.

\bibitem{JQ19} Jiang T, Qi Y. Empirical distributions of eigenvalues of product ensembles. \emph{J. Theor. Probab.}, 2019, {\bf 32}: 353-394.

\bibitem{Kopel2020} Kopel P, O'Rourke S, Vu V. Random matrix products: Universality and least singular values. \emph{ Ann. Probab.}, 2020, {\bf  48}(3): 1372-1410.

\bibitem{Kostlan1992} Kostlan E. On the spectra of Gaussian matrices. \emph{ Linear Algebra Appl.}, 1992, {\bf 162}: 385-388.

\bibitem{LW15} Liu D, Wang Y. Universality for products of random matrices I: Ginibre and truncated unitary cases. \emph{Int. Math. Res. Not.}, 2015, {\bf 16}: 6988-7015.

\bibitem{LWW23} Liu D, Wang D, Wang, Y. Lyapunov exponent, universality and phase transition for products of random matrices. \emph{Comm. Math. Phys.}, 2023, {\bf  399}: 1811-1855. 

\bibitem{LW24} Liu D, Wang Y. Phase transitions for infinite products of large non-Hermitian random matrices. \emph{Ann. Inst. H. Poincar\'e Probab. Statist.}, 2024, {\bf 60}(4): 2813-2848.

\bibitem{MaQi2024}
Ma X, Qi Y. Limiting Spectral Radii for Products of Ginibre Matrices and Their Inverses. \emph{
J. Theoret. Probab.}, 2024, {\bf 37}: 
3756-3780.

\bibitem{MaMeng2025} Ma Y, Meng X. Exact convergence rate of spectral radius of complex Ginibre to Gumbel distribution. \emph{Random matrices: Th Appl.}, 2026, 2650003. 


\bibitem{MaTian} Ma Y, Tian B. Revisit on the convergence rate of normal extremes. arXiv:2507.09496, 2025.


\bibitem{MaWang25} Ma Y, Wang S. Optimal $W_1$ and Berry-Esseen bound between the spectral radius of large chiral non-Hermitian random matrices and Gumbel. arXiv:2501.08661, 2025.

\bibitem{Mehta} Mehta M L. Random matrices and the statistical theory of energy levels. Academic Press, 1967.

\bibitem{Petrov1975} Petrov V V. Sums of independent random variables. Springer, 1975.
\bibitem{QiXie} Qi Y, Xie M. Spectral radii of products of random rectangular matrices. \emph{J. Theoret. Probab.}, 2020, {\bf 33}: 2185-2212. 

\bibitem{Reiss1989} Reiss R D. Approximate distributions of order statistics. Springer, 1989.

\bibitem{RS11} O'Rourke S, Soshnikov A. Products of independent non-Hermitian random matrices. \emph{ Electron. J. Probab.}, 2011, {\bf 16}(81): 2219-2245.

\bibitem{RRSV15} O'Rourke S, Renfrew D, Soshnikov A, Vu V. (2015) Products of independent elliptic random matrices. \emph{ J. Stat. Phys.}, 2015, {\bf 160}(1): 89-119.

\bibitem{Taobook} Tao T. Topics in random matrix theory. American Mathematical Society, 2012.

\bibitem{Tao and Vu} Tao T, Vu V. From the Littlewood-Offord problem to the circular law: universality of the spectral distribution of random matrices. \emph{ Bull. Amer. Math. Soc.}, 2009, {\bf 46}(3): 377-396.

\bibitem{Tik} Tikhomirov A N. On the asymptotics of the spectrum of the product of two rectangular random matrices. \emph{ Sib. Math. J.}, 2011, {\bf 52}(4): 747-762.
\bibitem{Wang18} Wang Y. Order statistics of the moduli of the eigenvalues of product 
random matrices from polynomial ensembles. \emph{Electron. Commun. Probab.}, 2018, {\bf 23}(21): 1-14. 

\bibitem{Zeng16} Zeng X. Eigenvalues distribution for products of independent spherical ensembles. \emph{ J. Phys. A: Math. Theor.}, 2016, {\bf 49}: 235201.

\bibitem{Zeng17} Zeng X. Limiting empirical distribution for eigenvalues of products of random rectangular matrices. \emph{Stat. Probab. Lett.}, 2017, {\bf 126}: 33-40.



\end{thebibliography}
\end{document}